\documentclass[a4paper,french,twoside,11pt]{smfartVS}

\usepackage{smfenum,smfthm,amsmath,amssymb}
\usepackage{mathrsfs,euscript,color}
\usepackage[latin1]{inputenc}
\input{xypic}

\numberwithin{equation}{section} 
\theoremstyle{plain}

\def\CC{\mathbb{C}}
\def\FF{\mathbb{F}}
\def\NN{\mathbb{N}}
\def\QQ{\mathbb{Q}}

\def\ZZ{\mathbb{Z}} 

\def\A{{\rm A}}
\def\B{{\rm B}}

\def\D{{\rm D}}
\def\E{{\rm E}}
\def\F{{\rm F}}
\def\G{{\rm G}}
\def\H{{\rm H}}
\def\I{{\rm I}}
\def\J{{\rm J}}
\def\K{{\rm K}}
\def\L{{\rm L}}
\def\M{{\rm M}}
\def\N{{\rm N}}
\def\O{{\rm O}}
\def\P{{\rm P}}

\def\R{{\rm R}}
\def\SS{{\rm S}}
\def\T{{\rm T}}
\def\U{{\rm U}}
\def\V{{\rm V}}
\def\W{{\rm W}}
\def\X{{\rm X}}
\def\Y{{\rm Y}}
\def\Z{{\rm Z}}

\def\Cc{\mathscr{C}}
\def\Dd{\mathscr{D}}

\def\Gg{\mathscr{G}}
\def\Hh{\mathscr{H}}
\def\Ii{\mathscr{I}}

\def\Ll{\mathscr{L}}

\def\Oo{\mathscr{O}}

\def\Rr{\mathscr{R}}
\def\Ss{\mathscr{S}}

\def\Uu{\mathscr{U}}

\def\Zz{\mathscr{Z}}

\def\AA{\mathfrak{A}}

\def\a{\alpha} 
\def\b{\beta}

\def\e{o}
\def\ee{e}
\def\g{\gamma}

\def\k{\kappa}
\def\l{\lambda}

\def\p{\mathfrak{p}}
\def\q{\star}
\def\s{\sigma}
\def\t{\theta}

\def\Om{\Omega}

\def\ie{c'est-à-dire }

\def\rp{\rangle}
\def\>{\geqslant}
\def\<{\leqslant}

\def\Hom{{\rm Hom}}
\def\End{{\rm End}}
\def\Aut{{\rm Aut}}
\def\Mat{\mathscr{M}}
\def\GL{{\rm GL}}
\def\Gal{{\rm Gal}}

\def\ind{{\rm ind}}

\def\mult#1{{#1}^{\times}}

\def\ffr#1{\smash{\mathop{\longrightarrow}\limits^{#1}}}

\def\Cc{\EuScript{C}}
\def\Hh{\EuScript{H}}

\def\Oo{\EuScript{O}}

\def\Ss{\EuScript{S}}

\def\GB{{\bar{\G}}}
\def\MB{{\bar{\M}}}
\def\PB{{\bar{\P}}}

\def\NB{{\bar{\N}}}

\def\IA{\ip}
\def\GA{\GB}
\def\MA{\MB}
\def\PA{\PB}

\def\NA{\NB}

\def\ff{k}

\def\kk{k}
\def\ll{\mathfrak{l}}
\def\vv{\mathfrak{v}}
\def\mm{\mathfrak{m}}
\def\nn{\mathfrak{n}}
\def\ss{\mathfrak{s}}
\def\tt{\mathfrak{t}}
\def\ap{{\rm ap}}
\def\sc{{\rm sc}}
\def\ip{\boldsymbol{i}}

\def\te{j}
\def\rp{\boldsymbol{r}}
\def\st{{\rm st}}
\def\St{{\rm St}}

\def\CR{{\rm R}}
\def\qlb{\overline{\QQ}_{\ell}}
\def\zlb{\overline{\ZZ}_{\ell}}
\def\flb{\overline{\FF}_{\ell}}

\def\supp{{\rm supp}}
\def\cusp{{\rm cusp}}
\def\scusp{{\rm scusp}}

\def\seg{{\rm Seg}}
\def\MS{{\rm Mult}}
\def\Irr{{\rm Irr}}
\def\Rnd{{\rm Rnd}}

\def\Div{\ZZ}
\def\Dive{\NN}

\def\Iw{\I}
\def\r{{\bf r}}
\def\z{{z}}
\def\m{\mm}

\def\({\left(}
\def\){\right)}

\def\KM{\textbf{\textsf{K}}}
\def\KMS{\textbf{\textsf{S}}}

\def\vr{\varrho}

\def\qr{q(\rho)}
\def\om{\omega}
\def\Fr{\phi}

\def\sy#1{\boldsymbol{[}#1\boldsymbol{]}}
\def\widetild#1{#1^{\vee}}
\def\ov#1{{#1}^{-}}

\newcounter{nonum}

\newtheorem{defin}[nonum]{Définition}

\newcounter{introlet}

\newtheorem{theon}[introlet]{Théorème}

\newcounter{intronum}
\def\theintronum{\arabic{intronum}}

\newcounter{notanum}
\def\thenotanum{\arabic{notanum}}
\newenvironment{notasec}{\refstepcounter{notanum}
\noindent{\bf\thenotanum.}}{\medskip}

\author{Alberto M\'\i nguez}
\address{Institut de Mathématiques de Jussieu, Université Paris 6, 
4 place Jussieu, 75005, Paris, France. 
URL: {\rm http://www.math.jussieu.fr/$\sim$minguez/}} 
\email{minguez@math.jussieu.fr}

\author{Vincent Sécherre}
\address{Université de Versailles Saint-Quentin-en-Yvelines\\
Laboratoire de Mathémati\-ques de Versailles\\
45 avenue des Etats-Unis\\
78035 Versailles cedex, France}
\email{vincent.secherre@math.uvsq.fr}

\title{Représentations lisses modulo $\ell$ de $\GL_{m}(\D)$}

\begin{altabstract}
Let $\F$ be a non-Archimedean locally compact field of residue 
characteristic $p$, let $\D$ be a finite dimensional central 
division $\F$-algebra and $\CR$ be an algebraically closed field 
of char\-acteristic different from $p$. 
We classify all smooth ir\-re\-du\-ci\-ble 
representations of $\GL_{m}(\D)$ for $m\>1$, with coefficients in $\CR$, 
in terms of multisegments, generalizing works by Zelevinski, 
Tadi\'c and Vignéras. 
We prove that any irreducible $\CR$-representation of $\GL_{m}(\D)$ has a 
unique supercuspidal support, and thus get two classifications: 
one by supercuspidal multisegments, classifying representations 
with a given supercuspidal support, and one by aperiodic multisegments, 
classifying representations with a given cuspidal support. 
These constructions are made in a purely local way, with a substantial use 
of type theory. 
\end{altabstract}

\thanks{
Ce travail a bénéficié de financements de l'EPSRC 
(GR/T21714/01, EP/G001480/1) 
et de l'Agence Natio\-nale de la Recherche 
(ANR-08-BLAN-0259-01, ANR-10-BLANC-0114).
Le premier auteur est aussi financé en partie par 
FEDER, MTM2007-66929 et MTM2010-19298}

\begin{document}

\maketitle
\setcounter{tocdepth}{2}
\tableofcontents

\section*{Introduction}

Soit $\F$ un corps commutatif localement compact non archimédien de 
ca\-rac\-téristique ré\-si\-duelle $p$ et soit $\D$ une algèbre à division 
centrale de dimension finie sur $\F$ dont le degré ré\-duit est noté $d$.  
Pour tout entier $m\>1$, on pose $\G_{m}=\GL_{m}(\D)$, qui est une forme 
intérieure de $\GL_{md}(\F)$. 
Les représentations lisses irréductibles complexes de 
$\GL_{md}(\F)$ ont été 
classées par Zelevinski \cite{Ze2} en  termes de paramètres appelés
\textit{multisegments}.
Dans le cas où $\F$ est de caractéristique nulle, 
Tadi\'c \cite{Tadic} a donné une classification des représen\-ta\-tions lisses 
irréductibles complexes de $\G_m$ en termes de multi\-segments.
La méthode qu'il utilise repose sur les résultats 
de \cite{DKV} (eux-mêmes reposant sur la formule des traces) 
et en particulier sur la correspondance de Jacquet-Langlands locale et la 
classification des représentations tempérées en fonction de la série 
discrète ({\it ibid.}, théorème B.2.d). 
Dans \cite{Bad}, 
Badu\-lescu étend  ces deux résultats au cas où 
$\F$ est de caractéristique $p$, et on trouve dans \cite{BHLS} la 
classification des représentations lisses irréductibles 
complexes de $\G_m$ sans restriction sur la caractéristique de $\F$. 

Dans cet article, on s'intéresse au problème de la classification des 
représentations lisses irré\-ductibles de $\G_m$ à coefficients dans un corps 
$\CR$ algébriquement clos de carac\-té\-ristique différente de $p$.
Dans le cas où cette carac\-téristique est un nombre premier $\ell$, 
ces représentations seront dites modulaires. 
L'intérêt de comprendre et classer les représenta\-tions modulaires 
pro\-vient de l'étude des congruen\-ces de formes automorphes. 

La théo\-rie des représentations mo\-du\-lai\-res des groupes 
réductifs $p$-adiques a été dé\-ve\-lop\-pée par Vignéras dans 
\cite{Vig1,Vig2}.
En particulier, les représentations modulaires 
du groupe $\GL_n(\F)$ y sont étudiées en détail. 
Comparée à la théorie complexe,
la théo\-rie mo\-du\-lai\-re présente de grandes simi\-larités, 
mais aussi des différences importantes, à la fois dans les résultats 
et les méthodes. 
Les représentations modulaires d'un groupe compact 
ne sont pas toujours semi-simples. 
Le fait que $\ell$ soit différent de $p$ équivaut à 
l'exis\-tence d'une mesure de Haar à valeurs 
dans $\CR$ sur le groupe $p$-adique, 
mais la mesure d'un sous-groupe ouvert compact peut être nulle.  
Il faut dis\-tin\-guer dans le cas modulaire 
entre les deux notions de représentation irréductible cuspidale 
(\ie dont tous les modules de Jacquet relativement à un sous-groupe 
parabolique propre sont nuls) et super\-cus\-pi\-da\-le (\ie qui 
n'est sous-quotient d'aucune induite pa\-ra\-bo\-li\-que d'une 
re\-pré\-sen\-ta\-tion irréductible d'un sous-groupe de Levi propre). 

Si l'on essaie d'étendre aux représentations modulaires du groupe non déployé 
$\G_m$ les techni\-ques employées par 
Zelevinski \cite{Ze2}, Tadi\'c \cite{Tadic} et Vignéras \cite{Vig1,Vig2}, 
on est confronté aux pro\-blè\-mes suivants.
Il n'y a pas de version modulaire de la formule des traces et du théorème de
Paley-Wie\-ner (qui servent à prouver que l'induite
normalisée d'une représentation de carré inté\-gra\-ble d'un sous-groupe 
de Levi est irréductible et à établir la 
correspondance de Jacquet-Lang\-lands).
Il n'y a pas non plus de version modulaire du 
théorème du quo\-tient de Langlands, 
qui permet dans \cite{Tadic} de décrire les re\-pré\-senta\-tions 
irréductibles en fonction des tempérées. 
Les foncteurs de Jacquet ne suffisent pas à déterminer les 
repré\-sen\-tations, \ie qu'il y a des repré\-sen\-tations irréductibles non 
isomorphes d'un même groupe $\G_m$ dont tous les modules de Jacquet propres 
sont isomorphes.
La différence entre représentations cuspidales et super\-cus\-pi\-dales 
joue également un rôle important. 
On a naturellement une notion de support super\-cus\-pi\-dal 
(voir le \S\ref{defsuper}) mais on ignore en général si une repré\-sen\-tation 
irréductible modulaire d'un groupe réductif $p$-adique possède un unique 
support supercuspidal 
(contrairement à ce qui se passe pour le support cuspidal~: voir le théorème 
\ref{UniSuppCusp}). 

L'un des principaux résultats de cet article est la preuve de l'unicité du 
support su\-percuspidal pour les représentations irréductibles de $\G_m$
(voir plus bas dans cette introduction et le théorème \ref{unicitesupp}). 

\begin{theon}
\label{UniSCintro}
Toute représentation irréductible de $\G_m$ a un unique support 
su\-per\-cuspidal.  
\end{theon}

À ceci s'ajoutent les problèmes suivants, spécifiques au cas où 
$\D$ est non commutative~:
les représentations cuspidales de $\G_m$ n'ont pas de 
modèle de Whittaker et il n'y a pas de théorie des dérivées pour les 
représentations irréductibles de $\G_m$, dont l'usage est crucial dans 
\cite{Vig2}~;~la réduction modulo $\ell$ d'une $\qlb$-représentation 
irréductible cuspidale entière de $\G_m$ n'est pas tou\-jours irréductible 
et il y a des $\flb$-représentations irréductibles cuspidales 
non supercuspidales qui ne se relèvent pas en des $\qlb$-représentations
(voir \cite{MS11}). 

L'un des principaux outils employés dans cet article est une version modulaire 
de la théorie des types de Bush\-nell et Kutzko pour $\G_m$, qui a été 
développée dans un précédent article \cite{MS11}.
Elle permet de comparer la théorie des représentations 
lisses de $\G_m$ à celle de certaines algèbres de Hecke affines. 
L'un des principaux résultats de \cite{MS11} est la définition, 
pour toute représentation irréductible cuspidale $\rho$ de $\G_{m}$, 
d'un caractère non ramifié $\nu_{\rho}$ de ce groupe 
possédant la propriété suivante~: si $\rho'$ est une 
représentation irréductible cuspidale de $\G_{m'}$, avec $m'\>1$, 
alors l'induite parabolique normalisée $\rho\times\rho'$ (voir \eqref{VentreDieu})
est réductible si et seulement si 
$m'=m$ et si $\rho'$ est iso\-morphe à $\rho\nu_{\rho}^{}$ ou à $\rho\nu_{\rho}^{-1}$. 
Ceci est le point de départ de la classi\-fi\-ca\-tion des représentations
irréductibles puisque, 
une fois défini ce caractère, on peut définir la notion de 
segment (voir la définition \ref{DefSeg11}).

\begin{defin}
Un {\it segment} est une suite finie de la forme~:
\begin{equation*}
[a,b]_\rho=(\rho\nu_{\rho}^a,\rho\nu_{\rho}^{a+1},\dots,\rho\nu_{\rho}^{b}),
\end{equation*} 
où $a,b\in\ZZ$ sont des entiers tels que $a\<b$ et où $\rho$ est une 
re\-pré\-sen\-ta\-tion irréductible cuspidale de $\G_m$.
\end{defin}

Grâce à la propriété de quasi-pro\-jectivité des types 
construits dans \cite{MS11}, il existe une bijection explicite entre 
les représentations irréductibles dont le support cuspidal est 
inertiel\-lement équivalent à $\rho\otimes\dots\otimes\rho$ 
(où $\rho$ apparaît $n$ fois) et les 
modules simples sur une certaine algèbre de Hecke affine $\Hh(n,q(\rho))$
de type $\A_{n-1}$ et de paramètre $q(\rho)$, 
une puissance de $p$ associée à $\rho$. 
Par ce biais, on associe à un segment 
$\Delta=[a,b]_{\rho}$ de longueur $n=b-a+1$ deux repré\-sen\-tations 
irréductibles~: 
\begin{equation*}
\Z(\Delta) \text{ et } \L(\Delta)
\end{equation*}
de $\G_{mn}$, respectivement sous-représentation et quotient de l'induite 
$\rho\nu_{\rho}^{a}\times\rho\nu_{\rho}^{a+1}\times\dots\times\rho\nu_{\rho}^{b}$
et correspondant respectivement au caractère trivial et au caractère signe 
de $\Hh(n,q(\rho))$. 
Tant que $\qr$ n'est pas congru à $1$ modulo $\ell$, il est possible de 
définir les représentations $\Z(\Delta)$ et $\L(\Delta)$ sans passer par la 
théorie des types et les algèbres de Hecke affines 
(voir la proposition \ref{DefAltRec}), mais cette approche est nécessaire si l'on 
veut inclure le cas où $\qr$ est congru à $1$ modulo $\ell$. 
Dans le cas où $\CR$ est le corps des nombres complexes, $\Z(\Delta)$ est une 
représentation de Speh généralisée et $\L(\Delta)$ une représentation de 
Steinberg généralisée,
c'est à dire une représentation essentiellement de carré intégrable.
Ces représentations jouis\-sent d'un certain nombre de propriétés qui sont 
établies dans la section \ref{segments}. 
Elles permettent de prouver le résultat suivant, 
qui prouve le bien-fondé de notre définition des segments liés
(voir la définition \ref{Rotis}, ainsi que le théo\-rè\-me \ref{nuevo2}). 

\begin{theon}
Soit $r\>1$ et soient $\Delta_1,\dots,\Delta_r$ des segments. 
Les conditions suivantes sont équi\-valentes~:
\begin{enumerate}
\item 
Pour tous $i,j\in\{1,\dots,r\}$ tels que $i\neq j$, 
les segments $\Delta_i$ et $\Delta_j$ sont non liés.
\item 
L'induite $\Z(\Delta_1)\times \dots \times \Z( \Delta_r)$ est irréductible.
\item 
L'induite $\L(\Delta_1)\times \dots \times \L( \Delta_r)$ est irréductible.
\end{enumerate}
\end{theon}

Ce théorème est une version purement algébrique, valable pour des représentations 
modulaires, du ré\-sul\-tat selon lequel l'induite parabolique 
normalisée d'une représentation complexe de carré intégrable est irréduc\-tible.  

Un des principaux résultats de \cite{MS11} est la construction de foncteurs $\KM$
permettant de faire un lien entre représenta\-tions de $\G_m$ et 
re\-pré\-sentations des grou\-pes linéaires $\GL$ sur une extension finie 
$\kk$ du corps résiduel de $\F$.  
Ils constituent un outil technique important et peuvent être vus comme la 
généralisation du foncteur associant à une représentation lisse 
de $\G_m$ la représentation de $\GL_{m}(\kk_\D)$ 
(où $\kk_\D$ est le corps résiduel de $\D$)
sur l'espace de ses invariants sous le radical pro-unipotent du 
sous-groupe compact maximal $\GL_{m}(\Oo_\D)$, où $\Oo_\D$ est l'anneau des 
entiers de $\D$.

Étant donnés un entier $f\>1$ et une extension finie 
$\kk$ du corps résiduel de $\F$,
les représentations irréductibles modulaires de $\GL_f(\kk)$
ont été étudiées et classées par Dipper et James \cite{James,DJ1},
et une théorie des dérivées a été développée par Vignéras dans \cite{Vig1}. 
On notera en particulier les résultats suivants (voir la section \ref{ApA})~:
\begin{enumerate}
\item 
toute représentation irré\-duc\-tible de $\GL_f(\kk)$
possède un unique support supercus\-pidal~;
\item
on a une notion de représentation irréductible non dégénérée
de $\GL_f(\kk)$~; 
\item
on a une classification des représentations irréductibles 
cuspidales de $\GL_f(\kk)$ en fonction des représentations 
irréductibles supercuspidales de $\GL_{f'}(\kk)$ pour $f'$ divisant $f$.
\end{enumerate}
Grâce aux foncteurs évoqués plus haut, ces trois assertions 
vont permet\-tre de prouver plusieurs résultats importants 
qui aboutissent à l'unicité du support supercuspidal (théorème 
\ref{UniSCintro} ci-dessus).  
D'abord (1) permet de prouver l'uni\-cité du support supercuspidal à 
inertie près (voir la proposition \ref{KOLA}). 
Ensuite (2) permet d'associer 
à tout entier $n\>1$ et à toute représentation irréductible 
cuspidale $\rho$ de $\G_m$ une repré\-sen\-tation irréductible~: 
\begin{equation*}
\St(\rho,n)
\end{equation*}
définie comme l'unique sous-quotient irréductible de l'induite 
$\rho\times\rho\nu_{\rho}^{}\times\dots\times\rho\nu_{\rho}^{n-1}$ 
dont l'image par un foncteur $\KM$ convenable, défini à partir 
de $\rho$, 
contienne une certaine représentation irréductible non dégé\-né\-rée. 
Dans le cas où $\CR$ est le corps des nombres complexes, 
$\St(\rho,n)$ est la représentation de Steinberg 
généralisée $\L([0,n-1]_\rho)$, mais dans le cas modulaire ces deux 
représentations dif\-fè\-rent dès que $n$ est assez grand 
(voir la remarque \ref{L=St}). 
Enfin (3) permet de prouver le résultat suivant,
qui fournit une classification des re\-pré\-sen\-tations 
irréductibles cuspidales en fonction des représentations 
supercuspida\-les, dans le cas où le corps $\CR$ est de 
caractéristique non nulle $\ell$
(voir le théorème \ref{AppCuspSuper}).

\begin{theon}
\begin{enumerate}
\item 
Étant donnée une représentation irréductible cuspidale $\rho$, 
il existe un entier $e(\rho)$ tel que 
$\St(\rho,n)$ soit cuspidale si et seulement si 
$n=1$ ou $n=e(\rho)\ell^r$ avec $r\>0$. 
\item
Pour toute représentation irréductible cuspidale non supercuspidale
$\pi$, il y a une repré\-sen\-tation 
irréductible supercuspidale $\rho$ et un unique $r\>0$ 
tels que $\pi$ soit iso\-morphe à la repré\-sen\-tation 
$\St_r(\rho)=\St(\rho,e(\rho)\ell^r)$.
\item
Si $\rho'$ est une représentation irréductible supercuspidale telle que 
les représentations 
$\St_r(\rho')$ et $\St_r(\rho)$ soient isomorphes, 
alors il existe $i\in\ZZ$ tel que $\rho'$ soit isomorphe à 
$\rho\nu_{\rho}^i$. 
\end{enumerate}
\end{theon}

La dernière étape de la preuve de l'unicité du support supercuspidal 
est la proposition cruciale \ref{ST}, 
qui contrôle l'apparition de facteurs cuspidaux dans des induites 
para\-boli\-ques.  
On est alors en mesure de prouver le théorème \ref{unicitesupp}. 

On arrive maintenant au problème de la classification des représentations 
irréductibles de $\G_m$. 
On a une notion naturelle d'équi\-valence entre segments
(voir la définition \ref{DefEquSeg}), ce qui permet d'introduire la
définition suivante. 

\begin{defin}
Un {\it multisegment} est une application $\m$ à support fini de l'ensemble 
des classes d'équivalence de segments à valeurs dans $\NN$, qu'on 
représente sous la forme d'une somme finie~:
\begin{equation*}
\m=\Delta_1+\dots+\Delta_r=[a_1,b_1]_{\rho_1}+\dots+[a_r,b_r]_{\rho_r},
\end{equation*}
où $\Delta_1,\dots,\Delta_r$ sont des segments. 
\end{defin}

Si l'on note $m_i$ l'entier tel que $\rho_i$ soit une représentation 
de $\G_{m_i}$, la somme des $(b_i-a_i+1)m_i$ est appelée le 
\textit{degré} de $\m$ et la somme formelle des classe d'équivalence 
des $\rho_{i}\nu_{\rho_i}^{j}$ pour $i\in\{1,\dots,r\}$ et 
$j\in\{a_i,\dots,b_i\}$ est appelée le \textit{support} de $\m$.

Un multisegment est dit \textit{supercuspidal} si toutes les 
représentations $\rho_1,\dots,\rho_r$ sont supercuspidales et 
\textit{apériodique} si, pour tout entier $n\>0$ et 
toute représentation irréductible cuspidale $\rho$, il existe 
$a\in\ZZ$ tel que la classe d'équivalence du segment 
$[a,a+n]_{\rho}$ n'apparaisse pas dans $\m$.
Ces deux sortes de multisegments vont permettre une classification des 
représentations irréductibles en fonction de leurs supports supercuspidal 
et cuspidal res\-pec\-tivement. 
Ces deux sortes de multisegments se correspondent bijectivement~:
grâce à la classification des représentations cuspidales en fonction des 
supercuspidales, on définit une application~:
\begin{equation*}
\m\mapsto\m_\sc
\end{equation*}
associant à un multisegment un multisegment supercuspidal. 
On vérifie (voir le lemme \ref{unimap}) 
qu'il existe un unique multisegment apériodique $\mathfrak{a}$ 
tel que $\mathfrak{a}_\sc=\m_\sc$~; on le note $\m_\ap$. 

La définition de $\St(\rho,n)$ et la preuve de la proposition \ref{ST} 
reposent toutes deux sur une idée commune donnant lieu à la notion de 
représentation irréductible \textit{résiduellement non dégénérée}
(voir le \S\ref{ModWhitResRND}). 
Dans le cas où $\D=\F$, cette notion coïncide avec celle de 
représentation irréductible non dégénérée de $\GL_n(\F)$ définie par 
Vignéras (voir le corollaire \ref{coroSt}). 
Grossièrement, il s'agit, par l'intermédiaire des foncteurs $\KM$ 
introduits ci-dessus,
de transporter aux représentations de $\G_m$ la notion de 
représentation non dégénérée qui existe pour les représentations de 
$\GL$ sur un corps fini de caractéristique $p$.
Cette idée culmine dans la définition suivante~: 
à tout multi\-seg\-ment $\m$ on fait correspondre un sous-groupe de Levi 
standard $\M_{\m}$ de $\G$ (dont les tailles des blocs sont donnés 
par les degrés des segments apparaissant dans $\m$)
puis une représentation irré\-duc\-tible~:
\begin{equation*}
\Sigma(\m)
\end{equation*}
de $\M_\m$ (notée $\St_{\overline{\mu}_\m}(\m)$ au paragraphe
\ref{ModWhitRes}) 
définie comme l'unique sous-quotient irré\-duc\-tible résiduellement non 
dégénéré du module de Jacquet de 
$\Z(\Delta_1)\times\dots\times\Z(\Delta_r)$
relatif au sous-groupe parabolique standard $\P_\m$ 
de facteur de Levi $\M_\m$. 
Cette induite possède un unique sous-quotient irréductible~:
\begin{equation*}
\Z(\m)
\end{equation*}
dont le module de Jacquet relatif à $\P_\m$ possède un sous-quotient 
isomorphe à $\Sigma(\m)$. 
L'intérêt d'introduire la représentation $\Sigma(\m)$ est qu'il n'est pas 
difficile de montrer qu'elle ne dépend que de $\m_\sc$, 
et que l'application $\m\mapsto\Sigma(\m)$ est injective
sur l'ensemble des multisegments super\-cus\-pi\-daux 
(voir le \S\ref{OccamG}).
Il s'ensuit que la restriction $\Z_\sc$ de $\m\mapsto\Z(\m)$ à l'ensemble 
des multisegments super\-cuspidaux est injective 
(voir le théorème \ref{inje}). 

La preuve de la surjectivité de l'application 
$\Z_\sc$ et le calcul des supports cuspidal et
super\-cuspidal de $\Z(\m)$ pour un multisegment $\m$ quelconque 
sont plus difficiles.
Nous traitons tous ces problème 
en même temps dans un raisonnement par récurrence sur le degré de $\m$
(voir les propositions \ref{bijj} et \ref{cuspZ}).
C'est là que nous utilisons de façon cruciale un argument de comptage
(voir le lemme \ref{FinaleAriki})
reposant sur la classification des modules irréductibles sur une algèbre de 
Hecke affine en une racine de l'unité (Ariki \cite{Ariki,ArikiBook}, 
Chriss-Ginzburg \cite{CG}). 
De façon précise, cet argument
(voir le \S\ref{BlaACGM}) permet de conclure que 
l'application injective \eqref{INJineg} est bijective.
Nous obtenons finalement le théorème de classi\-fi\-cation 
(voir le théorème \ref{bijj2thm}). 

\begin{theon}
\begin{enumerate}
\item 
L'application $\m\mapsto\Z(\m)$ induit une surjection de l'ensemble 
des multisegments de degré $m$ sur l'ensemble des représentations 
irréductibles de $\G_m$.
\item
Étant donnés deux multisegments $\m,\m'$, 
les représentations $\Z(\m),\Z(\m')$ sont isomorphes si et seulement si 
$\m_\sc^{}=\m'_\sc$ (ou, de façon équivalente, 
si $\m_\ap^{}=\m'_\ap$).
\item
Pour tout multisegment $\m$, le support cuspidal de $\Z(\m)$ est égal 
au support de $\m_\ap$ et son support supercuspidal est égal au support 
de $\m_\sc$. 
\end{enumerate}
\end{theon}

Enfin, dans le paragraphe \ref{rrre}, nous étudions le problème de 
la réduction mod $\ell$ des $\qlb$-repré\-sen\-tations irréductibles 
entières de $\G_m$. 
Nous prouvons (voir le théorème \ref{reductionsegment2}) 
que le morphisme de réduction mod $\ell$ est surjectif, \ie que 
toute $\flb$-représentation irréductible de $\G_m$ est la réduction 
mod $\ell$ d'une $\qlb$-représentation virtuelle entière de longueur 
finie. 

Nous insistons sur le fait que ce travail fournit une classification des 
représentations irré\-duc\-tibles de $\G_m$ ne reposant pas
sur les résultats antérieurs de Zelevinski, Tadi\'c et Vignéras. 
Par contre, il s'appuie sur la classifications des représentations 
irréductibles modulaires de $\GL_n$ sur un corps fini de caractéristique 
$p$ (Dipper-James) et sur la classification des modules irréductibles sur une 
algèbre de Hecke affine en une racine de l'unité (Ariki, Chriss-Ginzburg).  

Dans la section \ref{JeNeSersARien}, on établit des résultats généraux concernant 
l'induction parabolique dans un groupe réductif $p$-adique quelconque~: 
on prouve que la semi-simplifiée de l'induite para\-bo\-lique d'une représentation 
irréductible d'un sous-groupe de Levi ne dépend pas du sous-groupe parabolique 
choisi (proposition \ref{commu}), et on prouve au paragraphe \ref{MoiNonPlus} 
des critères d'irréductibilité et d'unicité d'un quotient irréductible d'une induite 
parabolique. 

Les sections \ref{ApA}, \ref{ApB} et \ref{InvRho} ne contiennent pas de résultat 
nouveau.  
Dans la section \ref{ApA}, on résume la théorie des représentations 
modulaires de $\GL_n$ sur un corps fini de caractéristique $p$. 
Dans la section \ref{ApB}, on résume la classification des modules simples 
sur une $\CR$-algèbre de Hecke affine de type A en termes de multisegments 
apériodiques. 
Dans la section \ref{InvRho}, on résume les résultats et les outils de théorie des 
types obtenus dans \cite{MS11} dont nous aurons besoin. 

Dans la section \ref{ClassiPot}, on prouve l'unicité du support supercuspidal à inertie 
près et on classe les représentations cuspidales de $\G_m$ en fonction des 
représentations supercuspidales de $\G_{m'}$ avec $m'$ divisant $m$. 
Dans la section \ref{segments}, on définit la notion de segment et on associe à tout 
seg\-ment deux repré\-sen\-tations irréductibles dont on étudie les propriétés. 
Dans la section \ref{Sec9}, on définit la notion de représentation résiduellement 
non dégénérée de $\G_m$. 
Dans la section \ref{Sec10} enfin, on classe les représentations irréductibles de 
$\G_m$ en termes de multisegments apériodiques et supercuspidaux. 

\section*{Remerciements}

Nous remercions Jean-François Dat, Guy Henniart, 
Vanessa Miemietz, Shaun Stevens et Ma\-rie-France Vignéras
pour de nombreuses discussions à propos de ce travail.

Une partie de ce travail a été réalisée lors du séjour des auteurs à 
l'Erwin Schrödinger Institute en janvier-février 2009 et du second 
auteur à l'Institut Henri Poincaré de janvier à mars 2010~;
que ces deux institutions soient remerciées pour leur accueil 
et leur soutien financier. 
Une autre partie en a été réalisée lors de plusieurs 
séjours à l'Univer\-sity of East Anglia~: nous remercions 
celle-ci pour son accueil et Shaun Stevens 
pour ses nombreuses invitations. 

Alberto M\'\i nguez remercie le CNRS
pour les six mois de délégation dont il a bénéficié en 2011. 

Vincent Sécherre remercie 
l'Institut de Ma\-thé\-matiques de Luminy
et 
l'Uni\-ver\-sité de la Médi\-terranée Aix-Marseille~2,
où il était en poste durant la majeure partie de ce travail. 


\section*{Notations et conventions}

\begin{notasec}
\label{NotCorps}
On note 
$\NN$ l'ensemble des entiers naturels,
$\ZZ$ l'anneau des entiers relatifs et 
$\CC$ le corps des nombres complexes.
\end{notasec}

\begin{notasec}
\label{effec}
Si $\X$ est un ensemble, on note $\Div(\X)$ le groupe 
abélien libre de base $\X$ constitué des applications de 
$\X$ dans $\ZZ$ à 
support fini et $\Dive(\X)$ le sous-monoïde constitué des applications à 
valeurs dans $\NN$.
Étant donnés $f,g\in\Div(\X)$, on note $f\<g$ si $g-f\in\Dive(\X)$, 
ce qui définit une relation d'ordre partiel sur $\Div(\X)$.
\end{notasec}

\begin{notasec}
\textit{Dans cet article}, $\F$ est un corps commutatif localement compact non 
archi\-médien de carac\-téristique résiduelle notée $p$
et $\CR$ est un corps algébriquement clos de ca\-ractéristique dif\-fé\-rente de $p$. 
\end{notasec}

\begin{notasec}
Toutes les $\F$-algèbres sont supposées unitaires et de dimension finie.
Par $\F$-\emph{al\-gèbre à di\-vi\-sion} on entend $\F$-algèbre 
centrale dont l'anneau sous-jacent est un corps, pas néces\-sai\-rement
com\-mu\-tatif. 
Si $\K$ est une extension finie de $\F$, 
ou une algèbre à division sur une extension finie de $\F$, 
on note $\Oo_\K$ son anneau d'entiers, $\p_\K$ son idéal maximal, 
$\kk_\K$ son corps résiduel et $q_\K$ le cardinal de $\kk_\K$.
En particulier, on pose $q=q_\F$ une fois pour toutes. 
\end{notasec}

\begin{notasec}
Une $\CR$-{\it re\-pré\-sen\-ta\-tion lisse} d'un groupe topologique $\G$ 
est la donnée d'un $\R$-espace vectoriel $\V$ et d'un homomorphisme 
de $\G$ dans $\Aut_{\CR}(\V)$ tel que le stabilisateur dans $\G$ 
de tout vecteur de $\V$ soit ouvert. 
{\it Dans cet article, toutes les représentations sont des 
  $\CR$-re\-pré\-sen\-ta\-tions lisses.} 

Une représentation de $\G$ sur un $\CR$-espace vectoriel $\V$ est 
{\it ad\-mis\-si\-ble} si, pour tout sous-groupe ouvert $\H$ de $\G$, 
l'espace $\V^{\H}$ de ses vec\-teurs $\H$-invariants est de dimension 
finie. 

Un $\CR$-{\it caractère} de $\G$ est un homo\-mor\-phis\-me de $\G$ dans 
$\mult\CR$ de noyau ouvert.  

Si $\pi$ est une $\CR$-représentation de $\G$, on désigne par $\pi^{\vee}$ 
sa contragrédiente. 
Si en outre $\chi$ est un $\CR$-caractère de $\G$, on note $\chi\pi$ ou 
$\pi\chi$ la représentation tordue $g\mapsto\chi(g)\pi(g)$.

Si aucune confusion n'est à craindre, on écrira \textit{caractère} et 
\textit{re\-présentation} plutôt que $\CR$-carac\-tè\-re et 
$\CR$-repré\-sen\-ta\-tion. 
\end{notasec}

\begin{notasec}
\label{+++}
On désigne par $\Rr_{\CR}(\G)$ la catégorie des $\CR$-re\-pré\-sen\-ta\-tions 
lisses de $\G$, 
par $\Irr_{\CR}(\G)$ l'ensemble des classes d'isomorphisme de ses
re\-pré\-sen\-ta\-tions irréductibles et par $\Gg_{\CR}(\G)$ le groupe 
de Gro\-then\-dieck de ses re\-pré\-sen\-ta\-tions de longueur finie. 
On omet\-tra souvent $\CR$ dans les notations. 

Si $\s$ est une représentation de longueur finie de $\G$, on désigne par 
$\sy{\s}$ son image dans $\Gg(\G)$.
En particulier, si $\s$ est irréductible, $\sy{\s}$ désigne 
sa classe d'isomorphisme. 
Lorsqu'aucune confusion ne sera possible, il nous arrivera d'identifier 
une représentation avec sa classe d'isomorphisme.

Le $\ZZ$-module $\Gg(\G)$ s'identifie à $\Div(\Irr(\G))$, 
qui est muni d'une relation d'ordre notée $\<$,
tandis que $\Dive(\Irr(\G))$ est l'ensemble des 
$\sy{\s}$ où $\s$ décrit les re\-pré\-sen\-ta\-tions 
de longueur finie de $\G$.
\end{notasec}

\section{Préliminaires}
\label{SectionPreliminaire}

\subsection{Induction et restriction paraboliques}
\label{Gred}

On suppose dans tout ce paragraphe que $\G$ est le groupe des points 
sur $\F$ d'un groupe réduc\-tif connexe défini sur $\F$. 
On renvoie à \cite[II.2]{Vig1} pour plus de précisions.
Toute repré\-sen\-tation ir\-ré\-duc\-ti\-ble de $\G$ est 
admis\-si\-ble et admet un caractère central d'après \cite[II.2.8]{Vig1}. 

\subsubsection{}
\label{MenelasEtHermione}

\textit{On choisit une fois pour toutes une racine carrée de $q$ dans $\CR$}, 
notée $\sqrt{q}$.
Si $\P=\M\N$ est un sous-groupe parabolique de $\G$ muni d'une décomposition 
de Levi, son module $\delta_\P$ est un caractère de $\M$ dans $\mult\CC$ 
de la forme $m\mapsto q^{v_\P(m)}$ où $v_\P$ est un homomorphisme 
de groupes de $\M$ dans $\ZZ$.
On définit un $\CR$-ca\-rac\-tère~:
\begin{equation*}
m\mapsto(\sqrt{q})^{v_\P(m)}
\end{equation*}
qui est un caractère non ramifié de $\M$ 
dont le carré est le module de $\P$ à valeurs dans $\CR^\times$. 
On note $\rp_{\P}^{\G}$ le foncteur de restriction 
para\-bo\-lique nor\-ma\-li\-sé (relativement à ce caractère) 
de $\Rr(\G)$ dans $\Rr(\M)$ et $\ip_{\P}^{\G}$
son adjoint à droite, \ie le foncteur d'induction 
para\-bo\-li\-que nor\-ma\-li\-sé qui lui correspond.
Ces foncteurs sont exacts, et ils pré\-ser\-vent l'admissibilité et 
le fait d'être de longueur finie (voir \cite[II]{Vig1}, paragraphes 2.1, 3.8 
et 5.13).

\begin{rema}
Lorsque $p$ est impair, remplacer la racine carrée choisie plus haut par 
son opposée a pour effet de tordre les foncteurs 
$\ip_{\P}^{\G}$ et $\rp_{\P}^{\G}$ par un caractère non 
ra\-mi\-fié d'ordre $2$.
\end{rema} 

Soit $\P^-$ le sous-groupe parabolique de $\G$ opposé à $\P$ relativement à 
$\M$. 

\begin{prop}
Si $\pi$ et $\s$ sont des représentations admissibles de $\G$ et de $\M$ 
respectivement, 
on a un iso\-mor\-phisme de $\CR$-espaces vec\-to\-riels~:
\begin{equation*}
\Hom_{\G}(\ip^{\G}_{\P^{-}}(\s),\pi)\simeq
\Hom_{\M}(\s,\rp^{\G}_{{\P}}(\pi))
\end{equation*}
dit de seconde adjonc\-tion (voir \cite[II.3.8]{Vig1}).  
\end{prop}


\subsubsection{}
\label{Tellmarch}

On fixe un tore déployé maximal $\A$ de $\G$ et un sous-groupe 
parabolique mi\-ni\-mal $\P$ de $\G$ contenant 
$\A$. 
On note respectivement $\W=\W(\G,\A)$ et $\Phi=\Phi(\G,\A)$ le 
groupe de Weyl et l'en\-sem\-ble des racines réduites de $\G$ 
re\-la\-ti\-ve\-ment à $\A$.
Le choix de $\P$ dé\-ter\-mi\-ne une base $\SS$ de $\Phi$ ainsi 
qu'un ensemble $\Phi^+$ de racine positives dans $\Phi$.
Pour toute partie $\I\subseteq\SS$, on note $\P_\I$ le sous-groupe 
parabolique de $\G$ contenant $\P$ déterminé par $\I$ et $\M_\I$ le 
sous-groupe de Levi contenant $\A$ lui correspondant.
Pour $\I,\J\subseteq\SS$, on pose~:
\begin{equation*}
\Dd(\I,\J)=
\Dd(\M_\I,\M_\J)=
\{w\in\W\ |\ w^{-1}(\I)\subseteq\Phi^+\ 
\text{et}\ w(\J)\subseteq\Phi^+\}.
\end{equation*}
Étant donnée une représentation $\s$ de $\M_\J$, on a le lemme 
géo\-mé\-tri\-que de Bernstein et Zele\-vin\-ski~:
\begin{equation}
\label{lemmegeometrique}
\sy{\rp^{\G}_{\P_{\I}}(\ip_{\P_\J}^{\G}(\s))}=
\sum\limits_{w\in\Dd(\I,\J)}
\sy{\ip^{\M_\I}_{\M_\I^{}\cap\P_\J^{w^{-1}}}
(w\cdot\rp^{\M_\J}_{\M_\J^{}\cap\P_\I^{w}}(\s))},
\end{equation}
qui est une égalité dans le groupe de Grothendieck $\Gg(\M_\I)$,
et où $w\;\cdot$ désigne la con\-ju\-gai\-son par $w$.
On renvoie à \cite[2.8]{Dat2} 
pour plus de pré\-ci\-sions.

\subsubsection{}
\label{DefRepCusp}

Une représentation irréductible de $\G$ est 
{\it cus\-pi\-da\-le} 
si son image par $\rp^{\G}_{\P}$ est nulle pour tout sous-groupe parabolique 
propre $\P$ de $\G$, \ie si elle n'est isomorphe à aucun quotient (ou, de 
fa\c{c}on équivalente, à aucune sous-représentation) d'une induite 
parabolique d'une re\-pré\-sen\-ta\-tion d'un sous-groupe de Levi propre.  
Elle est {\it super\-cus\-pi\-da\-le} si elle n'est isomorphe à aucun 
sous-quo\-tient d'une induite parabolique d'une 
re\-pré\-sen\-ta\-tion irréductible d'un sous-groupe de Levi propre. 

On note $\Cc_{\CR}(\G)$ et $\Ss_{\CR}(\G)$ les ensembles 
formés respectivement des clas\-ses d'iso\-mor\-phis\-me 
de re\-pré\-sen\-ta\-tions ir\-ré\-duc\-ti\-bles cuspidales et 
supercuspidales de $\G$.

\subsection{Représentations entières}
\label{ROk}
\label{Antisthene}

Soit $\ell$ un nom\-bre premier différent de $p$. 
On note
$\QQ_{\ell}$ le corps des nombres $\ell$-adiques, 
$\ZZ_{\ell}$ son anneau d'entiers et 
$\FF_{\ell}$ le corps résiduel de $\ZZ_{\ell}$. 
On fixe une clôture algébrique 
$\overline{\QQ}_{\ell}$ de $\QQ_{\ell}$.
On note 
$\overline{\ZZ}_{\ell}$ son anneau d'entiers et 
$\overline{\FF}_{\ell}$ le corps résiduel de $\overline{\ZZ}_{\ell}$, 
qui est une clôture algébrique de $\FF_{\ell}$. 

On suppose, à l'exception du paragraphe \ref{DefRepEnt} où $\G$ 
est un groupe localement profini quelconque, que $\G$ 
est le groupe 
des points sur $\F$ d'un groupe réductif con\-ne\-xe défini sur $\F$.

\subsubsection{}
\label{DefRepEnt}

Soit $\G$ un groupe topologique localement profini. 
Une représentation de $\G$ sur un $\qlb$-espace vectoriel $\V$ est dite 
{\it entière} si elle est ad\-mis\-si\-ble et si elle admet une 
{\it structure entière}, \ie un sous-$\zlb$-module de $\V$ stable par $\G$ et
engendré par une base de $\V$ sur $\qlb$ 
(\cite{Vig1,Vig6}). 

\subsubsection{}
\label{Semiotique}

On suppose que $\G$ est le groupe 
des points sur $\F$ d'un groupe réduc\-tif con\-ne\-xe défini sur $\F$.
Soit $\pi$ une représentation irréductible entière de $\G$ sur un 
$\qlb$-es\-pa\-ce vectoriel $\V$. 
D'après \cite[Theorem 1]{Vig7} et \cite[II.5.11]{Vig1}, on a les propriétés 
suivantes~:
\begin{enumerate}
\item 
toutes les struc\-tures en\-tières de $\pi$ sont de type fini comme 
$\zlb\G$-modules~; 
\item
si $\vv$ est une struc\-ture en\-tière de $\pi$, la
re\-pré\-sen\-ta\-tion de $\G$ sur $\vv\otimes\flb$ est de lon\-gueur finie~; 
\item
la semi-sim\-pli\-fiée de $\vv\otimes\flb$, qu'on note $\r_{\ell}(\pi)$
et qu'on appelle la {\it réduction modulo $\ell$} de $\pi$, ne dé\-pend pas du choix 
de $\vv$ mais seulement de la classe d'isomorphisme de $\pi$. 
\end{enumerate}
Par li\-néa\-rité, on en déduit un mor\-phis\-me de groupes~: 
\begin{equation}
\label{HomRes}
\r_{\ell}:\Gg_{\qlb}(\G)^{{\rm en}}\to\Gg_{\flb}^{}(\G),
\end{equation}
où $\Gg_{\qlb}(\G)^{{\rm en}} $ est le sous-groupe de $\Gg_{\qlb}(\G)$
engendré par les clas\-ses d'iso\-mor\-phis\-me de 
$\qlb$-re\-pré\-sen\-ta\-tions irréductibles entières de $\G$. 

\begin{rema}
Si $\H$ est un groupe profini, toute $\qlb$-re\-pré\-sen\-ta\-tion de 
dimension finie de $\H$ est en\-tiè\-re (\cite{Serre}, 
théorème 32), et on a un morphisme de réduction $\r_{\ell}$
analogue à (\ref{HomRes}).
\end{rema}

\subsubsection{}
\label{inds}

On fixe des racines carrées de $q$ dans $\qlb$ et $\flb$ de sorte que 
la seconde soit la réduction modulo $\ell$ de la première. 
Soit $\P=\M\U$ un sous-groupe parabolique de $\G$.
Si $\vv$ est une structure entière d'une $\qlb$-représentation entière
$\tilde\s$ de $\M$, 
le sous-espace $\ip^{\G}_{\P}(\vv)$ 
des fonc\-tions à valeurs dans $\vv$ est 
une structure entière de $\ip^{\G}_{\P}(\tilde\s)$ et il y a un isomorphisme
canonique de $\flb$-représentations de 
$\ip^{\G}_{\P}(\vv)\otimes\flb$ vers $\ip^{\G}_{\P}(\vv\otimes\flb)$. 
Pour plus de précisions, voir \cite[II.4.14]{Vig1}.

\subsubsection{}
\label{jacqs}

Le cas de la restriction parabolique est plus délicat. 
\textit{On suppose dans ce paragraphe uni\-quement que $\G$ est 
un groupe symplectique, orthogonal, unitaire ou une forme 
intérieure du groupe $\GL_n(\F)$, $n\>1$.}
Soit $\vv$ une structure entière d'une $\qlb$-représentation entière de 
longueur finie $\tilde\s$ de $\G$. 
D'après \cite[Proposition 6.7]{Dat2}, l'ima\-ge 
de $\vv$ par la projection 
de $\tilde\s$ sur $\rp^{\G}_{\P}(\tilde\s)$ est 
une structure entière de $\rp^{\G}_{\P}(\tilde\s)$ et on a 
$\r_{\ell}(\sy{\rp^\G_\P(\tilde\s)})=\sy{\rp^\G_\P(\r_{\ell}(\tilde\s))}$.

\subsection{Formes intérieures de $\GL_{n}(\F)$}
\label{Banalite}

\subsubsection{}

On fixe une $\F$-algèbre à division $\D$ 
de degré réduit noté $d$.
Pour tout $m\>1$, on dé\-si\-gne par $\Mat_{m}(\D)$ la 
$\F$-algèbre des matrices de taille $m\times m$ à coefficients dans 
$\D$ et par $\G_{m}=\GL_{m}(\D)$ le groupe de ses éléments 
in\-ver\-si\-bles. 
Il est commode de convenir que $\G_{0}$ est le groupe trivial. 

\subsubsection{}
\label{DefNu}

Soit $\N_{m}$ la norme réduite de $\Mat_{m}(\D)$ sur $\F$. 
On note $|\ |_{\F}$ 
la valeur absolue nor\-malisée de $\F$, \ie celle 
don\-nant à une uniformisante de $\F$ la valeur $q^{-1}$.
Puisque l'image de $q$ dans $\CR$ est in\-ver\-si\-ble, 
elle définit un $\CR$-caractère de $\mult\F$ noté $|\ |_{\F,\CR}$.
L'ap\-pli\-ca\-tion $g\mapsto|\N_{m}(g)|_{\F,\CR}$ est un 
$\CR$-ca\-rac\-tè\-re de $\G_{m}$, qu'on notera simplement 
$\nu$. 

\subsubsection{}
\label{UZero}

On note $\Irr$ la réunion disjointe des 
$\Irr(\G_{m})$, $m\>0$ et $\Gg$ la somme 
directe des $\Gg(\G_m)$, $m\>0$.
On identifie $\Gg$ à $\Div(\Irr)$. 
Si $\pi$ est une représentation de longueur finie de $\G_m$, 
on pose $\deg(\pi)=m$, qu'on ap\-pel\-le le {\it degré} de $\pi$,
et l'application $\deg$ fait de $\Gg$ un $\ZZ$-module gradué.

\subsubsection{}
\label{GeEm}

Si $\a=(m_{1},\ldots,m_{r})$ est une famille 
d'entiers $\>0$ de somme $m$, il lui correspond le 
sous-groupe de Levi standard $\M_{\a}$ de $\G_{m}$ constitué des matrices 
diagonales par blocs de tailles $m_{1},\ldots,m_{r}$ respectivement, que 
l'on identifie naturellement à 
$\G_{m_{1}}\times\cdots\times\G_{m_{r}}$. 
On note $\P_{\a}$ 
le sous-groupe para\-bo\-li\-que de $\G_{m}$ de facteur de Levi 
$\M_{\a}$ formé des matrices tri\-an\-gu\-lai\-res su\-pé\-rieures 
par blocs de tailles $m_{1},\ldots,m_{r}$ respectivement, et on note 
$\N_{\a}$ son radical unipotent.
Les foncteurs $\ip_{\P_{\a}}^{\G_{m}}$ et $\rp_{\P_{\a}}^{\G_{m}}$ sont simplement 
notés respectivement $\ip_{\a}$ et $\rp_{\a}$.
Si pour chaque entier $i\in\{1,\ldots,r\}$ on a une 
représentation $\pi_{i}$ de $\G_{m_i}$, on note~: 
\begin{equation}
\label{VentreDieu}
\pi_1\times\cdots\times\pi_r=\ip_{\a}(\pi_1\otimes\cdots\otimes\pi_r).
\end{equation}
Si $\pi_1,\dots,\pi_r$ sont de longueur finie, 
la quantité $\sy{\pi_1\times\cdots\times\pi_r}$
ne dépend que de $\sy{\pi_{1}},\dots,\sy{\pi_{r}}$.
L'ap\-pli\-cation~:
\begin{equation*}
(\sy{\pi_1},\dots,\sy{\pi_r})\mapsto\sy{\pi_1\times\cdots\times\pi_r}
\end{equation*}
induit par linéarité une application multilinéaire 
de $\Gg(\G_{m_1})\times\dots\times\Gg(\G_{m_r})$ dans 
$\Gg(\G_{m})$. 
Ceci fait de $\Gg$ une 
$\ZZ$-al\-gè\-bre asso\-cia\-tive gra\-duée, dont on verra à la proposition 
\ref{commuD} qu'elle est commutative. 

On note aussi ${\rp}_{\a}^{-}$ le foncteur de restriction 
para\-bo\-li\-que relativement au sous-groupe para\-bo\-li\-que opposé 
à $\P_{\a}$ relativement à $\M_\a$, 
\ie formé des matrices tri\-an\-gu\-lai\-res in\-fé\-rieures
par blocs de tailles $m_{1},\ldots,m_{r}$ respectivement. 

\section{Compléments sur l'induction parabolique}
\label{JeNeSersARien}

Dans toute cette section, on suppose que $\G$ est le groupe 
des points sur $\F$ d'un groupe ré\-duc\-tif connexe défini sur $\F$. 
Au paragraphe \ref{FCIP}, on suppose que $\G$ 
admet des sous-groupes discrets cocompacts.
Au paragraphe \ref{KafkaChateau}, on suppose que $\G=\GL_m(\D)$ avec 
$m\>1$. 

\subsection{Support cuspidal et équivalence inertielle}
\label{SCEI}

\subsubsection{}
\label{SuppCusp}

Une \textit{paire cuspidale} de $\G$ est un couple $(\M,\vr)$ formé 
d'un sous-groupe de Levi $\M$ de $\G$ et d'une représentation
irréductible cuspidale $\vr$ de $\M$.
Si $\pi$ est une re\-pré\-sen\-ta\-tion ir\-ré\-ductible de $\G$,
on note~: 
\begin{equation*}
\cusp(\pi)
\end{equation*}
son sup\-port cuspidal, \ie la clas\-se de
$\G$-con\-ju\-gai\-son d'une paire cuspidale $(\M,\vr)$ de $\G$ 
telle que $\pi$ soit isomorphe à une sous-re\-pré\-sen\-ta\-tion 
de $\ip^{\G}_{\P}(\vr)$ pour au moins un sous-groupe para\-bolique $\P$ de 
facteur de Levi $\M$. 
Pour la commodité du lecteur, on fournit ici une preuve de l'unicité 
du support cuspidal (comparer avec \cite[II.2.20]{Vig1}).

\begin{theo}
\label{UniSuppCusp}
Soit $\pi$ une re\-pré\-sen\-ta\-tion irréductible de $\G$. 
Il existe une paire cus\-pidale $(\M,\vr)$ de $\G$, unique 
à $\G$-con\-ju\-gai\-son près, et un sous-groupe parabolique $\P$ de 
facteur de Levi $\M$, tels que 
$\pi$ soit une sous-représen\-ta\-tion de 
$\ip^{\G}_{\P}(\vr)$.
\end{theo}

\begin{proof}
Soient $(\M,\vr)$ et $(\M',\vr')$ deux paires cuspidales de $\G$ et 
$\P$ et $\P'$ deux sous-groupes paraboliques de $\G$ de facteurs de 
Levi respectifs $\M$ et $\M'$ tels que $\pi$ soit iso\-mor\-phe à une 
sous-re\-pré\-sen\-ta\-tion de $\ip^{\G}_{\P}(\vr)$ et de 
$\ip^{\G}_{\P'}(\vr')$. 
Par adjonction, $\vr$ est un quo\-tient de $\rp_\P^\G(\pi)$ et, 
par exactitude
du foncteur de Jacquet, $\vr$ est un facteur de com\-po\-si\-tion 
irréductible de $\rp_\P^\G(\ip^{\G}_{\P'}(\vr'))$. 
On déduit du lemme géométrique 
\eqref{lemmegeometrique} que $\M'$ est conjugué à un sous-groupe de $\M$. 
De m\^eme, on montre que $\M$ est conjugué à un sous-groupe de $\M'$ et
donc que $\M$ et $\M'$ appartiennent à la m\^eme classe de $\G$-conjugaison, 
et on peut les supposer tous les deux standards (relativement au choix 
d'un tore déployé maximal de $\G$). 
Dans ce cas, d'après le lemme géométrique à nouveau, 
$\rp_\P^\G(\ip^{\G}_{\P'}(\vr'))$ est 
composée des $w\cdot\vr'$ où $w$ parcourt 
$\Dd(\M,\M')$. 
La représentation $\vr$ étant un sous-quotient irréductible de
$\rp_\P^\G(\ip^{\G}_{\P'}(\vr'))$, elle est donc isomorphe à un 
$w\cdot\vr'$, ce qui montre que $(\M,\vr)$ et $(\M',\vr')$ sont 
deux paires cuspidales conjuguées. 
\end{proof}

On a ainsi une ap\-pli\-ca\-tion~:
\begin{equation*}
\cusp:\Irr(\G)\to
\{\text{classes de $\G$-conjugaison de paires cuspidales de $\G$}\}
\end{equation*}
surjective et à fibres finies. 

En outre, pour toute représentation irréductible $\pi$ de $\G$, il existe 
une paire cuspidale $(\M',\vr')$ dans $\cusp(\pi)$ et un sous-groupe 
para\-bolique $\P'$ de facteur de Levi $\M'$ tels que $\pi$ soit un quotient 
de $\ip^\G_{\P'}(\vr')$.

\subsubsection{}
\label{defsuper} 

Une \textit{paire supercuspidale} de $\G$ est un couple $(\M,\vr)$ 
constitué d'un sous-groupe de Levi $\M$ de $\G$ et d'une 
représentation irréductible supercuspidale $\vr$ de $\M$.
Si $\pi$ est une re\-pré\-sen\-ta\-tion irréductible de $\G$,
il existe 
une paire super\-cuspidale $(\M,\vr)$ de $\G$ telle que $\pi$ soit un 
sous-quotient de $\ip^{\G}_{\P}(\vr)$ pour un sous-groupe para\-bo\-li\-que 
$\P$ de facteur de Levi $\M$,
et on conjecture qu'une telle paire est unique à $\G$-con\-ju\-gai\-son 
près (voir \cite[II.2.6]{Vig1}). 
Pour le cas du groupe $\GL_n(\F)$, lire \cite[V.4]{Vig2}.

Dans la section \ref{Sec9}, nous prouvons cette conjecture 
lorsque $\G$ est une forme 
intérieure de $\GL_n(\F)$ (voir le théorème \ref{unicitesupp}). 

\subsubsection{}
\label{Persephone}

Soit $(\M,\vr)$ une paire cuspidale de $\G$. 
Une paire cuspidale $(\M',\vr')$ de $\G$  
est dite \textit{iner\-tiellement équivalente} à 
$(\M,\vr)$ s'il existe un caractère 
non ramifié $\chi$ de $\M$ tel que  
$(\M',\vr')$ soit conjuguée à $(\M,\vr\chi)$ 
sous $\G$.
On note $[\M,\vr]_\G$
la classe d'inertie (\ie la classe 
d'équivalence inertielle)  de $(\M,\vr)$.
Si $\Om$ est la classe d'inertie d'une
paire cuspidale de $\G$, on note~:
\begin{equation}
\label{Mellamphy}
\Irr(\Om)
\end{equation}
l'ensemble des classes d'isomorphisme de 
re\-pré\-sen\-ta\-tions irréductibles
de $\G$ qui sont des sous-quotients 
d'une induite parabolique d'un élément de $\Om$.
On note aussi~:
\begin{equation*}
\Irr(\Om)^{\q}=\cusp^{-1}(\Om)
\end{equation*}
l'ensemble des classes d'isomorphisme de 
re\-pré\-sen\-ta\-tions irréductibles de $\G$ 
dont le support cuspidal appartient à $\Om$, 
\ie qui sont des quotients 
d'une induite parabolique d'un élément de $\Om$.

\subsubsection{}

Un $\qlb$-caractère de $\G$ est entier si et 
seulement s'il est à valeurs dans $\zlb$. 
Une $\qlb$-re\-présen\-ta\-tion irréductible cus\-pi\-da\-le de $\G$ est 
entière si et seu\-le\-ment si son caractè\-re central l'est.
Si $\G$ est comme au paragraphe \ref{jacqs},
une $\qlb$-représen\-ta\-tion ir\-ré\-ductible de $\G$ est entière 
si et seulement si son support cuspidal l'est
(c'est une consé\-quence des para\-graphes \ref{inds} et \ref{jacqs}).

\subsection{Sous-quotients irréductibles d'une induite parabolique}
\label{FCIP}

L'objet de ce paragraphe est de prouver le résultat suivant, 
dont on trouvera une preuve un peu différente chez Dat \cite{Dat3} 
(voir {\it ibid.}, lemme $4.13$).
Les deux preuves s'appuient sur la propriété d'irréductibilité 
générique obtenue dans \cite{Dat2} (voir plus bas) pour un groupe $\G$ 
admettant des sous-groupes discrets cocompacts.

\begin{prop}
\label{commu}
On suppose que $\G$ possède des sous-groupes discrets cocompacts.
Soit $\s$ une représentation irréductible d'un 
sous-groupe de Levi $\M$ de $\G$ et soient 
$\P$ et $\P'$ 
deux sous-groupes paraboliques de $\G$ de facteur de Levi $\M$.
Alors $\sy{\ip_{\P}^{\G}(\s)}=\sy{\ip_{\P'}^{\G}(\s)}$.
\end{prop}

\begin{proof}
Il suffit de prouver le résultat lorsque $\P$ et $\P'$ 
sont deux sous-groupes para\-boliques maximaux 
opposés par rapport à $\M$. 
Soit $\Psi(\M)$ le groupe des caractères non rami\-fiés 
de $\M$.
On va prouver que, pour tout 
$\chi\in\Psi(\M)$, on a 
$\sy{\ip_{\P}^{\G}(\chi\s)}=\sy{\ip_{\P'}^{\G}(\chi\s)}$. 

D'après \cite[Theorem 5.1]{Dat2}, il existe un ouvert de Zariski non vide 
$\Uu\subseteq\Psi(\M)$ tel que
$\ip_{\P}^{\G}(\chi\s)$ et $\ip_{\P'}^{\G}(\chi\s)$ soient irréductibles
pour tout $\chi\in\Uu$.
La re\-pré\-sen\-ta\-tion $\chi\s$ est un sous-quotient de 
$\rp^{\G}_{\P'}(\ip_{\P}^{\G}(\chi\s))$ et,
d'après le lemme géo\-mé\-tri\-que (\ref{lemmegeometrique}),
les autres sous-quotients irréductibles apparaissent dans les~:
\begin{equation*}
\label{stg}
\sy{\ip^{\M}_{\M\cap\P^{w^{-1}}}(w\cdot\rp^{\M}_{\M\cap\P'{}^{w}}(\chi\s))},
\quad
w\in\Dd(\P',\P),
\quad
w\neq1.
\end{equation*} 
Les caractères centraux de ces autres sous-quotients irréductibles sont 
donc de la forme~: 
\begin{equation*}
w\cdot(\chi|_{\Z_\M}\omega_{\pi})
\end{equation*}
où $\pi$ est un sous-quotient irréductible de 
$\rp^{\M}_{\M\cap\P'{}^{w}}(\s))$ de caractère central $\omega_\pi$
et où $\Z_\M$ est le centre de $\M$.
Ainsi, pour $\chi$ dans un ouvert de Zariski non vide de 
$\Psi(\M)$ qu'on peut sup\-po\-ser égal à 
$\Uu$, les sous-quotient irréductibles de 
$\rp^{\G}_{\P'}(\ip_{\P}^{\G}(\chi\s))$
distincts de $\chi\s$ 
ont un caractère central différent de celui de $\chi\s$. 
Pour un tel $\chi$, la représentation $\chi\s$ est un facteur direct 
de $\rp^{\G}_{\P'}(\ip_{\P}^{\G}(\chi\s))$ 
(sa restriction au centre $\Z_\G$ est bien un facteur direct,
qui est stable sous l'action de $\G$).
Par réciprocité de Frobenius, on a un homomorphisme non trivial de 
$\ip_{\P}^{\G}(\chi\s)$ dans $\ip_{\P'}^{\G}(\chi\s)$.
Ces deux dernières représentations étant ir\-ré\-duc\-ti\-bles 
pour $\chi\in\Uu$, on a~: 
\begin{equation}
\label{EtapeUneG}
\ip_{\P}^{\G}(\chi\s)\simeq\ip_{\P'}^{\G}(\chi\s)
\end{equation}
pour tout $\chi\in\Uu$.

On fixe une famille décroissante $(\K_i)_{i\>1}$ de pro-$p$-sous-groupes 
ouverts compacts de $\G$ for\-mant une base de voisinages de l'élément neutre.  
Pour $i\>1$, on pose $\Hh_{i}=\Hh(\G,\K_i)$, l'algèbre des fonctions de $\G$ 
dans $\CR$ qui sont à support compact et bi-invariantes par $\K_i$. 
Pour tout caractère non ra\-mifié $\chi\in\Psi(\M)$, on note 
$\V_{i}(\chi)$ et $\W_{i}(\chi)$ 
les sous-espaces des vecteurs $\K_{i}$-invariants de
$\ip_{\P}^{\G}(\chi\s)$ et $\ip_{\P'}^{\G}(\chi\s)$ respectivement. 
Ce sont des $\Hh_{i}$-mo\-du\-les de dimension finie sur $\CR$. 
On va montrer que, pour tout $\chi\in\Psi(\M)$, 
les $\Hh_{i}$-modules $\V_{i}(\chi)$ et $\W_{i}(\chi)$ ont les mêmes
facteurs de composition. 

D'après (\ref{EtapeUneG}), pour tous $\chi\in\Uu$ et $i\>1$,
les $\Hh_{i}$-modules 
$\V_{i}(\chi)$ et $\W_{i}(\chi)$ sont isomorphes.
Pour tout $\chi\in\Uu$, tout $i\>1$ et tout $f\in\Hh_{i}$, 
on a donc une égalité des polynômes caractéristiques de $f$ 
sur $\V_{i}(\chi)$ et $\W_{i}(\chi)$~: 
\begin{equation}
\label{poliG} 
{\rm Pcar}_{\V_{i}(\chi)}(f)={\rm Pcar}_{\W_{i}(\chi)}(f).
\end{equation}
Les fonctions qui, à chaque caractère $\chi\in\Psi(\M)$, 
associent les 
coefficients des poly\-nô\-mes ca\-rac\-té\-ris\-ti\-ques d'un élément 
$f\in\Hh_{i}$ dans $\V_{i}(\chi)$ et $\W_{i}(\chi)$ respectivement 
sont des fonctions régulières de la variété $\Psi(\M)$, car ces 
coefficients dépendent polynômialement des coefficients de la matrice 
de $f$. 
L'identité \eqref{poliG}, vraie pour $\chi\in\Uu$, est donc vraie pour 
tout $\chi\in\Psi(\M)$, tout $i\>1$ et tout $f\in\Hh_{i}$. 

\begin{lemm}
Soient $\V$ et $\W$ deux $\Hh_{i}$-modules de dimension finie tels que, 
pour tout $f\in\Hh_{i}$, on ait une égalité des polynômes 
caractéristiques de $f$ sur $\V$ et $\W$~: 
\begin{equation}
\label{EgaPolCarG} 
{\rm Pcar}_{\V}(f)={\rm Pcar}_{\W}(f). 
\end{equation}
Alors $\V$ et $\W$ ont les mêmes facteurs de composition en tant que 
$\Hh_{i}$-modules.
\end{lemm}

\begin{proof}
Pour chaque $\Hh_{i}$-module simple $\mm$, on note $v(\mm)$
et $w(\mm)$ les mul\-ti\-pli\-ci\-tés de $\mm$ dans $\V$ et $\W$ 
respectivement, 
et on suppose qu'il existe un module simple $\mm$ tel que ces multiplicités 
diffèrent. 
D'après \cite[\S2.2]{Bou} (voir le corollaire 2 au théo\-rè\-me 1), 
il existe un élément $f\in\Hh_{i}$ tel que $f|_{\mm}={\rm id}_{\mm}$, 
et tel que $f|_{\mm'}=0$ pour tout module simple $\mm'$ non isomorphe 
à $\mm$.
D'après (\ref{EgaPolCarG}), le scalaire $1$ a des multiplicités égales 
dans les deux polynômes caractéristiques, \ie que $v(\mm)=w(\mm)$. 
\end{proof}

Ainsi, pour $i\>1$ et $\chi\in\Psi(\M)$, 
les $\Hh_{i}$-modules $\V_{i}(\chi)$ et $\W_{i}(\chi)$ ont 
les mêmes facteurs de composition. 
Étant donné $\chi\in\Psi(\M)$, on fixe un entier 
$i\>1$ tel que tous les sous-quotients des re\-pré\-sen\-ta\-tions 
$\ip_{\P}^{\G}(\chi\s)$ et $\ip_{\P'}^{\G}(\chi\s)$ 
possèdent des vecteurs $\K_{i}$-invariants non nuls.
Le foncteur~:
\begin{equation*}
\pi\mapsto\pi^{\K_{i}}
\end{equation*}
des $\K_i$-invariants est exact et induit une 
bijection entre l'ensemble des classes de 
re\-pré\-sen\-ta\-tions irréductibles de $\G$ possédant des 
vecteurs $\K_{i}$-in\-va\-riants non nuls et celui des 
classes de $\Hh_{i}$-modules simples (voir \cite[I.6.3]{Vig1}). 
Il définit donc un isomorphisme de groupes entre le sous-groupe 
$\Gg_{i}$ de $\Gg(\G)$ engendré par les classes des 
représentations irréductibles de $\G$ possédant des vecteurs 
$\K_{i}$-invariants non nuls et le groupe de Grothendieck de la catégorie 
des modules de lon\-gueur finie de $\Hh_{i}$. 

D'après ce qui précède, on a $\sy{\V_{i}(\chi)}=\sy{\W_{i}(\chi)}$, 
puis $\sy{\ip_{\P}^{\G}(\chi\s)}=\sy{\ip_{\P'}^{\G}(\chi\s)}$ dans $\Gg_{i}$. 
On en déduit que $\ip_{\P}^{\G}(\chi\s)$ et 
$\ip_{\P'}^{\G}(\chi\s)$ ont les mêmes facteurs de composition, 
donc la même image dans $\Gg(\G)$. 
\end{proof}

\subsection{Deux lemmes sur l'irréductibilité d'une induite parabolique}
\label{MoiNonPlus}

On suppose encore que $\G$ est le groupe 
des points sur $\F$ d'un groupe ré\-duc\-tif connexe 
quelconque défini sur $\F$. 

\begin{lemm}
\label{mult1}
Soit $\M$ un sous-groupe de Levi de $\G$, soit $\s$ une 
représentation ir\-ré\-ductible de $\M$ et soit $\P$ 
un sous-groupe parabolique de $\G$ de facteur de Levi $\M$.
On note $\ov{\P}$ le sous-groupe parabolique de $\G$ opposé
à $\P$ relativement à $\M$.
\begin{enumerate}
\item 
Supposons que $\s$ apparaît avec multiplicité $1$ dans
$\sy{\rp^{\G}_{\P}(\ip^{\G}_{\P}(\s))}$. 
Alors $\ip^{\G}_{\P}(\s)$ a une seule sous-représentation 
irréductible~;
sa multiplicité dans $\sy{\ip^{\G}_{\P}(\s)}$ est égale à $1$. 
\item 
Supposons que $\s$ apparaît avec multiplicité $1$ dans
$\sy{\rp^{\G}_{\ov{\P}}(\ip^{\G}_{\P}(\s))}$. 
Alors $\ip^{\G}_{\P}(\s)$ a un seul quo\-tient irréductible~; sa 
multiplicité dans $\sy{\ip^{\G}_{\P}(\s)}$ est égale à $1$. 
\end{enumerate}
\end{lemm}

\begin{proof}
Supposons qu'il existe deux sous-représentations irréductibles 
$\pi_1$ et $\pi_2$ de $\ip^{\G}_{\P}(\s)$, et notons $\pi$ leur 
somme directe.
Par réciprocité de Frobenius, on trouve que~:
\begin{eqnarray*}
\dim_{\CR}(\Hom_\M (\rp^{\G}_{\P}(\pi),\s))=
\dim_{\CR}(\Hom_\G(\pi,\ip^{\G}_{\P}(\s)))\>2.
\end{eqnarray*}
Par exactitude du foncteur de Jacquet $\rp_\P^\G$, on a aussi 
$\sy{\rp^{\G}_{\P}(\pi)} \< 
\sy{\rp^{\G}_{\P}( \ip^{\G}_{\P}(\s))}$.
On trouve donc que $\s$ apparaît
avec multiplicité au moins $2$ dans la quantité 
$\sy{\rp^{\G}_{\P}(\ip^{\G}_{\P}(\s))}$, ce qui est absurde.
La deuxième assertion 
se prouve de façon analogue en utilisant la seconde adjonction. 
\end{proof}

\begin{lemm}
\label{sir}
On suppose que $\G$ possède des sous-groupes discrets cocompacts.
Soit $\pi$ une représentation de longueur finie de $\G$. 
On suppose qu'il y a un sous-groupe 
para\-bo\-li\-que $\P$ de $\G$ de facteur de Levi $\M$ et une 
représentation irréductible $\s$ de $\M$ tels que~: 
\begin{enumerate}
\item 
$\pi$ est une sous-représentation de $\ip_{\P}^{\G}(\s)$ et un quotient 
de $\ip_{\ov{\P}}^{\G}(\s)$~;
\item 
la multiplicité de $\s$ dans $\sy{\rp_{\P}^{\G}(\ip_{\P}^{\G}(\s))}$ 
est égale à $1$. 
\end{enumerate}
Alors $\pi$ est irréductible.
\end{lemm}

\begin{proof}
D'après la condition (2) et le lemme \ref{mult1}, 
$\ip_{\P}^{\G}(\s)$ admet une 
unique sous-représentation irréductible $\pi_1$. 
D'après la proposition \ref{commu}, la multiplicité de $\s$ dans 
$\sy{\rp_{\P}^{\G}(\ip_{\ov{\P}}^{\G}(\s))}$ 
est $1$.
L'induite $\ip_{\ov{\P}}^{\G}(\s)$ 
admet donc un unique quotient ir\-ré\-duc\-ti\-ble, noté $\pi_2$.
Aussi $\pi_1$ est-elle l'unique sous-représentation 
irréductible de $\pi$ et $\pi_2$ son unique quotient irréductible. 
Or, d'après les deux propriétés d'adjonction du para\-gra\-phe 
\ref{MenelasEtHermione}, le facteur $\s$ apparaît à la fois
dans $\rp_{\P}^{\G}(\pi_1)$ et dans $\rp_{\P}^{\G}(\pi_2)$.
Puisque la multiplicité de $\s$ dans 
$\sy{\rp_{\P}^{\G}(\ip_{\P}^{\G}(\s))}$ est 
égale à $1$, sa multiplicité dans $\sy{\rp_{\P}^{\G}(\pi)}$ 
doit être égale à $1$. 
On a donc $\pi_1=\pi_2$, de sorte que $\pi$ est irréductible. 
\end{proof} 

On utilisera ce lemme dans le cas particulier où $\pi$ est isomorphe à 
$\ip_{\P}^{\G}(\s)$ et à $\ip_{\ov{\P}}^{\G}(\s)$ pour prouver
l'irréductibilité d'une induite parabolique. 

\subsection{Le cas de $\GL_m(\D)$}
\label{KafkaChateau}

On se place maintenant dans le cas où $\G=\GL_{m}(\D)$, 
qui a des sous-groupes discrets cocompacts pour tout $m\>1$ 
(voir Borel-Harder \cite{BoHa}).

\subsubsection{}
\label{Sinibaldi}

On désigne par $\Cc$ la réunion disjointe des 
$\Cc(\G_{m})$ pour $m\>1$ 
(voir les paragraphes \ref{DefRepCusp} et \ref{SuppCusp}).
Étant donné une classe de conjugaison $\mathfrak{c}$ d'une paire cus\-pi\-dale 
de $\G_m$, $m\>1$, 
il existe une famille $\a=(m_{1},\ldots,m_{r})$ d'entiers 
$\>1$ de somme $m$, et, 
pour chaque $i\in\{1,\dots,r\}$, il existe une
re\-pré\-sen\-ta\-tion cuspi\-dale $\rho_i\in\Cc(\G_{m_i})$, 
de telle sorte que $\mathfrak{c}$ soit la classe de conjugaison de la paire~:
\begin{equation*}
(\M_{\a},\rho_{1}\otimes\dots\otimes\rho_{r}).
\end{equation*}
On identifiera $\mathfrak{c}$ à la 
somme $\sy{\rho_1}+\dots+\sy{\rho_r}$ dans $\Dive(\Cc)$. 
On identifiera ainsi l'ensemble des supports cuspidaux 
(resp. supercuspidaux) des $\G_m$, $m\>1$, à $\Dive(\Cc)$ 
(resp. à $\Dive(\Ss)$).

\subsubsection{}

D'après la proposition \ref{commu}, on a le résultat suivant.

\begin{prop}
\label{commuD}
La $\ZZ$-algèbre $\Gg$ est commutative.
\end{prop}

Ceci étend les résultats de Zelevinski \cite{Ze2} pour les 
représentations complexes de $\GL_n(\F)$, 
de Tadi\'c \cite{Tadic} pour les représentations complexes de
$\GL_m(\D)$ 
(voir aussi la propriété (P1) dans \cite[\S2.2]{BHLS}) et de Vignéras
\cite{Vig1} pour les $\CR$-représentations de $\GL_n(\F)$ quand la
carac\-téristi\-que de $\CR$ est différente de $2$ (voir la remarque 
ci-dessous). 

\begin{rema}
La preuve de Vignéras utilise le fait que, pour toute représenta\-tion 
irré\-duc\-ti\-ble $\pi$ de $\GL_n(\F)$, la représentation $g \mapsto\pi({}^tg^{-1})$
est isomorphe à la contra\-gré\-diente de $\pi$, où ${}^tg$ désigne la
transposée de $g\in\GL_n(\F)$.  
Ceci est un résultat de Gelfand et Kazhdan quand $\CR$ est le corps des
nombres complexes.  
La preuve, écrite en détail dans \cite[Theorem 7.3]{BZ1}, utilise le théorème
6.10 de {\it ibid.} avec $n=2$, dont la preuve n'est pas valable 
{\it a priori} si la caractéristique de $\CR$ vaut $2$. 
Remarquons que cette preuve ne peut s'étendre telle quelle au cas où $\D$ est 
différente de $\F$ car la transposée d'une matrice inversible de $\Mat_m(\D)$ 
n'est pas toujours inversible. 
\end{rema}


\subsubsection{}
\label{lemmegeo}

Dans ce paragraphe, on donne une version combinatoire du lemme géométrique
de Bernstein-Zelevinski du paragraphe  \ref{Tellmarch} dans le cas particulier
où $\G=\G_m$ (voir \cite{BZ2}). 

Soient $\a=(m_1,\dots,m_r)$ et $\b=(n_1,\dots,n_s)$ deux familles d'entiers 
de sommes toutes deux égales à $m\>1$.
Pour chaque $i\in\{1,\ldots,r\}$, soit $\pi_i$ une représentation 
ir\-ré\-duc\-ti\-ble de $\G_{m_i}$ et notons $\pi$ le produit tensoriel 
$\pi_1\otimes\cdots\otimes\pi_r\in\Irr(\M_{\a})$. 
On note $\mathscr{M}^{\a,\b}$ l'ensemble des ma\-tri\-ces $\B=(b_{i,j})$ composées
d'entiers positifs tels que~:
\begin{equation*}
\label{2} 
\sum_{j=1}^s b_{i,j}=m_i, \quad i\in\{1,\ldots,r\}, 
\quad \sum_{i=1}^rb_{i,j}=n_j, \quad j\in\{1,\ldots,s\}.
\end{equation*}
Fixons $\B\in\mathscr{M}^{\a,\b}$ et notons $\a_i=(b_{i,1},\dots,b_{i,s})$ et 
$\b_j=(b_{1,j},\dots,b_{r,j})$, qui sont des familles de somme $m_i$ et 
de $n_j$ respectivement. 
Pour chaque $i\in\{1,\ldots,r\}$, on écrit~:
\begin{equation*}
\s^{(k)}_{i}=\s^{(k)}_{i,1}\otimes\cdots\otimes\s^{(k)}_{i,s}, 
\quad
\s^{(k)}_{i,j}\in\Irr(\G_{b_{i,j}}),
\quad 
k\in\{1,\ldots,r_i\},
\end{equation*}
les différents sous-quotients irréductibles de $\rp_{\a_i}(\pi_i)$.
Pour tout $j\in\{1,\ldots,s\}$ et toute famille d'entiers
$(k_1,\ldots,k_r)$ tels que $1\<k_i\<r_i$, 
on définit une représentation $\sigma_{j}$ de $\G_{n_j}$ par~: 
\begin{equation*}
\s_{j}=\ip_{\b_j}
\(\s^{(k_1)}_{1,j}\otimes\cdots\otimes\s^{(k_r)}_{r,j}\). 
\end{equation*}
Alors les représentations~: 
\begin{equation*}
\label{Druon}
\s_{1}\otimes\cdots\otimes\s_{s}, 
\quad
\B\in \mathscr{M}^{\a,\b},
\quad
1\<k_i\<r_i,
\end{equation*}
sont les sous-quotients irréductibles de
$\rp_{\b}(\ip_{\a}(\pi))$. 
Voir Zelevinski \cite[\S1.6]{Ze2}, la preuve étant va\-la\-ble pour 
$\R$ algébriquement clos de ca\-rac\-té\-ris\-ti\-que 
différente de $p$ et pour $\D$ quel\-con\-que.

\medskip

La proposition suivante est la combinaison du lemme \ref{mult1} 
et de \cite[Corollaire 2.2]{Min1} dont la preuve est valable pour 
un corps algébriquement clos quel\-con\-que de ca\-rac\-té\-ris\-ti\-que 
différente de $p$. 
On reprend les notations ci-dessus avec $s=2$.

\begin{prop}
\label{mmm1}
On suppose que, pour tout $i\in\{1,\ldots,r\}$, toute famille d'entiers 
$\a_i$ de somme $m_i$ et tout sous-quotient 
$\s^{(k)}_{i}=\s^{(k)}_{i,1}\otimes\s^{(k)}_{i,2}$ de 
$\rp_{\a_i}(\pi_i)$, 
on a~:
\begin{equation*}
\cusp(\s^{(k)}_{i,2})\nleqslant\sum\limits_{i<j\<r}\cusp(\pi_j).
\end{equation*}
Alors $\pi$ apparaît avec multiplicité $1$ dans 
$\sy{\rp_{\a}(\ip_{\a}(\pi))}$. 
Ainsi $\ip_{\a}(\pi)$ a une unique sous-repré\-senta\-tion irréductible, 
dont la multiplicité dans $\sy{\ip_{\a}(\pi)}$ est égale à $1$. 
\end{prop}

\section{Représentations modulaires de $\GL_n$ sur un corps fini}
\label{ApA}

Sont réunis dans cette section les résultats de la théorie des représentations 
modulaires de $\GL_n$ sur un corps fini dont nous avons besoin.  
On fixe un corps fini $\ff$ de caractéristique $p$ et de cardinal $q$.
La référence principale est l'article de James \cite{James} (voir aussi 
\cite{DJ1,JamesBook}).  

\subsection{Préliminaires}
\label{PrelimFini4}

Pour $m\>1$, on pose $\GA_{m}=\GL_{m}(\ff)$, et on 
note $\GA_{0}$ le grou\-pe trivial. 
\label{UniSCfini}
On note $\overline{\Irr}$ l'ensemble des classes d'iso\-mor\-phis\-me de 
représen\-ta\-tions irréductibles des $\GA_{m}$, $m\>0$, 
et $\overline{\Gg}$ le $\ZZ$-module libre de base $\overline{\Irr}$. 
On note $\overline{\Cc}$ et $\overline{\Ss}$ les sous-ensembles de 
$\overline{\Irr}$ constitués res\-pec\-ti\-ve\-ment des clas\-ses 
de re\-pré\-sen\-ta\-tions cuspidales et supercuspidales. 

Si $\a=(m_{1},\ldots,m_{r})$ est une famille 
d'entiers $\>0$ de somme $m$, il lui correspond le 
sous-groupe de Levi standard $\MA_{\a}$ de $\GA_{m}$ constitué des matrices 
diagonales par blocs de tailles $m_{1},\ldots,m_{r}$ respectivement, que 
l'on identifie naturellement au produit 
$\GA_{m_{1}}\times\cdots\times\GA_{m_{r}}$. 
On note $\PA_{\a}$ le sous-groupe para\-bo\-li\-que de 
$\GA_{m}$ de facteur de Levi $\MA_{\a}$ formé des matrices 
tri\-an\-gu\-lai\-res supérieures par blocs de tailles 
$m_{1},\ldots,m_{r}$ respectivement, et on note $\NA_{\a}$ 
son radical unipotent.
On note $\IA_\a$ le foncteur d'induction parabolique de $\MA_\a$ dans 
$\GA_m$ le long de $\PA_\a$.  
Si, pour cha\-que 
$i\in\{1,\ldots,r\}$, on a une représentation $\pi_{i}$ 
de $\GA_{m_i}$, on notera~: 
\begin{equation}
\label{VentreDieuFini}
\pi_1\times\cdots\times\pi_r
=\IA_{\a}(\pi_1\otimes\cdots\otimes\pi_r).
\end{equation}
Si $\pi$ est une représentation de longueur finie de $\GA_{m}$
pour un $m\>0$,
on note $\sy{\pi}$ son ima\-ge dans $\overline{\Gg}$ et $\deg(\pi)=m$ 
son degré. 
L'application degré et la formule (\ref{VentreDieuFini}) font de 
$\overline{\Gg}$ une $\ZZ$-al\-gè\-bre asso\-cia\-tive commutative 
gra\-duée. 
La commutativité de $\overline{\Gg}$ se déduit par exemple de \cite{HL}, 
où How\-lett et Lehrer prouvent que l'induction parabolique ne dépend pas 
du sous-groupe para\-bolique quand le facteur de Levi est fixé. 

\label{ParUniSCfini}

Toute représentation irréductible a un support cuspidal et un support 
super\-cuspi\-dal, \ie qu'on a une application surjective à fibres finies~:
\begin{equation*}
\cusp:\overline{\Irr}\to\Dive(\overline{\Cc})
\end{equation*}
telle que, étant données des représentations irréductibles cuspidales 
$\s_1,\dots,\s_n$, on ait~:
\begin{equation*}
\cusp^{-1}(\sy{\s_1}+\dots+\sy{\s_n})=\{\text{classes de quotients 
  irréductibles de $\s_1\times\dots\times\s_n$}\}
\end{equation*}
et on a une application surjective à fibres finies~:
\begin{equation*}
\scusp:\overline{\Irr}\to\Dive(\overline{\Ss})
\end{equation*}
telle que, étant données des représentations irréductibles supercuspidales 
$\s_1,\dots,\s_n$, on ait~:
\begin{equation*}
\scusp^{-1}(\sy{\s_1}+\dots+\sy{\s_n})=\{\text{classes de sous-quotients 
  irréductibles de $\s_1\times\dots\times\s_n$}\}
\end{equation*} 
(voir \cite[III.2.5]{Vig1}).

\subsection{Représentations non dégénérées}
\label{Lasso}

On rappelle quelques définitions concernant les partitions
(voir Macdo\-nald \cite[I.1]{MacDo} pour plus de précisions).
Une partition d'un entier $m\>1$ est une suite décroissante~:
\begin{equation*}
\label{DefPartAp}
\mu=(\mu_1\>\mu_2\>\dots\>\mu_r)
\end{equation*}
d'entiers strictement positifs dont la somme est égale à $m$.
Sa partition conjuguée est la partition
$\mu'=(\mu'_1\>\mu'_2\>\dots\>\mu'_s)$ de $m$ où $\mu'_j$ 
est le nombre d'entiers $i\in\{1,\dots,r\}$ tels que $\mu_i\>j$.
Si $\mu,\nu$ sont des partitions de $m$, on écrit $\mu\trianglelefteq\nu$ 
lorsque~: 
\begin{equation*}
\sum\limits_{i\<k}\mu_i\<\sum\limits_{i\<k}\nu_i
\end{equation*}
pour tout $k\>1$. 
On a $\mu\trianglelefteq\nu$ si et seulement si $\nu'\trianglelefteq\mu'$, 
\ie que l'application $\mu\mapsto\mu'$ est décroissante
pour $\trianglelefteq$.
Pour cette relation, la plus grande 
partition de $m$ est $(m)$, tandis que la plus petite est
$(1,\dots,1)$.

Si $\mu=(\mu_1 \> \mu_2\> \dots)$ est une partition d'un entier $m$ 
et $\nu=(\nu_1 \> \nu_2\> \dots)$ une partition d'un entier $n$, on 
note $\mu+\nu$ la partition
$(\mu_1+\nu_1\>\mu_2+\nu_2 \>\dots)$
de l'entier $m+n$. 

\begin{rema}
On peut penser à une partition comme à un élément de $\Dive(\NN^{*})$.
\end{rema}

\label{RNDWh}

Soit $\GA=\GA_m$ pour un $m\>1$.
On note $\NA=\NA_{(1,\dots,1)}$ le sous-groupe des matrices unipotentes 
triangulaires supérieures de $\GA$. 
On fixe un $\CR$-caractère non trivial $\psi$ de $\ff$.
Pour ce qui suit, on renvoie plus spécifiquement à \cite[III.1-2]{Vig1}.

Pour toute partition $\mu=(\mu_1\>\mu_2\>\dots\>\mu_r)$ de $m$, 
on note $\psi_{\mu}$ le caractère de $\NA$ défini par~:
\begin{equation*}
\label{DefPsiMu}
\psi_{\mu}:u\mapsto\psi\Big(\sum\limits_{i\in\J_\mu}u_{i,i+1}\Big),
\end{equation*}
la somme portant sur l'ensemble $\J_\mu$
de tous les entiers $i\in\{1,\dots,m-1\}$ qui ne sont 
de la forme $\mu_1+\dots+\mu_k$ pour aucun $k\in\{1,\dots,r\}$.

\begin{defi}
Une représentation $\pi$ de $\GA$ est $\mu$-\textit{dégénérée} 
si $\Hom_{\NA}(\pi,\psi_{\mu})$ est non nul. 
\end{defi}

Si $\mu$ est la partition maximale $(m)$, 
le caractère $\psi_{(m)}$ est non dégénéré
et on re\-trou\-ve la notion classique de modèle de Whittaker 
par réciprocité de Frobenius. 
On dit que $\pi$ est 
\textit{non dégénérée} si elle est $(m)$-dégénérée. 
Le résultat suivant est une conséquence de \cite[III.1.10]{Vig1}. 

\begin{prop}
\label{Adeg}
Soient $\s_1, \dots, \s_r$ des représentations irréductibles.  
Les conditions suivan\-tes sont équivalentes~:
\begin{enumerate}
\item 
L'induite $\s_1\times\dots\times\s_r$
possède un sous-quotient irréductible non dégénéré.
\item 
Pour tout $1 \< i \< r$, la représentation $\s_i$ est non dégénérée. 
\end{enumerate}
Si elles son satisfaites, ce sous-quotient irréductible 
non dégénéré est unique et sa multiplicité dans l'induite est $1$.
\end{prop}

\subsection{Classification de James}
\label{ClasJam}
\label{DefDSM}

Soient $m,n\>1$ et soit $\s$ une représentation irréductible 
cuspidale de $\GA_m$. 
On décrit dans ce paragraphe la classification par James \cite{James} 
des représentations irréductibles de $\GA_{mn}$ isomorphes à un 
sous-quotient de 
$\s\times\dots\times\s$ en fonction des par\-titions de $n$.

Soit $\Hh=\Hh(\s,n)$ l'algèbre des
endomorphismes de $\s\times\dots\times\s$,
qui est une algèbre de Hecke de type $\A_{n-1}$
et de paramètre $q^{m}$.
Cette représentation induite est presque-projective au sens de Dipper 
\cite{DipHom1} (voir aussi \cite[4.1]{MS11}), 
On a ainsi une bijection entre classes de représentations 
irréductibles isomorphes à un quotient de $\s\times\dots\times\s$ et 
classes de $\Hh$-modules irréductibles.

\begin{rema}
\label{ClasJamRem}
Si $\CR$ est de caractéristique non nulle $\ell$, 
et si $q^m$ n'est pas congru à $1$ modulo $\ell$, 
alors $\s$ est une représentation projective (voir \cite[III.2.9]{Vig1}).
\end{rema}

D'après \cite[6]{James}, il existe une unique sous-représentation 
irréductible de $\s\times\dots\times\s$, notée~:
\begin{equation}
\label{Defzsn}
\z(\s,n),
\end{equation}
dont le $\Hh$-module associé soit le caractère trivial de $\Hh$.  

\begin{theo}[James \cite{James}]
\label{ResumeJames}
Soit $\mu=(\mu_1\>\mu_2\>\dots\>\mu_r)$ une par\-ti\-tion de $n$.
\begin{enumerate}
\item
La représentation $\z(\s,\mu_1)\times\dots\times\z(\s,\mu_r)$ admet un unique 
sous-quotient irré\-duc\-tible~: 
\begin{equation*}
\z(\s,\mu)
\end{equation*} 
dégénéré relativement à la partition 
$m\mu'=(m\mu'_1\>m\mu'_2\>\dots)$.
\item
Les sous-quotients irréductibles de l'induite 
$\z(\s,\mu_1)\times\dots\times\z(\s,\mu_r)$ 
sont de la forme $\z(\s,\nu)$ avec $\mu\trianglelefteq\nu$, 
et $\z(\s,\mu)$ apparaît dans cette induite avec multiplicité $1$.
\item
L'application~:
\begin{equation*}
\mu\mapsto\z(\s,\mu)
\end{equation*}
est une bijection entre partitions de $n$ et 
classes de re\-pré\-sen\-ta\-tions irré\-ductibles de $\GA_{mn}$
isomorphes à un sous-quotient de $\s\times\dots\times\s$.
\end{enumerate}
\end{theo}

\begin{rema}
\label{defmsrem}
\begin{enumerate}
\item
On note $e(\s)$ le plus petit entier $k\>1$ tel que~: 
\begin{equation*}
1+q^m+\dots+q^{m(k-1)}=0
\end{equation*} 
dans $\CR$.  
(Si $\CR$ est de caractéristique nulle, on convient que $e(\s)=+\infty$.)
Alors $\z(\s,\mu)$ est isomorphe à un quotient de 
$\s\times\dots\times\s$ si et seulement si $\mu$ est $e(\s)$-régulière, 
\ie si aucun $\mu_i$ n'est ré\-pé\-té plus de $e(\s)$ fois, ou encore si 
$\mu'_j<e(\s)$ pour tout $j$. 
\item
Si $\s$ est supercuspidale, l'application 
$\mu\mapsto\z(\s,\mu)$ est une bijection 
entre par\-titions de $n$ et 
classes de re\-pré\-sen\-ta\-tions irré\-ductibles de $\GA_{mn}$ de 
support super\-cus\-pidal $n\cdot\sy{\s}=\sy{\s}+\dots+\sy{\s}$.
\end{enumerate}
\end{rema}

\label{Clothier}

La représentation~:
\begin{equation*}
\st(\s,n)=\z(\s,(1,\dots,1))
\end{equation*}
est l'unique sous-quotient non dégénéré de l'induite $\s\times\dots\times\s$, 
dans laquelle elle est de mul\-ti\-pli\-cité $1$.
Le résultat suivant fournit une classification 
de $\overline{\Cc}$ en fonction de $\overline{\Ss}$. 

\begin{prop}
\label{AppCuspFini}
On a les propriétés suivantes~: 
\begin{enumerate}
\item 
$\st(\s,n)$ est cuspidale si et seulement si $n=1$ ou $n=e(\s)\ell^r$, 
avec $r\>0$.
\item
L'application~:
\begin{equation*}
(\s,r)\mapsto\st_r(\s)=\st(\s,e(\s)\ell^r)
\end{equation*}
définit une bijection entre $\overline{\Ss}\times\NN$ et l'ensemble des classes 
de représentations irréductibles cuspi\-dales non supercuspidales. 
\end{enumerate}
\end{prop}

On a le résultat suivant. 
Voir le paragraphe \ref{Antisthene} pour la définition de $\r_\ell$. 

\begin{prop}
\label{RedDSM}
Soit $\s$ une $\flb$-représentation irréductible cuspidale de $\GA_m$, 
et soit $\tilde\s$ une $\qlb$-représenta\-tion cuspidale 
relevant $\s$ (\ie que $\r_\ell(\tilde\s)=\s$).
Soit $n\>1$. 
\begin{enumerate}
\item
On a $\r_{\ell}(\z(\tilde\s,n))=\z(\s,n)$.
\item
La représentation $\r_\ell(\st(\tilde\s,n))$ est irréductible si et 
seulement si $n<e(\s)$. 
\item 
Il existe une structure entière $\vv$ de 
$\st(\tilde\s,n)$ telle que $\st(\s,n)$ soit une 
sous-repré\-sen\-tation de $\vv\otimes\flb$.
\end{enumerate}
\end{prop}

\begin{proof}
D'après \cite{James}, les sous-quotients irréductibles de 
$\r_\ell(\st(\tilde\s,n))$ sont de la forme $\z(\s,\mu)$ avec 
$(n)\trianglelefteq\mu$ et $\z(\s,n)$ apparaît avec multiplicité $1$. 
Donc $\r_\ell(\st(\tilde\s,n))$ est irréductible. 

Pour les deux autres assertions, voir \cite[III.2.8]{Vig1} et \cite[III.2.4, Remarque 3]{Vig1}. 
\end{proof}

\label{Deflsn}

De façon analogue à \eqref{Defzsn}, 
l'induite $\s\times\dots\times\s$ a 
un unique quotient irréductible~: 
\begin{equation}
\label{deflsn}
l(\s,n),
\end{equation}
dont le $\Hh$-module associé soit le caractère signe de $\Hh$.  

\begin{prop}
\label{prpisegKfini}
On a $l(\s,n)=\st(\s,n)$ si et seulement si $n<e(\s)$.
\end{prop}

\begin{proof}
L'une des implications provient du fait que $\st(\s,n)$ correspond à un module 
simple de $\Hh$ si et seulement si la partition $(1,\dots,1)$ est $e(\s)$-régulière, 
\ie si $n<e(\s)$. 

Inversement, supposons que $n<e(\s)$. 
Comme le groupe $\GB$ est fini et comme le résultat est connu quand $\CR$ 
est de caractéristique nulle, 
on peut supposer que $\CR=\flb$. 
Soit $\tilde\s$ une $\qlb$-re\-pré\-senta\-tion cuspidale relevant $\s$.
D'après la proposition \ref{RedDSM}, $\st(\s,n)$ est la 
réduction mod $\ell$ de $\st(\tilde\s,n)$.
Le $\Hh(\tilde\s,n)$-module cor\-res\-pondant à $\st(\tilde\s,n)$ est le 
$\qlb$-caractère signe. 
Par conséquent, le $\Hh(\s,n)$-module corres\-pondant à $\st(\s,n)$ est le 
$\flb$-carac\-tère signe.
\end{proof}

\section{Représentations modulaires des algèbres de Hecke affines}
\label{ApB}

Dans cette section, on a réuni les résultats de la théorie des représentations 
modulaires des algèbres de Hecke affines dont nous aurons besoin.  
Il s'agit de $\CR$-algèbres de Hecke affines de type A en un paramètre non nul
qui peut être une racine de l'unité. 
Le résultat principal est le théorème \ref{ACG} de classification des modules 
simples en termes de multisegments apériodiques d'après 
\cite{Ze2,Rog2,Ariki,CG,AM,Mathas}. 
On définit aussi des caractères (définition \ref{DefZLcarH}) qui serviront 
dans la section \ref{segments} pour définir les réprésentations irréductibles 
attachées à un segment. 

\subsection{L'algèbre de Hecke affine}
\label{FamCarHeckeIwa}

Soient $n\>1$ et $\xi\in\mult\CR$. 
On note $\Hh_n=\Hh_\CR(n,\xi)$ la $\CR$-algèbre de Hecke affine engendrée 
par les symboles $\SS_1,\dots,\SS_{n-1}$ et $\X_1,\dots,\X_n$ et leurs inverses 
$(\X_1)^{-1},\dots,(\X_n)^{-1}$ avec les relations~:
\begin{eqnarray}
\label{R1aff}
(\SS_i+1)(\SS_i-\xi)&=&0, \quad i\in\{1,\dots n-1\},\\
\SS_i\SS_j&=&\SS_j\SS_i, \quad |i-j|\>2,\\
\SS_i\SS_{i+1}\SS_i&=&\SS_{i+1}\SS_i\SS_{i+1},\quad i\in\{1,\dots n-2\},\\
\X_i\X_j&=&\X_j\X_i, \quad i,j\in\{1,\dots,n\},\\
\X_j\SS_i&=&\SS_i\X_j, 
\quad i\notin\{j,j-1\},\\
\label{Rderaff}
\SS_i\X_{i}\SS_i&=&\xi\X_{i+1}, \quad i\in\{1,\dots,n-1\},
\end{eqnarray}
auxquelles s'ajoutent les relations 
$\X_j(\X_j)^{-1}=(\X_j)^{-1}\X_j=1$ pour $j\in\{1,\dots,n\}$. 

Compte tenu de ces relations,
on peut définir deux familles de caractères de $\Hh_n$. 

\begin{defi}
\label{DefZLcarH}
Soient $a,b\in\ZZ$ tels que $a\<b$, et soit $n=b-a+1$.
\begin{enumerate}
\item
On note $\Zz(a,b)$ le caractère de $\Hh_{n}$ défini par~:
\begin{equation*}
\SS_i\mapsto \xi,\ i\in\{1,\dots,n-1\},
\quad
\X_j\mapsto \xi^{a+j-1},\ j\in\{1,\dots,n\}.
\end{equation*}
\item
On note $\Ll(a,b)$ le caractère de $\Hh_{n}$ défini par~:
\begin{equation*}
\SS_i\mapsto-1,\ i\in\{1,\dots,n-1\},
\quad 
\X_j\mapsto\xi^{b-j+1},\ j\in\{1,\dots,n\}.
\end{equation*}
\end{enumerate}
\end{defi}

Soit $\a=(n_1,\dots,n_r)$ une famille d'entiers $\>1$ de somme $n$.  
Notons $\Hh_\a$ la sous-algèbre de $\Hh_{n}$ engendrée par les $\X_i$ 
et leurs inverses et les $\SS_i$ avec $i$ différent de $n_1+\dots+n_k$ 
pour tout $k\in\{1,\dots,r\}$.
Elle est canoniquement isomorphe à 
$\Hh_{n_1}\otimes\dots\otimes\Hh_{n_r}$. 

\begin{exem}
\label{DefIabExem}
Dans le cas où $\a=(1,\dots,1)$, la sous-algèbre $\Hh_{(1,\dots,1)}$ 
est commutative et égale à $\CR[\X_1^{\pm1},\dots,\X_n^{\pm1}]$. 
On note~:
\begin{equation}
\label{DefIab}
\Ii(a,b)
\end{equation}
le $\Hh_{n}$-module des $\Hh_{(1,\dots,1)}$-homomorphismes 
de $\Hh_{n}$ dans la restriction de $\Zz(a,b)$ à 
$\Hh_{(1,\dots,1)}$. 
Les caractères $\Zz(a,b)$, $\Ll(a,b)$ apparaissent respectivement comme 
sous-module et quo\-tient de $\Ii(a,b)$, tous deux avec multiplicité $1$. 
\end{exem}

On note $\EuScript{Z}_{n}=\EuScript{Z}_\CR(n,\xi)$ le centre de $\Hh_{n}$. 
C'est la sous-algèbre de $\CR[\X_1^{\pm 1},\dots,\X_n^{\pm 1}]$ cons\-ti\-tuée 
des élé\-ments invariants sous l'action du groupe symétrique 
${\mathfrak{S}_n}$ sur les $\X_i$. 

\begin{defi}
Le \textit{caractère central} d'un $\Hh_{n}$-module irréductible est 
le caractère de $\EuScript{Z}_{n}$ par lequel celui-ci agit sur ce module.
\end{defi}

Le caractère central d'un $\Hh_{n}$-module irréductible est caractérisé par 
$n$ scalaires non nuls et non ordonnés, \ie par un 
multi-ensemble de longueur $n$ dans $\Dive(\mult\CR)$. 
On a une application~:
\begin{equation*}
{\rm cent}:\Irr(\Hh_{n})\to\Dive(\mult\CR)
\end{equation*}
surjective et à fibres finies, associant à un $\Hh_{n}$-module irréductible son 
caractère central. 

Le caractère central joue pour les modules un rôle analogue à celui du support 
cuspidal pour les représentations. 
Pour un $\Xi\in\Dive(\mult\CR)$ de longueur $n$, 
on note $\Irr(\Hh_{n},\Xi)$ l'ensemble des classes 
d'isomorphisme de $\Hh_{n}$-modules irréductibles de caractère central $\Xi$. 

\begin{theo}
\label{Rogaton}
Soit $r\>1$ un entier et soient $z_1,\dots,z_r\in\mult\CR$ tels que 
$z_i/z_j\notin\xi^\ZZ$ pour tous $i\neq j$. 
Pour chaque $i$, soit $n_i\>1$ un entier et soit $\Xi_i\in\Dive(z_i\xi^\ZZ)$.  
On note $\Xi$ la somme des $\Xi_i$, qu'on suppose de longueur $n$, 
et on pose $\a=(n_1,\dots,n_r)$. 
\begin{enumerate}
\item 
Pour chaque $i$, soit $\mm_i\in\Irr(\Hh_{n_i},\Xi_i)$.
Alors $\Hom_{\Hh_\a}(\Hh_n,\mm_1\otimes\dots\otimes\mm_r)$ est un 
$\Hh_n$-module irréducti\-ble de caractère central $\Xi$. 
\item
L'application ainsi définie~:
\begin{equation*}
\Irr(\Hh_{n_1},\Xi_1)\times\dots\times\Irr(\Hh_{n_r},\Xi_r)\to\Irr(\Hh_n,\Xi)
\end{equation*}
est bijective. 
\end{enumerate}
\end{theo}

\begin{proof}
La méthode de Rogawski \cite[4.1]{Rog2} est encore valable ici. 
\end{proof}

\subsection{L'algèbre de Hecke-Iwahori}
\label{AHI}

On rappelle que $q$ désigne le cardinal du corps résiduel de $\F$~; 
on note $\xi$ son image dans $\CR$. 
Soit $n\>1$ un entier et soit $\Iw$ le sous-groupe d'Iwahori standard de
$\G=\GL_{n}(\F)$. 
On note $e_1,\dots,e_n$ les vecteurs de la base canonique du 
$\F$-espace vectoriel $\F^n$.

Pour $i\in\{1,\dots,n-1\}$, on note $s_i$ la matrice de 
permutation transposant $e_i$ et $e_{i+1}$ et laissant invariants 
les autres vecteurs de la base, et on note $\SS_i$ la fonction 
ca\-rac\-téristique de $\Iw s_i\Iw$.

Pour $j\in\{0,\dots,n\}$, on note $t_j$ la matrice diagonale 
de $\G$ agissant par $e_k\mapsto\varpi e_k$ pour chaque 
$k\in\{1,\dots,j\}$ et laissant invariants les autres vecteurs de la base, 
et $\T_j$ la fonc\-tion ca\-rac\-téristique de $\Iw t_j\Iw$, 
qui est inversible.
Pour tout $j\in\{1,\dots,n\}$, on pose~:
\begin{equation*}
\X_j=\xi^{-j+(n+1)/2}\T_{j-1}(\T_{j})^{-1}.
\end{equation*}
Alors les éléments $\SS_1,\dots,\SS_{n-1}$ et $\X_1,\dots,\X_n$ et leurs 
inverses $(\X_1)^{-1},\dots,(\X_n)^{-1}$ engendrent l'algèbre 
$\Hh_\CR(\G,\Iw)$ des fonctions de $\G$ dans $\CR$ à support compact 
et bi-invariantes par $\I$, 
avec les relations \eqref{R1aff} à \eqref{Rderaff} ci-dessus. 
On a ainsi un isomorphis\-me~:
\begin{equation*}
\Upsilon:\Hh_\CR(\G,\Iw)\to\Hh_{\CR}(n,\xi)=\Hh_n
\end{equation*}
de $\CR$-algèbres. 

\begin{rema}
La description par générateurs 
et relations ci-dessus se déduit de celle donnée dans 
\cite{MW} en appliquant l'involution $\flat$ de $\Hh_\CR(\G,\Iw)$ 
définie par~:
\begin{equation*}
\SS_i^{\flat}=\SS_{n-i}^{},\ i\in\{1,\dots,n-1\},
\quad
\X_j^{\flat}=\X_{n+1-j}^{-1},\ j\in\{1,\dots,n\}.
\end{equation*}
\end{rema}

Soit $\a=(n_1,\dots,n_r)$ une famille d'entiers $\>1$ de somme égale à $n$.  
L'intersection $\I\cap\M_\a$ est le sous-groupe d'Iwahori standard 
du sous-groupe de Levi $\M_\a$.
L'algèbre $\Hh_\CR(\M_\a,\I\cap\M_\a)$ a des générateurs~:
\begin{equation*}
\SS_{i,j}, \quad \X_{i,k}, \quad (\X_{i,k})^{-1}, \quad
i\in\{1,\dots,r\},\ j\in\{1,\dots,n_i-1\},\ k\in\{1,\dots,n_i\},
\end{equation*}
et on a un isomorphis\-me de $\CR$-algèbres~:
\begin{equation*}
\Upsilon_\a:\Hh_\CR(\M_\a,\I\cap\M_\a)\to
\Hh_\CR(n_1,\xi)\otimes\dots\otimes\Hh_\CR(n_r,\xi)=\Hh_{\a}. 
\end{equation*}
Il existe un unique morphisme de $\CR$-algèbres $\te_\a$ de 
$\Hh_\CR(\M_\a,\I\cap\M_\a)$ dans $\Hh_\CR(\G,\I)$ tel que~: 
\begin{equation*}
\te_\a(\SS_{i,j})=\SS_{n_1+\dots+n_{i-1}+j},
\quad
\te_\a(\X_{i,k})=\X_{n_1+\dots+n_{i-1}+k},
\end{equation*}
pour $i\in\{1,\dots,r\}$, $j\in\{1,\dots,n_i-1\}$, $k\in\{1,\dots,n_i\}$, 
et le morphisme composé $\Upsilon\circ\te_\a$ est égal à $\Upsilon_\a$. 




\subsection{Classification des modules irréductibles}
\label{BlaACGM}

Rappelons que $\xi$ désigne l'image de $q$ dans $\CR$.  
Si $\CR$ est de caractéristique non nulle, on note $e$ le 
plus petit entier $\>1$ tel que~: 
\begin{equation*}
1+\xi+\dots+\xi^{e-1}=0
\end{equation*}
dans $\CR$.
Sinon, on pose $e=+\infty$.

\begin{defi} 
\label{DefSeg11formel}
\begin{enumerate}
\item 
Un \textit{segment} est une suite finie de la forme~:
\begin{equation*}
\left[a,b \right]=(\xi^a,\xi^{a+1},\dots,\xi^{b})
\end{equation*} 
où $a,b\in\ZZ$ sont des entiers tels que $a\<b$. 
\item
Un \textit{multisegment} est un multi-ensemble de segments. 
\end{enumerate}
\end{defi}

Si $m$ est fini, un multisegment est dit \textit{apériodique} s'il 
ne contient aucun multisegment de la forme~:
\begin{equation*}
[a,b]+[a+1,b+1]+\dots+[a+e-1,b+e-1].
\end{equation*}
Si $e=+\infty$, on convient que tout multisegment est apériodique. 
Désignons par $\Psi$ l'ensemble des multi\-segments apériodiques. 
Étant donné un multisegment apériodique~:
\begin{equation*}
\psi=[a_1,b_1]+\dots+[a_r,b_r]\in\Psi,
\end{equation*}
on note $\supp(\psi)$ la somme formelle 
des $\sy{\xi^{j}}$ pour $i\in\{1,\dots,r\}$ et $j\in\{a_i,\dots,b_i\}$, 
qu'on appelle le \textit{support} de $\psi$. 
Étant donné $\Xi\in\Dive(\xi^\ZZ)$, on note $\Psi(\Xi)$ 
l'ensemble des multi\-seg\-ments apériodiques de support $\Xi$. 

\begin{rema}
\label{RX1}
On suppose que $\CR$ est de caractéristique $\ell$ non nulle et que $\xi$ vaut
$1$. 
Un multisegment s'identifie à une partition et $\Psi$ à l'ensemble 
des partitions $\ell$-régulières. 
\end{rema}

Le théorème suivant donne la classi\-fication des $\Hh_n$-modules 
irréductibles en termes de multisegments apériodiques. 

\begin{theo}
\label{ACG}
Il existe une bijection~:
\begin{equation*}
\psi\mapsto\Zz_\psi
\end{equation*} 
entre les multisegments apériodiques de longueur $n$
et les classes d'isomorphisme de $\Hh_n$-modules 
irréductibles à caractère central dans $\Dive(\xi^\ZZ)$
telle que
le caractère central de $\Zz_\psi$ soit $\supp(\psi)$. 
\end{theo}

\begin{proof}
Si $\CR$ est de caractéristique $0$, le résultat est une conséquence de 
la classification de Zelevinski \cite{Ze2} 
(voir aussi Rogawski \cite{Rog2}). 

Si $\CR$ est de caractéristique $\ell$ non nulle et si $\xi\neq1$, 
il est dû à Ariki et Mathas 
\cite[Theorem B]{AM} et s'appuie sur \cite{Ariki} 
et les travaux de Chriss-Ginzburg \cite{CG}.

Si $\CR$ est de caractéristique $\ell$ non nulle et si $\xi=1$, 
alors $\Psi$ est (remarque \ref{RX1})
l'ensemble des par\-ti\-tions $\ell$-régu\-lières. 
D'après Mathas \cite[Theorem 3.7]{Mathas}, il existe une bijection entre 
par\-ti\-tions $\ell$-régulières de $n$ et 
classes de $\Hh$-modules irré\-duc\-ti\-bles de caractère central 
$\sy{1}+\dots+\sy{1}=n\cdot\sy{1}$. 
(Mathas utilise la notion de partition $\ell$-res\-trein\-te~; on 
se ramène à celle de partition $\ell$-régulière par conjugaison.)
\end{proof}

On en tire le corollaire suivant.

\begin{coro}
\label{BlaMathas}
Pour $\Xi\in\Dive(\xi^\ZZ)$ de longueur $n$, 
les ensembles finis $\Psi(\Xi)$ et $\Irr(\Hh_n,\Xi)$ ont le même cardinal.
\end{coro}

\section{Rappels sur les représentations cuspidales}
\label{InvRho}

Dans cette section, on rappelle certains résultats obtenus 
dans \cite{MS11} dont nous aurons besoin par la suite. 
Dans le paragraphe \ref{paramTSM}, on rappelle un peu du vocabulaire 
de la théorie des types qui nous sera nécessaire. 
Dans le paragraphe \ref{BenBen}, on rappelle la définition de certains 
invariants associés à une représentation irréductible cuspidale. 
Dans le paragraphe \ref{DefKMS}, on rappelle les 
prin\-ci\-pales propriétés du foncteur $\KM$ défini et étudié dans 
\cite[\S5]{MS11}.  

\subsection{Types simples maximaux}
\label{paramTSM}

Soient $m\>1$ un entier et $\G=\G_{m}$. 
Dans \cite{MS11}, nous prouvons que toute représentation 
ir\-ré\-duc\-ti\-ble cus\-pi\-da\-le de $\G$ s'obtient 
par induction compacte d'une représentation irréductible 
d'un sous-groupe ouvert compact modulo le centre de $\G$
(\textit{ibid.}, théorème 3.11).
Plus précisément, nous construisons une famille de paires 
$(\J,\l)$ formées d'un sous-groupe ouvert compact $\J$ de 
$\G$ et d'une représentation irréductible $\l$ de $\J$ possédant 
les propriétés suivantes~:
\begin{enumerate}
\item
pour toute représentation irréductible cuspidale $\rho$ de $\G$, il existe 
une paire $(\J,\l)$, unique à conjugaison près, telle que la restriction de 
$\rho$ à $\J$ ait une sous-représentation isomorphe à $\l$~;
\item
deux représentations irréductibles cuspidales de $\G$ contiennent une 
même paire $(\J,\l)$ si et seulement si elles sont inertiellement 
équivalentes~;
\item
étant donnés une représentation irréductible cuspidale $\rho$ de $\G$ et 
une paire $(\J,\l)$ contenue dans $\rho$, il existe une unique représentation 
irréductible 
prolongeant $\l$ à son normalisateur dont l'induite compacte à $\G$ soit 
isomorphe à $\rho$. 
\end{enumerate}
De telles paires sont appelées des types simples maximaux de $\G$. 
Leur construction 
est technique et implique de nombreux paramètres. 
Nous renvoyons à \cite{MS11} pour un exposé détaillé, que nous résumons 
ci-dessous. 
Étant donné un type simple maximal $(\J,\l)$ de $\G$, il lui correspond~: 
\begin{enumerate}
\item 
des sous-grou\-pes ouverts distingués $\H^1\subseteq\J^1$ de $\J$~;
\item
un caractère $\t$ de $\H^1$, appelé caractère simple~; 
\item 
une extension finie $\E$ de $\F$ contenue dans $\Mat_{m}(\D)$~;
\item
un entier $m'\>1$ et une extension finie $\kk$ de $\kk_\E$.  
\end{enumerate}
Ces objets possèdent les propriétés suivantes~: 
\begin{enumerate}
\item
le centralisateur de $\E$ dans $\J$ est un sous-groupe compact maximal du 
centralisateur de $\E$ dans $\G$~; 
\item
le quotient $\J/\J^1$ est iso\-mor\-phe au groupe $\GL_{m'}(\kk)$~;
\item
la représentation $\l$ se décompose sous la forme $\k\otimes\s$, 
où $\s$ est une représentation irré\-duc\-tible de $\J$ triviale sur $\J^1$ 
s'identifiant à une représentation cuspidale de $\GL_{m'}(\kk)$, et où 
$\k$ est une représentation irré\-duc\-tible de $\J$ dont la restriction à 
$\J^1$ est l'unique représentation irré\-duc\-tible de $\J^1$ dont la restriction à 
$\H^1$ contient $\t$. 
\item
la représentation $\k$ n'est pas unique | elle ne l'est qu'à torsion près par 
un caractère de $\J$ trivial sur $\J^1$ | mais elle 
peut être choisie de façon à avoir le même ensemble 
d'entrelacement que $\t$ dans $\G$, 
auquel cas on dit que c'est une $\b$-extension de $\t$. 
\end{enumerate}

\subsection{Invariants associés à une représentation cuspidale}
\label{BenBen}

Dans \cite{MS11}, on associe à toute représentation 
ir\-ré\-duc\-ti\-ble cus\-pi\-da\-le $\rho$ de $\G$ 
un ca\-rac\-tère non ramifié $\nu_\rho$ de $\G$ et 
des entiers $n(\rho),f(\rho),q(\rho),\e(\rho),b(\rho),s(\rho)$ 
(\textit{ibid.}, $3.4, 4.5$) 
qui ne dépendent que de la classe d'iner\-tie de $\rho$.
Rappelons que~:
\begin{eqnarray}
n(\rho)&=&\text{nombre de }\chi:\G\to\mult\CR
\text{ non ramifiés tels que } \rho\chi\simeq\rho,\\
f(\rho)&=&md/e(\E/\F),
\quad 
e(\E/\F)=\text{ indice de ramification de } \E/\F,\\
\qr&=&q^{f(\rho)}=\text{cardinal d'une extension de degré } m' 
\text{ de } \kk,\\
\e(\rho)&=&\text{ordre (éventuellement infini) de $\qr$ dans $\mult\R$}.
\end{eqnarray}
Puis, si $(\J,\k\otimes\s)$ est un type simple maximal de $\G$ contenu dans 
$\rho$, alors~:
\begin{eqnarray}
\label{defbe}
b(\rho)&=&\text{ cardinal de la } \Gal(\kk/\kk_\E)\text{-orbite de } \s,\\
s(\rho)&=&\text{ ordre du stabilisateur de } \s \text{ dans } 
\Gal(\kk/\kk_\E),\\
\nu_\rho&=&\nu^{s(\rho)}.
\end{eqnarray}
Rappelons 
(voir \cite[4.5]{MS11}) 
la propriété importante suivante du caractère non ramifié 
$\nu_\rho$. 

\begin{prop}
\label{cuspidal}
Soit $\rho'$ une représentation irréductible cus\-pi\-da\-le 
de $\G_{m'}$, $m'\>1$.
Alors l'in\-duite $\rho\times\rho'$ est réductible si et seulement si 
$m'=m$ et $\rho'$ est isomorphe à $\rho\nu_\rho^{}$ ou $\rho\nu_{\rho}^{-1}$.
\end{prop}

On pose maintenant~:
\begin{equation}
\ZZ_{\rho}=\{\sy{\rho\nu_{\rho}^{i}}\ |\ i\in\ZZ\}.
\end{equation}
Si $\CR$ est de caractéristique $\ell$ non nulle, on note~:
\begin{equation}
\label{MRHO}
\ee(\rho)
\end{equation}
le plus petit $k\>1$ tel que $1+\qr+\dots+\qr^{k-1}$ 
soit congru à $0$ modulo $\ell$. 
Il vaut $\ell$ si $\e(\rho)=1$, et il vaut $\e(\rho)$ sinon.
Si $\CR$ est de caractéristique nulle, on convient que 
$\ee(\rho)=\e(\rho)=+\infty$. 

\begin{rema}
\label{Comparemsmr}
L'entier $\ee(\rho)$ est égal à l'entier $e(\s)$ défini à la remarque \ref{defmsrem}. 
\end{rema}

On a la propriété importante suivante. 

\begin{lemm}[\cite{MS11}, lemme 4.41]
\label{rappell441}
On a $\e(\rho)={\rm card}\ \ZZ_\rho$. 
\end{lemm}


Enfin, on associe à $\rho$ une endo-classe~:
\begin{equation*}
\boldsymbol{\Theta}_{\rho}
\end{equation*}
qui est l'endo-classe commune aux caractères simples contenus dans 
$\rho$. 
Ce processus est décrit dans \cite[9.2]{BSS1} pour les représentations 
complexes et 
fonc\-tion\-ne de façon simi\-laire pour les re\-pré\-sen\-ta\-tions modulaires.

Notons $\ind^\G_\J(\l)$ l'induite compacte de $\l$ à $\G$. 

\begin{prop}
\label{Meleagant}
Pour tout $n\>1$,
il y a un iso\-mor\-phisme canonique de $\R$-algèbres~:
\begin{equation}
\label{Meleagant1}
\Psi_{\rho,n}:\Hh_{\R}(n,\qr)\to
\End_{\G}(\ind^\G_\J(\l)\times\dots\times\ind^\G_\J(\l))
\end{equation} 
où $\ind^\G_\J(\l)$ est répété $n$ fois.
\end{prop}

\begin{proof}
Voir le paragraphe \ref{AHI} et \cite[Proposition 4.30]{MS11}. 
\end{proof}

\begin{rema}
Ces deux algèbres ne dépendent que de la classe d'inertie de $\rho$, 
mais ce n'est pas vrai de l'isomorphisme $\Psi_{\rho,n}$ entre les deux, 
qui dépend de $\rho$. 
\end{rema}

\begin{rema}
Soit ${\J}^\sharp$ le normalisateur de $\J$ dans $\G$. 
D'après \cite[3.1]{MS11}, il existe une unique repré\-sen\-ta\-tion $\l^\sharp$ de 
${\J}^\sharp$ prolongeant $\l$ et dont l'induite à $\G$ soit isomorphe à 
$\rho$. 
Alors $\Psi_{\rho,1}$ est l'iso\-morphisme de $\R[\X,\X^{-1}]$ dans 
$\End_{\G}(\ind^\G_\J(\l))$ tel que $\Psi_{\rho,1}(\X)$ correspond par 
réciprocité de Frobenius au $\J$-homomorphisme $f$ de $\l$ dans 
l'induite de $\l$ à $\G$ défini par~:
\begin{equation*}
f(v)(g) = 
\left\{
\begin{array}{ll}
\l^\sharp(g)v & \text{si $g\in\J^\sharp$},\\
0 & \text{sinon},
\end{array}
\right.
\end{equation*}
pour $v$ dans l'espace de $\l$ et pour $g\in\G$.
Pour $n\>1$, l'iso\-morphisme $\Psi_{\rho,n}$ est 
caractérisé par une condition de compatibilité 
à l'induction (voir \cite[4.4]{MS11}).
\end{rema}

Dorénavant, les deux algèbres de \eqref{Meleagant1} seront identifiées. 
Pour tout entier $n\>1$, on note $\Om_{\rho,n}$ la classe inertielle de 
$n\cdot\sy{\rho}=\sy{\rho}+\dots+\sy{\rho}$. 
On rappelle que $\Irr(\Om_{\rho,n})^\q$ est l'ensemble des classes de 
représentations irréductibles dont le support cuspidal est dans 
$\Om_{\rho,n}$. 

\begin{prop}[\cite{MS11}, \S4.4]
\label{BijXin}
On a une bijection~:
\begin{equation*}
\boldsymbol{\xi}_{\rho,n}:\Irr(\Om_{\rho,n})^\q\to
\{\text{classes d'isomorphisme de $\Hh(n,\qr)$-modules simples}\}. 
\end{equation*}
\end{prop}


Fixons une extension finie $\F'/\F$ comme dans \cite[4.4]{MS11}. 
En particulier, son corps résiduel est de cardinal $\qr$. 
Si l'on applique ce qui précède au caractère trivial de $\F'^{\times}$, 
on obtient pour tout entier $n\>1$ une bijection $\boldsymbol{\xi}_{1_{\F'^{\times}},n}$ 
entre $\Irr(\Om_{1_{\F'^{\times}},n})^\q$ et l'ensemble des classes d'isomorphisme de 
$\Hh(n,\qr)$-modules simples. 
La composée~:
\begin{equation}
\label{BeautyCeleste}
{\bf\Phi}_{\rho,n}^{}=\boldsymbol{\xi}_{1_{\F'^{\times}},n}^{-1}
\circ\boldsymbol{\xi}_{\rho,n}^{}:
\Irr(\Om^{}_{\rho,n})^\q\to\Irr(\Om_{1_{\F'^{\times}},n})^\q
\end{equation}
est une bijection entre les ensembles 
$\Irr(\Om^{}_{\rho,n})^\q$ et $\Irr(\Om_{1_{\F'^{\times}},n})^\q$
compatible au support cuspidal d'après 
\cite[Proposition 4.33]{MS11}.

\begin{theo}[\cite{MS11}, théorème 4.18]
\label{ss} 
Soient $r\>1$ un entier et $\rho_1,\dots,\rho_r$ des 
représenta\-tions irré\-ductibles cuspidales deux à deux non 
inertiellement équivalentes. 
Pour chaque entier $i\in\{1,\dots,r\}$, on fixe un support cuspidal 
$\ss_i$ formé de représentations inertiellement équivalentes 
à $\rho_i$. 
\begin{enumerate}
\item 
Pour chaque entier $i$, soit $\pi_i$ une représentation irréductible de 
support cuspidal $\ss_i$.
Alors l'induite $\pi_1\times\dots\times\pi_r$ est irréductible.
\item
Soit $\pi$ une représentation irréductible de 
support cuspidal $\ss_1+\dots+\ss_r$.
Alors il existe des représentations $\pi_1,\dots,\pi_r$, 
uniques à isomorphisme près, telles que $\pi_i$ soit de support 
cuspidal $\ss_i$ pour chaque $i$ 
et telles que $\pi_1\times\dots\times\pi_r$ soit 
isomorphe à $\pi$.
\end{enumerate}
\end{theo}

Le résultat suivant affine le théorème \ref{ss}. 
Soient $m,m'\>1$ des entiers. 

\begin{prop}
\label{Zdroites}
Soient $\rho$ et $\rho'$ des représentations irréductibles cuspidales, 
respectivement de $\G_m$ et $\G_{m'}$.
On suppose que $\ZZ_\rho\neq\ZZ_{\rho'}$.
Soient $\pi$ et $\pi'$ des représentations irréductibles telles que 
$\cusp(\pi)\in\Dive(\ZZ_\rho)$ et $\cusp(\pi')\in\Dive(\ZZ_{\rho'})$.
Alors l'induite $\pi\times\pi'$ est irréductible. 
\end{prop}

\begin{proof}
D'après le théorème \ref{ss}, il suffit de traiter le cas où $\rho$ et $\rho'$ 
sont iner\-tiellement équivalentes.  
La méthode est la même que celle de la preuve de \cite[Proposition 4.37]{MS11}. 
Compte tenu de \cite[Propositions 4.13 et 4.20]{MS11}, on se ramène à un problème 
d'irréductibilité d'un module induit sur une algèbre de Hecke affine. 
Le résultat est alors une conséquence du théorème \ref{Rogaton}. 
\end{proof}

\subsection{Le foncteur $\KM$}
\label{DefKMS}

D'après le paragraphe \ref{paramTSM}, on a des sous-groupes ouverts 
compacts $\J\supseteq\J^1$ de $\G$, et on pose 
$\GB=\J/\J^1$, qu'on identifie à $\GL_{m'}(\kk)$. 
On note $\boldsymbol{\Theta}$ l'endo-classe associée à $\rho$. 
Dans \cite[\S5]{MS11}, on définit un foncteur~:
\begin{equation}
\label{LabelKmax}
\KM=\KM_{\k}:\Rr(\G)\to\Rr(\GB)
\end{equation} 
défini par $\pi\mapsto\Hom_{\J^1}(\k,\pi)$ et 
possédant les propriétés suivantes~:
\begin{enumerate}
\item 
il est exact~;
\item
il envoie représentations admissibles sur représentations de 
dimension finie~; 
\item
étant données des représentations irréductibles cuspidales 
$\rho_1,\dots,\rho_r$ d'endo-classe $\boldsymbol{\Theta}$ et 
dont la somme des degré vaut $m$, 
alors on a $\KM(\pi)\neq0$ pour tout sous-quotient irréductible $\pi$ de 
l'induite $\rho_1\times\dots\times\rho_r$.
\end{enumerate}

\begin{exem}
Si $\rho$ est de niveau $0$ et si $\k$ est le caractère trivial de 
$\GL_{m}(\Oo_\D)$, alors $\KM$ est le foncteur associant à toute
représentation de 
$\G$ la représentation de $\GB=\GL_m(\kk_\D)$ sur l'espace de ses invariants 
sous $1+\Mat_{m}(\p_\D)$. 
\end{exem}

On a la formule très utile suivante. 
Soit $(\J,\k\otimes\s)$ un type simple maximal contenu dans $\rho$. 
On renvoie à \eqref{defbe} pour la définition de l'entier $b(\rho)$. 

\begin{prop}[\cite{MS11}, lemme 5.3]
\label{DeliceDOrient}
On a un isomorphisme de représentations de $\GB$~:
\begin{equation*}
\KM(\rho)\simeq\s\oplus\s^{\Fr}\oplus\dots\oplus\s^{\Fr^{b(\rho)-1}}
\end{equation*}
où $\Fr$ est un générateur de $\Gal(\kk/\kk_\E)$.  
\end{prop}

On en tire la formule plus générale suivante. 
Soient $\rho_1,\dots,\rho_r$ des représentations irréduc\-tibles cus\-pidales 
d'endo-classe $\boldsymbol{\Theta}$ et dont la somme des degré 
$m_1+\dots+m_r$ vaut $m$.
Pour chaque $i$, il y a un type simple maximal $(\J_i,\k_i\otimes\s_i)$ 
contenu dans $\rho_i$, où $\k_i$ est une $\b$-extension compatible à $\k$ 
(voir \cite[5.3]{MS11} pour cette notion de compatibilité). 

\begin{prop}[\cite{MS11}, corollaire 5.16]
\label{DeliceDOrientR}
On a~: 
\begin{equation*}
\KM(\rho_1\times\dots\times\rho_{r})
\simeq \bigoplus\limits_{(i_1,\dots,i_r)}
\s_1^{\Fr^{i_1}}\times\dots\times\s_r^{\Fr^{i_r}},
\end{equation*}
où $\Fr$ est un générateur de $\Gal(\kk/\kk_\E)$ et 
où chaque $i_j$ décrit $\{0,\dots,b(\rho_j)-1\}$. 
\end{prop}

En procédant comme dans \cite[Remarque 5.20]{MS11}, on associe à tout entier 
$n\>1$ une $\b$-extension $\k_{n}$ 
définie sur un sous-groupe ouvert compact $\J_{n}$ de $\G_{mn}$ 
et un foncteur~: 
\begin{equation}
\label{Boublil}
\KM_{n}:\Rr(\G_{mn})\to\Rr(\GB_{m'n})
\end{equation}
défini par $\pi\mapsto\Hom_{\J^1_{n}}(\k_{n},\pi)$, où  
$\J_{n}^{}/\J_{n}^{1}$ est identifié à $\GB_{m'n}$. 
Pour $n=1$, on retrouve \eqref{LabelKmax}. 
Si $\a=(n_1,\dots,n_r)$ est une famille d'entiers $\>1$ 
de somme $n$, on a un foncteur~:
\begin{equation*}
\KM_{\a}:\Rr(\M_{(mn_1,\dots,mn_r)})\to\Rr(\MB_{(m'n_1,\dots,m'n_r)})
\end{equation*}
défini par $\pi\mapsto\Hom_{\J^1_{n_1}\times\dots\times\J^{1}_{n_r}}
(\k_{n_1}\otimes\dots\otimes\k_{n_r},\pi)$.
Par \cite[Propositions 5.11, 5.12 et 5.18]{MS11} et la proposition 
\ref{DeliceDOrient} ci-dessus, on a les propriétés suivantes.

\begin{prop}
\label{ResumePropKM}
\begin{enumerate}
\item
Pour chaque $i\in\{1,\dots,r\}$, soit $\pi_i$ une représentation 
irréductible de $\G_{mn_i}$. 
Alors on a $\KM_n(\pi_1\times\dots\times\pi_r)\simeq
\KM_{n_1}(\pi_1)\times\dots\times\KM_{n_r}(\pi_r)$. 
\item
En particulier, on a~:
\begin{equation}
\label{Phoebe}
\KM_{n}(\rho\chi_1\times\dots\times\rho\chi_{n})\simeq
\bigoplus\limits_{(i_1,\dots,i_r)}
\s^{\Fr^{i_1}}\times\dots\times\s^{\Fr^{i_n}}
\end{equation}
où $\Fr$ est un générateur de $\Gal(\kk/\kk_\E)$ et 
où chaque $i_j$ décrit $\{0,\dots,b(\rho)-1\}$. 
\item
Si $\chi_1,\dots,\chi_n$ sont des caractères non ramifiés de $\G$, 
alors $\KM_{n}(\pi)$ est non nul pour tout sous-quotient irréductible 
$\pi$ de $\rho\chi_1\times\dots\times\rho\chi_n$.
\item
Si $\pi$ est une représentation de $\G_{mn}$, 
on a $\KM_{n}(\pi)^{\NB_{(m'n_1,\dots,m'n_r)}}\simeq
\KM_\a(\rp_{(mn_1,\dots,mn_r)} (\pi))$ en tant que représentations de 
$\MB_{(m'n_1,\dots,m'n_r)}$.
\end{enumerate}
\end{prop}

\subsection{Relèvement des représentations supercuspidales}

Dans ce paragraphe, on rappelle une propriété importante de relèvement 
des $\flb$-représentations supercuspidales 
(voir \cite{MS11}, paragraphes 2.8 et 3.6).

\begin{theo}
\label{RappelLiftSupercusp}
Soit $\rho$ une $\flb$-représentation irréductible supercuspidale 
contenant un $\flb$-type simple maximal $(\J,\l)$. 
Supposons que $\l$ est de la forme $\k\otimes\s$ avec $\s$ 
super\-cuspidale. 
Alors il y a une $\qlb$-représentation irréductible cuspidale entière $\tilde\rho$ 
de $\G_m$ et un $\qlb$-type simple maxi\-mal $(\J,\tilde\l)$ tels que~:
\begin{enumerate}
\item
$\tilde\rho$ contient le type simple maximal $(\J,\tilde\l)$~;
\item
La réduction modulo $\ell$ de $\tilde\rho$ est égale à $\rho$ et 
la réduction modulo $\ell$ de $\tilde\l$ est égale à $\l$. 
\end{enumerate}
\end{theo}

\begin{rema}
\label{oups}
On verra plus loin (théorème \ref{RappelLiftSupercuspComplet}) 
que l'hypothèse sur $\s$ est superflue. 
\end{rema}

\section{Classification des représentations cuspidales par les 
  supercuspidales} 
\label{ClassiPot}

Dans cette section, on établit 
la classification des représentations irré\-duc\-tibles 
cuspidales en fonction des supercuspidales (paragraphe 
\ref{classi-supercuspidal}).  
On en déduit l'unicité du support supercuspidal à inertie près
(paragraphe \ref{USCIP}).  

\subsection{Classification des représentations cuspidales par les supercuspidales}
\label{classi-supercuspidal}

\textit{Dans ce paragraphe, on suppose que 
$\CR$ est de caractéristique $\ell$ non nulle.}
Soient $m,n\>1$ des entier et $\rho$ une représentation irréductible 
cuspidale de $\G=\G_m$. 
Soit $(\J,\k\otimes\s)$ un type simple maxi\-mal contenu dans $\rho$. 

\begin{lemm}
\label{sscirsc}
Si $\s$ est supercuspidale, alors $\rho$ est supercuspidale. 
\end{lemm}

\begin{proof}
Supposons que $\rho$ n'est pas supercuspidale. 
Il y a des représentations irréduc\-ti\-bles supercuspidales 
$\rho_1,\dots,\rho_r$, avec $\rho_i$ de degré $m_i\>1$ pour chaque 
$i\in\{1,\dots,r\}$, 
telles que $\rho$ soit un sous-quotient irréductible de 
$\rho_1\times\dots\times\rho_r$.
D'après le théorème \ref{ss} et comme $\rho$ est cuspidale, 
$\rho_1,\dots,\rho_r$ sont inertiellement équivalentes.
En outre, l'endo-classe commune aux $\rho_i$ 
est égale à celle de $\rho$
d'après \cite[Corollaire 5.10]{MS11}.
Selon les propositions \ref{DeliceDOrient} et \ref{DeliceDOrientR} dont on 
reprend les no\-ta\-tions, il y a des entiers $i_1,\dots,i_r$ tels que $\s$ soit 
un sous-quotient de~:
\begin{equation*}
\s_1^{\Fr^{i_1}}\times\dots\times\s_r^{\Fr^{i_r}}. 
\end{equation*}
Par conséquent, $\s$ n'est pas supercuspidale.
\end{proof}

\begin{rema}
Réciproquement, si $\s$ n'est pas supercuspidale, 
il existe des représentations irréduc\-ti\-bles supercuspidales 
$\s_1,\dots,\s_r$ 
telles que $\s$ soit un sous-quotient de 
$\s_1\times\dots\times\s_r$.
Pour chaque $i\in\{1,\dots,r\}$, on peut construire à partir de $\k$ 
et de $\s_i$ un type simple maximal $(\J_i,\l_i)$ de $\G_{m_i}$ 
pour un entier $m_i\>1$ convenable, de telle sorte que $\rho$ 
soit un sous-quotient de~:
\begin{equation*}
\ind_{\J_1}^{\G_{m_1}}(\l_1)\times\dots\times\ind_{\J_1}^{\G_{m_r}}(\l_r)
\simeq\ip_{\a}(\ind^{\M}_{\J_\M}(\l_\M))
\end{equation*}
avec $\a=(m_1,\dots,m_r)$, $\M=\M_\a$ et $(\J_\M,\l_\M)=
(\J_1\times\dots\times\J_r,\l_1\otimes\dots\otimes\l_r)$.
À ce stade, pour prouver que $\rho$ est supercuspidale, il nous faudrait 
un résultat analogue à \cite[Corollaire B.1.3]{Dat6}.
Nous procédons différemment ici (voir le lemme \ref{pfuit}). 
L'analogue à \cite[Corollaire B.1.3]{Dat6} est traité dans un travail 
en cours du second auteur avec S.~Stevens. 
\end{rema}

Formons le foncteur $\KM_{n}$ comme en \eqref{Boublil}.
On rappelle (voir le paragraphe \ref{Clothier}) que $\s\times\dots\times\s$ 
($n$ fois) possède un unique 
sous-quotient irréductible non dégénéré $\st(\s,n)$, 
apparaissant avec multiplicité $1$.
Compte tenu de \eqref{Phoebe}, ceci justifie la définition 
sui\-van\-te.  

\begin{defi}
\label{DefSTRN}
Pour tout entier $n\>1$, l'induite~:
\begin{equation*}
\rho\times\rho\nu_\rho^{}\times\dots\times\rho\nu_\rho^{n-1}
\end{equation*}
possède un unique sous-quotient irréductible $\pi$ tel que 
$\KM_{n}(\pi)$ admette le facteur $\st(\s,n)$ comme sous-quotient. 
On le note $\St(\rho,n)$, et il apparaît dans cette induite avec multiplicité $1$. 
\end{defi}

On renvoie à \eqref{MRHO} pour la définition de l'entier $\ee(\rho)$ 
associé à $\rho$. 

\begin{prop}
\label{StCusp}
La représentation $\St(\rho,n)$ est cuspidale si, et seulement si,
$n=1$ ou s'il existe un entier $r\>0$ tel que $n=\ee(\rho)\ell^r$.
\end{prop}

\begin{proof}
On suppose dans un premier temps que $\St(\rho,n)$ est cuspidale.
C'est l'unique sous-quotient irréductible de 
$\rho\times\rho\nu_\rho^{}\times\dots\times\rho\nu_\rho^{n-1}$ tel que~:
\begin{equation*}
\sy{\st(\s,n)}\<\sy{\KM_{n}(\St(\rho,n))}.
\end{equation*}
Par hypothèse de cuspidalité et d'après la proposition \ref{DeliceDOrient}, 
le membre de droite est 
une somme de représentations 
irréductibles cuspidales de $\GL_{m'n}(\kk)$.
On en déduit que $\st(\s,n)$ est cuspidale, ce qui implique, 
d'après la proposition \ref{AppCuspFini}, que $n=1$ ou qu'il existe un 
entier $r\>0$ tel que $n=\ee(\s)\ell^r$.  
La remarque \ref{Comparemsmr} permet de conclure. 

Inversement, on suppose qu'il existe $r\>0$ tel que $n=\ee(\rho)\ell^r$, 
et on va montrer que $\St(\rho,n)$ est cuspidale.
En premier lieu, on remarque que $\st(\s,n)$ est cuspidale d'après la proposition 
\ref{AppCuspFini} et la remarque \ref{Comparemsmr}. 
On a la propriété suivante, qui nous sera utile par la suite. 

\begin{lemm}
\label{lemmetwistst}
Étant donnés $z\in\mult\CR$ et $\pi\in\Irr$, 
on note $\pi_{z}$ la représentation $\pi$ tordue 
par le caractère non rami\-fié valant $z$
en un élément dont la norme réduite est de valuation $1$.
Alors, pour tout $z\in\mult\CR$, on a $\St(\rho_{z},n)=\St(\rho,n)_{z}$. 
\end{lemm}

La suite de la preuve se fait en deux étapes. 

\begin{lemm}
\label{looper}
Supposons que $\s$ est supercuspidale.
Alors $\St(\rho,n)$ est cuspidale.
\end{lemm}

\begin{proof}
D'après le lemme \ref{sscirsc}, $\rho$ est supercuspidale. 
Quitte à tordre $\rho$ par un caractère non ramifié, on peut supposer 
que son caractère central est à valeurs dans $\flb$, ce qui est 
justifié par le lemme \ref{lemmetwistst}. 
Dans ce cas, $\rho$ est définie sur $\flb$, et il suffit de prouver le 
résultat lorsque $\CR$ est le corps $\flb$, ce qu'on suppose. 

D'après le théorème \ref{RappelLiftSupercusp}, 
on relève $\rho$ en une $\qlb$-représentation ir\-ré\-ductible 
cuspidale entière $\tilde\rho$ en mê\-me temps qu'on relève 
$(\J,\k\otimes\s)$ en un type simple maxi\-mal 
$(\J,\tilde\k\otimes\tilde\s)$ contenu dans $\tilde\rho$.
Il corres\-pond à ce relèvement (voir le paragraphe \ref{DefKMS}) une $\b$-extension 
$\tilde\k_{n}$ et un foncteur~:
\begin{equation*}
\tilde\KM_{n}:\Rr_{\qlb}(\G_{mn})\to\Rr_{\qlb}({\GB}_{m'n}). 
\end{equation*}
Alors $\st(\tilde\s,n)$ est un sous-quotient de $\tilde\KM_{n}(\St(\tilde\rho,n))$.
Comme il s'agit de représentations d'un grou\-pe fini sur 
un corps de caractéristique $0$, il y a une 
sous-représentation $\tilde\omega$ telle que $\tilde\KM_{n}(\St(\tilde\rho,n))$ 
soit la somme directe de $\st(\tilde\s,n)$ et $\tilde\omega$. 
D'après la proposition \ref{RedDSM}, il 
existe une structure entière $\ll_1$ de 
$\st(\tilde\s,n)$ telle que $\st(\s,n)$ soit une sous-représentation 
de $\ll_1\otimes\flb$.
On choisit une structure entière quelconque $\ll_2$ de $\tilde\omega$, 
et on pose $\ll=\ll_1\oplus\ll_2$, 
qui est une structure entière de $\tilde\KM_{n}(\St(\tilde\rho,n))$. 
Comme dans \cite[III.5.13]{Vig1}, on prouve le résultat suivant.  

\begin{lemm}
Soit $\mathfrak{k}_{n}$ une structure entière de $\tilde\k_{n}$.
Il existe une structure entière $\vv$ de $\St(\tilde\rho,n)$ telle que
$\Hom_{\J^1_{n}}(\mathfrak{k}_{n},\vv)=\ll$. 
\end{lemm}

Par exactitude, on en déduit l'isomorphisme~:
\begin{equation*}
\KM_{n}(\vv\otimes\flb)\simeq\ll\otimes\flb
\end{equation*}
entre $\flb$-représentations 
dont $\st(\s,n)$ est une sous-représentation. 
Il existe donc un sous-quotient irréductible $\pi$ de 
$\vv\otimes\flb$
tel que $\st(\s,n)$ soit une sous-représentation de $\KM_{n}(\pi)$, 
et il est cuspidal car c'est une représentation irréductible 
contenant le type simple maximal $\k_{n}\otimes\st(\s,n)$. 
Par la propriété d'unicité de $\St(\rho,n)$, on déduit de ceci 
que $\pi$ est isomorphe à $\St(\rho,n)$.
\end{proof}

\begin{lemm}
\label{pfuit}
Supposons que $\s$ n'est pas supercuspidale.
Alors $\rho$ n'est pas supercuspidale et $\St(\rho,n)$ est cuspidale.
\end{lemm}

\begin{proof}
D'après la proposition \ref{AppCuspFini}, il y a un $t\>2$ divisant 
$m'$ et une repré\-sentation irréductible supercuspidale $\s_0$ de 
$\GL_{m'/t}(\kk)$ tels que $\s$ soit isomorphe à $\st(\s_0,t)$ et tels qu'on ait 
$t=\ee(\s_0)\ell^{s}$ pour un certain $s\>0$.

Si $(\J_0,\l_0)$ est un type simple maximal de $\G_{m/t}$ de la forme 
$\k_0\otimes\s_0$, où $\k_0$ est une $\b$-extension d'un 
transfert (voir \cite[2.2]{MS11}) 
du caractère simple $\t$ contenu dans $\k$, alors il lui correspond 
par le procédé décrit dans \cite[5.2]{VS3} 
(voir plus précisément la proposition 5.4) 
un type simple $(\J',\k'\otimes\s')$ de $\G$,
où $\s'$ est une représentation définie par 
inflation à partir du produit tensoriel $\s_0^{\otimes t}$ et où $\k'$ est 
une $\b$-extension d'un transfert de $\t$ dans $\G$.
Choisissons $\k_0$ de telle façon que la $\b$-extension $\k'$ ainsi obtenue 
soit le transfert de la $\b$-extension $\k$ au sens de \cite[2.3.5]{MS11}. 

Soit $\rho_0$ une représentation irréductible 
de $\G_{m/t}$ contenant le type simple maximal $\k_0\otimes\s_0$.  
Comme $\s_0$ est supercuspidale, le lemme \ref{sscirsc} implique que $\rho_0$ 
est supercuspidale et le lemme \ref{looper} implique que $\St(\rho_0,t)$ est 
cuspidale.
Quitte à tordre $\rho_0$ par un caractère non ramifié de $\G_{m/t}$, 
on peut donc supposer que $\rho$ et $\St(\rho_0,t)$ sont isomorphes
(voir le lemme \ref{lemmetwistst}).
Ceci prouve en particulier que $\rho$ n'est pas supercuspidale.

\begin{lemm}
\label{SpSp2}
Les représentations $\St(\rho,n)$ et $\St(\rho_0,tn)$ sont isomorphes.
\end{lemm}

\begin{proof}
D'abord, ce sont toutes deux des facteurs irréductibles de~:
\begin{equation*}
\rho_0\times\rho_0\nu_{\rho_0}^{}\times\dots\times\rho_0\nu_{\rho_0}^{tn-1}.
\end{equation*}
Ensuite $\st(\st(\s_0,t),n)=\st(\s_0,tn)$ est un 
sous-quotient de $\KM_{n}(\St(\rho,n))$, ce qui prouve 
le résultat attendu. 
\end{proof}

Comme $\ee(\rho)=\ell$, on a $tn=\ee(\rho_0)\ell^{s+r+1}$ et le lemme 
\ref{looper} appliqué à $\rho_0$ et $tn$
implique que la représentation $\St(\rho,n)$ est cuspidale. 
\end{proof}

Ceci met fin à la démonstration de la proposition \ref{StCusp}.
\end{proof}

Si l'on récapitule les lemmes \ref{sscirsc} et \ref{pfuit}, 
on a prouvé au passage le résultat suivant.

\begin{prop}
\label{sscersc}
Soit $\rho$ une représentation irréductible cuspidale
et soit $(\J,\k\otimes\s)$ un type simple maximal contenu dans $\rho$. 
Alors $\rho$ est supercuspidale si et seulement si 
$\s$ est supercuspidale.
\end{prop}

Ceci permet de simplifier la formulation du théorème 
\ref{RappelLiftSupercusp}. 

\begin{theo}
\label{RappelLiftSupercuspComplet}
Soit $\rho$ une $\flb$-représentation irréductible supercuspidale de $\G_m$, 
$m\>1$.
Alors il y a une $\qlb$-représentation irréductible cuspidale entière $\tilde\rho$ 
de $\G_m$ 
telle que la réduction modulo $\ell$ de $\tilde\rho$ soit égale à $\rho$. 
\end{theo}

Les constructions de types simples maximaux effectuées dans la preuve de la 
proposition \ref{StCusp} permettent d'obtenir le corollaire suivant. 

\begin{coro}
\label{NumStr}
Soient $\rho$ une représentation irréductible cuspidale et 
$r\>0$ un entier. 
On pose $\rho_r=\St(\rho,\ee(\rho)\ell^r)$.
On a~:
\begin{equation*}
n(\rho_r)=n(\rho)\e(\rho),
\quad
b(\rho_r)=b(\rho),
\quad
s(\rho_r)=s(\rho),
\quad 
f(\rho_r)=f(\rho)\ee(\rho)\ell^r,
\quad
\ee(\rho_r)=\ell.
\end{equation*}
\end{coro}

\begin{proof}
D'après le paragraphe \ref{Clothier}, $\st(\s,n)$
est l'unique sous-quotient irréductible non dégénéré de $\s\times\dots\times\s$, 
dans laquelle il est de mul\-ti\-pli\-cité $1$.
Pour $\g\in\Gal(\kk/\kk_\E)$, les repré\-sen\-tations $\st(\s,n)$ 
et $\st(\s^\g,n)$ sont donc isomorphes si et seulement si $\s$ et $\s^\g$ sont 
isomorphes. 
On en déduit que $b(\rho_r)=b(\rho)$, puis que $s(\rho_r)=s(\rho)$. 
L'égalité $f(\rho_r)=f(\rho)\ee(\rho)\ell^r$ suit de la définition de 
l'invariant $f$.
On en déduit $\e(\rho_r)=1$, donc $e(\rho_r)=\ell$.
L'égalité $n(\rho_r)=n(\rho)\e(\rho)$ suit de la formule (3.8) de
\cite{MS11}.
\end{proof}

On en déduit aussi le résultat suivant, qui sera utile dans la section 
\ref{Sec9}. 

\begin{coro}
\label{MlaDim}
Pour qu'il existe des caractères non ramifiés $\chi_1,\dots,\chi_n$
de $\G_m$ tels que l'induite $\rho\chi_1\times\dots\times\rho\chi_n$
possède un sous-quotient cuspidal, il faut et il suffit que 
$n=1$ ou qu'il existe un entier $r\>0$ tel que $n=\ee(\rho)\ell^r$.
\end{coro}

\begin{proof}
On suppose qu'il y a des caractères non ramifiés $\chi_1,\dots,\chi_n$ 
de $\G_m$ tels que $\rho\chi_1\times\dots\times\rho\chi_n$
possède un sous-quotient cuspidal $\pi$.
D'après la formule \eqref{Phoebe} et la proposition \ref{DeliceDOrient}, 
$\KM_{n}(\pi)$ est une représentation 
de $\GL_{m'n}(\kk)$ contenant un sous-quotient irréductible cuspidal 
de $\s\times\dots\times\s$.  
Il s'agit donc de $\st(\s,n)$, l'unique sous-quotient non dé\-gé\-néré de cette 
induite, et $n$ a la forme annoncée par la 
proposition \ref{AppCuspFini}.
La réciproque est une conséquence de la proposition \ref{StCusp}. 
\end{proof}

On a maintenant le théorème de classification suivant. 

\begin{theo}
\label{AppCuspSuper}
\begin{enumerate}
\item 
L'application~:
\begin{equation}
\label{MapCuspNS}
(\rho,r)\mapsto\St_r(\rho)=\St(\rho,\ee(\rho)\ell^r)
\end{equation}
définit une surjection de $\Ss\times\NN$ sur l'ensemble 
des classes de représentations irréductibles cus\-pi\-dales non supercuspidales. 
\item 
Soient $(\rho,r)$ et $(\rho',r')$ deus couples de $\Ss\times\NN$. 
Alors $\St_r(\rho)$ et $\St_{r'}(\rho')$ sont isomorphes si et seulement si
$r=r'$ et $\ZZ_{\rho}=\ZZ_{\rho'}$. 
\end{enumerate}
\end{theo}

\begin{rema}
Avec la proposition \ref{StCusp}, 
la partie 1 de ce théorème généralise les assertions 1 et 2 de 
\cite[III.5.14]{Vig1} au cas où $\D$ n'est pas nécessairement commutative. 
La partie 2 est nou\-vel\-le.
\end{rema}

\begin{proof}
Pour prouver la surjectivité, il suffit d'appliquer 
la proposition \ref{sscersc} puis les lem\-mes \ref{pfuit} et \ref{SpSp2}
avec $n=1$.

Soient maintenant $(\rho,r)$ et $(\rho',r')$ dans $\Ss\times\NN$ 
tels que $\St_r(\rho)$ et $\St_{r'}(\rho')$ soient isomor\-phes. 
D'abord $\rho$ et $\rho'$ ont la même endo-classe d'après 
\cite[Corollaire 5.10]{MS11}, ce dont on 
déduit que $\rho$ et $\rho'$ contiennent des types simples maximaux 
de la forme $\k_0^{}\otimes\s_0^{}$ et $\k_0^{}\otimes\s_0'$ respectivement, 
où $\s_0^{}$ et $\s_0'$ sont des représentations irréductibles 
supercuspidales du même groupe fini. 
En appliquant le foncteur $\KM_n$, on en déduit que $\st(\s'_0,r')$ 
et $\st(\s_0^{},r)$ sont conjuguées sous $\Gal(\kk/\kk_\E)$. 
On déduit de la proposition \ref{AppCuspFini} que $r=r'$ 
et que $\s_0^{}$ et $\s'_0$ sont $\Gal(\kk/\kk_\E)$-conjuguées, 
ce qui implique 
que $\rho$ et $\rho'$ sont inertiellement équivalentes. 

Soit maintenant $\X_\rho$ le groupe des caractères non ramifiés $\chi$ 
de $\G_m$ tels que $\sy{\rho\chi}\in\ZZ_{\rho}$. 
C'est un groupe cyclique contenant le sous-groupe 
des caractères non ramifiés stabi\-lisant $\sy{\rho}$.  
D'après le lemme \ref{rappell441}, 
il décrit dans la classe inertielle de $\rho$
une orbite de cardinal $\e(\rho)$. 
L'ordre de $\X_\rho$ est donc égal à $n(\rho)\e(\rho)$, qui est 
aussi, d'après le corollaire \ref{NumStr},
l'ordre du sous-groupe des caractères non ramifiés stabilisant la classe 
d'isomorphisme de $\St_r(\rho)$. 
Compte tenu du lemme \ref{lemmetwistst}, un caractère non ramifié $\chi$ de 
$\G_m$ vérifie donc 
$\St_r(\rho\chi)\simeq\St_r(\rho)$ si et seule\-ment si $\sy{\rho\chi}\in\ZZ_{\rho}$. 
Ceci met fin à la démonstration du théorème \ref{AppCuspSuper}. 
\end{proof}

\begin{rema}
\label{ValeCNS}
Si $\rho$ est cuspidale non supercuspidale, alors 
$\e(\rho)=1$ et $\ee(\rho)=\ell$.
\end{rema}

\subsection{Unicité du support supercuspidal à inertie près}
\label{USCIP}

Soient $m\>1$ un entier et $\G=\G_{m}$. 

\begin{prop}
\label{KOLA}
Soient $(\M,\vr)$ et $(\M',\vr')$ des paires supercuspidales de $\G$, 
et soient $\P$ et $\P'$ des sous-groupes paraboliques de $\G$ de 
facteurs de Levi respectifs $\M$ et $\M'$. 
On sup\-pose que $\ip^\G_\P(\vr)$ et $\ip^\G_{\P'}(\vr')$ 
ont un sous-quotient irréductible en commun.
Alors les paires
$(\M,\vr)$ et $(\M',\vr')$ sont iner\-tiel\-lement équivalentes.
\end{prop}

\begin{proof}
On peut supposer que les sous-groupes de Levi sont standards, \ie que 
$\M=\M_{\a}$ et $\M'=\M_{\a'}$ où $\a$ et $\a'$
sont chacunes des familles d'entiers de somme $m$, 
et on écrit $\vr=\rho_1\otimes\dots\otimes\rho_{n}$ et 
$\vr'=\rho'_1\otimes\dots\otimes\rho'_{n'}$. 
Soit $\pi$ un sous-quotient irréductible commun.  
En fixant une famille d'entiers $\g$ de somme $m$ 
telle que $\rp_\g(\pi)$ soit cuspidale et en 
appliquant le lemme géométrique, on se ramène au cas où $\pi$ 
est cuspidale. 
En raisonnant comme dans la preuve du lemme \ref{sscirsc}, on voit que 
$\rho_1,\dots,\rho_{n}$ sont inertiellement équivalentes et ont toutes la 
même endo-classe que $\pi$.
D'après la proposition \ref{DeliceDOrientR}, dont on reprend les notations, 
on a~:
\begin{equation*}
\KM(\rho_1\times\dots\times\rho_{n})=
\bigoplus\limits_{(i_1,\dots,i_n)}
\s_1^{\Fr^{i_1}}\times\dots\times\s_{n}^{\Fr^{i_n}}
\end{equation*}
et on peut même supposer que $\s_1=\dots=\s_n=\s$ car 
$\rho_1,\dots,\rho_{n}$ sont inertiellement équivalentes.
De façon analogue, $\rho'_1,\dots,\rho'_{n'}$ sont inertiellement équivalentes 
et ont toutes la même endo-classe que $\pi$, et on a~:
\begin{equation*}
\KM(\rho_1'\times\dots\times\rho_{n'}')=
\bigoplus\limits_{(i'_1,\dots,i'_{n'})}
\s'^{\Fr^{i'_1}}\times\dots\times\s'^{\Fr^{i'_{n'}}}
\end{equation*}
où chaque $\rho'_j$ contient un type simple maximal de la forme 
$(\J',\k'\otimes\s')$. 
Il existe donc des entiers $i_1,\dots,i_n$ et $i'_1,\dots,i'_{n'}$ tels que~:
\begin{equation*}
\s^{\Fr^{i_1}}\times\dots\times\s^{\Fr^{i_n}}
\quad
\text{et}
\quad 
\s'^{\Fr^{i'_1}}\times\dots\times\s'^{\Fr^{i'_{n'}}}
\end{equation*}
ont un sous-quotient irréductible en commun.
D'après la proposition \ref{sscersc}, les représentations $\s$ et $\s'$ 
sont supercuspidales. 
De l'unicité du support supercuspidal pour les 
représentations irréductibles de $\GB$
(voir le paragraphe \ref{ParUniSCfini}) on déduit que $n=n'$ et que~:
\begin{equation*}
\sy{\s^{\Fr^{i_1}}}+\dots+\sy{\s^{\Fr^{i_n}}}=\sy{\s'^{\Fr^{i'_1}}}+\dots+\sy{\s'^{\Fr^{i'_{n}}}}.
\end{equation*}
Quitte à réordonner $\rho_1,\dots,\rho_n$
(c'est-à-dire, à conjuguer $\vr$), on peut 
supposer que~:
\begin{equation*}
\sy{\s^{\Fr^{i_j}}}=\sy{\s'^{\Fr^{i'_j}}},
\quad j=1,\dots,n.
\end{equation*}
Puisque $\s$ et $\s'$ ont même degré et que $\rho_j^{}$ et 
$\rho'_j$ ont la même endo-classe,
on a $\deg(\rho^{}_j)=\deg(\rho_j')$. 
On déduit de \cite[Corollaire 5.5]{MS11} que 
$\rho^{}_j$ et $\rho'_j$ sont inertiellement équi\-va\-len\-tes.
Ainsi les paires supercuspidales $(\M,\vr)$ et $(\M',\vr')$ sont 
inertiellement équi\-va\-len\-tes.
\end{proof}

À l'aide de la proposition \ref{KOLA}, on prouve 
une variante inertielle du théorème \ref{ss}.
On utilise les notations du paragraphe \ref{Persephone}. 
Si $\rho$ est une représentation irréductible cuspidale, on note $\Om_\rho$ 
sa classe d'inertie. 

\begin{theo}
\label{ssvarscinertie}
Soit $r\>1$ un entier et soient $\rho_1,\dots,\rho_r$ des 
représentations irré\-ductibles supercuspidales deux à deux 
non inertiellement équivalentes. 
Pour chaque entier $i\in\{1,\dots,r\}$, soit une classe inertielle 
$\Om_i\subseteq\Dive(\Om_{\rho_i})$ de $\G_{m_i}$, $m_i\>1$, 
et soit $\Om$ la classe inertielle de $\G_m$, avec $m=m_1+\dots+m_r$, 
définie par $\Om_1,\dots,\Om_r$.
\begin{enumerate}
\item 
Pour chaque entier $i\in\{1,\dots,r\}$, 
soit $\pi_i\in\Irr(\Om_i)$ une représentation irréductible.
Alors l'induite $\pi_1\times\dots\times\pi_r$ est irréductible.
\item 
L'application~:
\begin{equation*}
(\pi_1,\dots,\pi_r)\to\pi_1\times\dots\times\pi_r
\end{equation*}
induit une bijection de $\Irr(\Om_1)\times\dots\times\Irr(\Om_r)$ 
dans $\Irr(\Om)$.
\end{enumerate}
\end{theo}

\begin{proof}
Décomposons chaque représentation $\pi_i$ sous la forme 
$\pi_i=\pi_{i,1}\times\pi_{i,2}\times\dots$ donnée par le théo\-rème \ref{ss}.  
Supposons qu'il existe deux couples $(i,j)\neq(i',j')$ tels qu'un 
terme $\sy{\rho_{i,j}}$ de $\cusp(\pi_{i,j})$ soit inertiellement inéquivalent 
à un terme $\sy{\rho_{i',j'}}$ de $\cusp(\pi_{i',j'})$.
Alors on a $i\neq i'$.
En appliquant un foncteur de Jacquet convenable, 
on fait appa\-raître que $\rho_{i,j}$ est un sous-quotient irréductible d'une 
induite d'un élément de $\Dive(\Om_{\rho_i})$ et que $\rho_{i',j'}$ est un 
sous-quotient irréductible d'une 
induite d'un élément de $\Dive(\Om_{\rho_{i'}})$.
D'après la proposition \ref{KOLA}, cela entraîne que $\rho_i$ et $\rho_{i'}$ 
sont inertiellement équivalentes, ce qui donne une contradiction puisque 
$i\neq i'$.
Le point (1) est alors une conséquence du théo\-rème \ref{ss}(1).  
Le point (2) s'obtient par un appauvrissement 
du théo\-rème \ref{ss}(2).  
\end{proof}

\section{La théorie des segments}
\label{segments}

Dans cette section, on définit la notion de segment 
(\S\ref{ParSeg4}), puis on associe à tout segment deux représentations 
irréductibles, dont on étudie les propriétés 
(\S\ref{GuyDomville}).
On montre ensuite (théorème \ref{nuevo2}) 
que la représentation induite à partir de 
représentations associées à des segments non liés 
(définition \ref{Rotis}) est irréductible.  

\subsection{Segments} 
\label{ParSeg4}

Soit un entier $m\>1$, et soit $\rho$ une re\-pré\-sen\-ta\-tion 
irréductible cuspidale de $\G_m$.

\begin{defi}
\label{DefSeg11}
Un {\it segment} est une suite finie de la forme~:
\begin{equation}
\left[a,b \right]_\rho=(\rho\nu_{\rho}^a,\rho\nu_{\rho}^{a+1},\dots,\rho\nu_{\rho}^{b}),
\end{equation} 
où $a,b\in\ZZ$ sont des entiers tels que $a\<b$. 
\end{defi}

Le segment $\left[a,b \right]_\rho$ peut être interprété comme la 
paire cuspidale~:
\begin{equation}
\label{PCAD}
(\M_{(m,\dots,m)},\rho\nu_{\rho}^a\otimes\rho\nu_{\rho}^{a+1}\otimes\dots
\otimes\rho\nu_{\rho}^{b}),
\end{equation}
où $\M_{(m,\dots,m)}$ 
est le sous-groupe de Levi standard de $\G_{m(b-a+1)}$
correspondant à la famille d'en\-tiers $(m,\dots,m)$.

\begin{defi}
\label{DefEquSeg}
Deux segments $\left[a,b\right]_{\rho}$ et $\left[a',b'\right]_{\rho'}$ 
sont \textit{équivalents} s'ils ont la même longueur $n$ et si
$\sy{\rho\nu_{\rho}^{a+i}}=\sy{\rho'\nu_{\rho'}^{a'+i}}$
pour tout $i\in\{0,\dots,n-1\}$.
\end{defi}

Si c'est le cas, on voit que $\rho'$ est inertiellement équivalent à 
$\rho$, ce qui implique $\nu_{\rho'}=\nu_{\rho}$.
Par conséquent, pour que deux segments soient équivalents, il suffit 
qu'ils aient la même longueur et que leurs extrémités initiales 
$\rho\nu_{\rho}^a$ et $\rho'\nu_{\rho'}^{a'}$ soient isomorphes. 

\medskip

Si $\Delta=\left[a,b\right]_{\rho}$ est un segment, on pose~: 
\begin{eqnarray}
\label{LDE}
n(\Delta)&=&b-a+1,\\
\deg(\Delta)&=&n(\Delta)m,\\
\supp(\Delta)&=&\sy{\rho\nu_\rho^a}+\sy{\rho\nu_\rho^{a+1}}+\dots+\sy{\rho\nu_\rho^b},
\end{eqnarray}
qu'on appelle respectivement la longueur, le degré et le support de $\Delta$.
Celui-ci est égal à la classe de $\G_{\deg(\Delta)}$-conjugaison de la 
paire cuspidale \eqref{PCAD} associée à $\Delta$. 
On note aussi~: 
\begin{eqnarray}
\label{LDEab}
a(\Delta)&=&\rho\nu_{\rho}^a,\\
b(\Delta)&=&\rho\nu_{\rho}^b, 
\end{eqnarray}
les extrémités initiale et finale de $\Delta$, et on note~:
\begin{equation}
\label{contragr}
\Delta^\vee=\left[-b,-a \right]_{\widetild{\rho}}
\end{equation}
le segment contragrédient de $\Delta$. 
Si $a+1\<b$, on pose~: 
\begin{eqnarray}
\label{DefMoins}
^-\Delta&=&\left[a+1,b\right]_\rho,\\
\label{DefPlus}
\Delta^-&=&\left[a,b-1\right]_\rho.
\end{eqnarray}
Les définitions suivantes généralisent celles de Zelevinski 
\cite[4.1]{Ze2}. 
Remarquons qu'elles dif\-fè\-rent de celles de Vignéras \cite[V.3]{Vig1}. 

\begin{defi}
\label{Rotis}
Soient $\Delta=\left[a,b\right]_\rho$ et $\Delta'=\left[a',b'\right]_{\rho'}$
des segments. 
\begin{enumerate}
\item 
On dit que $\Delta$ {\it précède} $\Delta'$ si l'on peut extraire 
de la suite~:
\begin{equation*}
(\rho\nu_{\rho}^a ,\dots,\rho\nu_{\rho}^{b}, 
\rho'\nu_{\rho'}^{a'},\dots,\rho'\nu_{\rho'}^{b'})
\end{equation*}
une sous-suite qui soit un segment de longueur strictement 
supérieure à $n(\Delta)$ et $n(\Delta')$.
\item
On dit que $\Delta$ et $\Delta'$ sont \textit{liés} si $\Delta$ 
précède $\Delta'$ ou si $\Delta'$ précède $\Delta$. 
\end{enumerate}
\end{defi}

Ces conditions sont traduites en termes combinatoires en \cite[Lemme 3.4 et 
Corollaire 3.6]{MS2} 

\begin{rema}
\label{dosliados}
\label{dosliados2}
Soient $\Delta$ et $\Delta'$ des segments.  
Les propriétés sui\-van\-tes découlent di\-recte\-ment des définitions~:
\begin{enumerate}
\item 
On suppose que $\Delta$ et $\Delta'$ ne sont pas liés, 
que $n(\Delta)\>n(\Delta')$ et que $b(\Delta)$ et 
$b(\Delta')$ ne sont pas isomorphes.
Alors $b(\Delta)$ n'apparaît pas dans $\Delta'$. 
Cette propriété sera utilisée dans la preuve 
du théorème \ref{nuevo2}. 
\item 
On suppose que $\Delta$ et $\Delta'$ ne sont pas liés 
et que $n(\Delta)\>n(\Delta')$. 
Alors $\Delta$ et $^-\Delta'$ ne sont pas liés. 
Cette propriété sera utilisée dans la preuve de la proposition 
\ref{2seg}. 
\end{enumerate}
\end{rema}

\subsection{Représentations associées à un segment} 
\label{PropZLS}
\label{GuyDomville}

Soit $\rho$ une représentation cuspidale de $\G_m$, 
et soit $\Delta= \left[a,b\right]_\rho$ un segment. 
On pose~:
\begin{equation}
\Pi(\Delta)=\rho\nu_\rho^a\times\dots\times\rho\nu_\rho^b
\end{equation}
et $n=b-a+1$.
On renvoie aux notations du paragraphe \ref{BenBen}.  
On pose $\G=\G_{mn}$ et on note 
$\Hh$ l'algèbre de Hecke-Iwahori $\Hh(n,\qr)$. 
D'après la proposition \ref{BijXin}, on a une bijection 
$\boldsymbol{\xi}_{\rho,n}$ entre 
$\Irr(\Om_{\rho,n})^\q$ et l'ensemble des 
classes d'isomorphisme de $\Hh$-modules simples. 
Compte tenu de la définition du $\Hh$-module $\Ii(a,b)$ en 
\eqref{DefIab} 
et de \cite[Corollaire 4.38]{MS11}, elle induit des bijections~:
\begin{enumerate}
\item 
entre sous-représentations irréductibles de $\Pi(\Delta)$ et sous-modules 
simples de $\Ii(a,b)$~;
\item
entre quotients irréductibles de $\Pi(\Delta)$ et quotients simples de $\Ii(a,b)$. 
\end{enumerate}
De ceci on tire la définition suivante, 
qui associe au segment $\Delta$ deux représentations irréductibles $\Z(\Delta)$ 
et $\L(\Delta)$ de $\G$ (voir la définition \ref{DefZLcarH} pour les notations).

\begin{defi} 
\label{segm}
\begin{enumerate}
\item 
La représentation $\Pi(\Delta)$ possède une unique 
sous-re\-pré\-sen\-ta\-tion ir\-ré\-ductible, 
notée $\Z(\Delta)$, telle que $\boldsymbol{\xi}_{\rho,n}(\Z(\Delta))$ soit 
le caractère $\Zz(a,b)$.
\item 
$\Pi(\Delta)$ possède un unique 
quotient ir\-ré\-ductible, noté $\L(\Delta)$, 
tel que $\boldsymbol{\xi}_{\rho,n}(\L(\Delta))$ soit le carac\-tère $\Ll(a,b)$.
\end{enumerate}
\end{defi}

\begin{rema}
\label{rema44}
\begin{enumerate}
\item 
Les représentations 
$\Z(\Delta)$ et $\L(\Delta)$ peuvent être de multipli\-cité 
$\>2$ comme sous-quotients de $\Pi(\Delta)$.
\item
Les représentations $\Z(\Delta)$ et $\L(\Delta)$ sont cuspidales si et 
seulement si $n(\Delta)=1$. 
\item
La proposition \ref{DefAltRec} donne, dans le cas où 
$\e(\rho) \neq1$, une définition de $\Z(\Delta)$ et de $\L(\Delta)$ 
n'utilisant pas la théorie des types, \ie n'utilisant pas $\boldsymbol{\xi}_{\rho,n}$.
\end{enumerate}
\end{rema}

\begin{rema}
\label{rematonta}
On suppose que $\CR$ est le corps des nombres complexes.  
\begin{enumerate}
\item 
La représentation $\Z(\Delta)$ est la représentation 
notée de la même façon dans \cite{Tadic}. 
Si en outre $\D=\F$, c'est la représentation notée 
$\left<\Delta\right>$ dans \cite{Ze2} et \cite{Vig1}. 
\item 
La représentation $\L(\Delta)$ est la représentation 
notée de la même façon dans \cite{Tadic}. 
Si en outre $\D=\F$, c'est la représentation notée 
$\left<\Delta\right>^t$ dans \cite{Ze2}. 
\end{enumerate}
\end{rema}

Le résultat suivant sera utile au paragraphe \ref{PropProdNonLies}. 

\begin{prop}
\label{Z=L}
Les représentations $\Z(\Delta)$ et $\L(\Delta)$ sont isomorphes si et
seulement si $\CR$ est de caractéristique non nulle $\ell$ et si 
$q(\rho)$ et $(-1)^n$ sont congrus à $-1$ modulo $\ell$,
\ie que, ou bien $q(\rho)$ est congru à $-1$ modulo $\ell$ et $n$ est impair, 
ou bien $\ell=2$. 
\end{prop}

\begin{proof}
Ces représentations sont isomorphes si et
seulement si les carac\-tères $\Zz(a,b)$ et $\Ll(a,b)$ sont égaux. 
Le résultat est une conséquence de la définition \ref{DefZLcarH}. 
\end{proof}

\begin{rema}
À noter que, dans ce cas, les représentations $\Z(\Delta)$ et $\L(\Delta)$ sont 
iso\-mor\-phes mais pas forcément égales comme sous-quotients de $\Pi(\Delta)$.  
\end{rema}

Le résultat suivant est une conséquence de la définition de $\Z(\Delta)$.
Il sera utile au paragraphe \ref{PropProdNonLies} et dans la 
section \ref{Sec10}.

\begin{prop}
\label{ZZp}
Soit ${\bf\Phi}_{\rho,n}$ la bijection définie par \eqref{BeautyCeleste}.  
Alors on a~: 
\begin{equation*}
{\bf\Phi}_{\rho,n}(\Z([a,b]_{\rho}))=\Z([a,b]_{1_{\F'^{\times}}}). 
\end{equation*}
\end{prop}

On va maintenant calculer les modules de Jacquet des représentations 
$\Z(\Delta)$ et $\L(\Delta)$.
Pour ça, on utilise les notations des paragraphes \ref{PrelimFini4} et 
\ref{ClasJam}, auxquels on renvoie le lecteur. 
On rappelle que $\Irr(\Om_{\rho,n})$ est l'ensemble des classes de 
représentations irréductibles qui sont sous-quotients irréductibles 
d'une induite de la forme 
$\rho\chi_1\times\dots\times\rho\chi_n$, 
où $\chi_1,\dots,\chi_n$ sont des caractères non ramifiés de $\G_m$. 
Compte tenu de \eqref{Phoebe}, on introduit la définition suivante. 

\begin{defi}
Pour tout $\pi\in\Irr(\Om_{\rho,n})$, on note 
$\KMS_{n}(\pi)$ la plus grande sous-représentation de $\sy{\KM_{n}(\pi)}$ 
(dans le groupe de Grothendieck de $\GL_{m'n}(\kk)$)
contenue dans $\sy{\s\times\dots\times\s}$.
\end{defi}

\begin{rema}
Si $\D=\F$, alors $\KMS_{n}(\pi)$ est simplement la semisimplifiée de 
$\KM_{n}(\pi)$.
\end{rema}

Ceci définit par linéarité un morphisme de $\ZZ$-modules~:
\begin{equation}
\KMS_{n}:\ZZ(\Irr(\Om_{\rho,n}))\to\Gg(\GL_{m'n}(\kk))\subseteq\overline{\Gg}. 
\end{equation}
Si $\a=(n_1,\dots,n_r)$ est une famille d'entiers $\>1$ 
de somme $n$, on définit $\KMS_{\a}$ de façon analogue à partir de $\KM_\a$. 
On a les propriétés suivantes. 

\begin{prop}
\label{KMAX1}
\label{CalculKMSI}
\begin{enumerate}
\item
Si $\chi_1,\dots,\chi_n$ sont des caractères non ramifiés de $\G$, 
alors on a~:
\begin{equation}
\label{PhoebeS}
\KMS_{n}(\rho\chi_1\times\dots\times\rho\chi_n)=\s\times\dots\times\s
\end{equation}
dans $\overline{\Gg}$.  
\item
Si $\pi_1,\dots,\pi_r$ sont des représentations de 
$\Irr(\Om_{\rho,n_1}),\dots,\Irr(\Om_{\rho,n_r})$ 
respectivement, on a~:
\begin{equation*}
\KMS_n(\pi_1\times\dots\times\pi_r)=
\KMS_{n_1}(\pi_1)\times\dots\times\KMS_{n_r}(\pi_r)
\end{equation*}
dans $\overline{\Gg}$.  
\item
Si $\pi$ est dans $\Irr(\Om_{\rho,n})$, on a~: 
\begin{equation*}
\KMS_{n}(\pi)^{\NB_{(m'n_1,\dots,m'n_r)}}=\KMS_\a(\rp_{(mn_1,\dots,mn_r)} (\pi))
\end{equation*}
dans $\Gg(\MB_{(m'n_1,\dots,m'n_r)})$.  
\end{enumerate}
\end{prop}

On commence par traiter le cas d'un sous-groupe parabolique minimal. 

\begin{lemm}
\label{Potter}
On a~:
\begin{align}
\label{eq:1}
&\rp_{(m,\dots,m)}\(\Z([a,b ]_\rho)\) 
=\rho\nu_\rho^a\otimes\dots\otimes\rho\nu_\rho^b,\\
&\rp_{(m,\dots,m)}\(\L([a,b ]_\rho)\)
=\rho\nu_\rho^b\otimes\dots\otimes\rho\nu_\rho^a,\\
&\rp^{-}_{(m,\dots,m)}\(\Z([a,b ]_\rho) \)
=\rho\nu_\rho^b\otimes\dots\otimes\rho\nu_\rho^a,\\
&\rp^{-}_{(m,\dots,m)}\(\L([a,b ]_\rho)\)
=\rho\nu_\rho^a\otimes\dots\otimes\rho\nu_\rho^b.
\end{align}
\end{lemm}

\begin{proof}
Nous prouvons la première assertion~; la seconde se prouve de 
façon analogue et les deux dernières se déduisent des deux premières 
par conjugaison. 
On pose $\Z=\Z([a,b ]_\rho)$.
Par réciprocité de Frobe\-nius, on a un morphisme surjectif~:
\begin{equation*}
\rp_{(m,\dots,m)}(\Z)\to 
\rho\nu_\rho^a\otimes\dots\otimes\rho\nu_\rho^b.
\end{equation*}
D'après le lemme géométrique et la définition de $\Z$, 
le membre de gauche est composé de sous-quotients irréductibles de la forme 
$\rho\chi_1\otimes\dots\otimes\rho\chi_n$ où $\chi_1\dots,\chi_n$ sont des 
caractères non ramifiés de $\G_m$. 
D'après la proposition \ref{DeliceDOrient} par exemple, 
aucun de ces sous-quotients n'est an\-nu\-lé par le foncteur 
$\KM_{(1,\dots,1)}$. 
D'après la proposition \ref{ResumePropKM}, on a~:
\begin{equation}
\label{Mandella}
\KM_{(1,\dots,1)}(\rp_{(m,\dots,m)}(\Z))\simeq
\KM_{n}(\Z)^{\NB_{(m',\dots,m')}}.
\end{equation}
D'après \eqref{Phoebe}, $\KM_{n}(\Z)$ est une sous-repré\-sen\-tation de~: 
\begin{equation}
\label{Decopls}
\KM_{n}(\rho\nu_{\rho}^{a}\times\dots\times\rho\nu_{\rho}^{b})
\simeq\bigoplus\limits_{(i_1,\dots,i_n)}
\s^{\Fr^{i_1}}\times\dots\times\s^{\Fr^{i_n}},
\end{equation}
où $i_1,\dots,i_n$ varient entre $1$ et $b(\rho)$.
Le membre de droite de \eqref{Mandella} est donc une 
sous-re\-pré\-sen\-tation de la somme directe finie~: 
\begin{equation*}
n!\cdot\bigoplus\limits_{(i_1,\dots,i_n)} 
\s^{\Fr^{i_1}}\otimes\dots\otimes\s^{\Fr^{i_n}},
\end{equation*}
où $n!$ est une multiplicité. 
On en déduit que~:
\begin{equation}
\label{MandellaS}
\KMS_{(1,\dots,1)}(\rp_{(m,\dots,m)}\(\Z\))=
\KMS_{n}(\Z)^{\NB_{(m',\dots,m')}}
\end{equation}
est une somme directe finie de copies de $\s\otimes\dots\otimes\s$.
On utilise maintenant la $\CR$-algèbre $\Hh(\s,n)$ définie 
au paragraphe \ref{DefDSM}. 
D'après \cite[Proposition 5.17]{MS11}, le $\Hh(\s,n)$-module~:
\begin{equation*}
\Hom_{\GB}(\s\times\dots\times\s,\KM_{n}(\Z))
\simeq 
\Hom_{\MB_{(m',\dots,m')}}(\s\otimes\dots\otimes\s,\KM_{n}(\Z)^{\NB_{(m',\dots,m')}})
\end{equation*}
est de dimension $1$. 
On en déduit que la représentation 
\eqref{MandellaS} est isomorphe à $\s\otimes\dots\otimes\s$. 
Comme 
\eqref{MandellaS} et $\rp_{(m,\dots,m)}\(\Z\)$ ont la même longueur, 
cette dernière est irréductible, donc isomorphe à 
$\rho\nu_\rho^a\otimes\dots\otimes\rho\nu_\rho^b$. 
\end{proof}

\begin{rema}
\label{Cinderella}
On en déduit que $\Z(\Delta)$ est aussi un quotient de 
$\rho\nu_\rho^b\times\dots\times\rho\nu_\rho^a$ et que 
$\L(\Delta)$ est aussi une sous-représentation 
de $\rho\nu_\rho^b\times\dots\times\rho\nu_\rho^a$. 
\end{rema}

\begin{prop}
\label{prpiseg}
On a les propriétés suivantes.
\begin{enumerate}
\item 
Si $k$ est un entier tel que $a<k\<b$, alors~:
\begin{eqnarray*}
\rp_{((k-a)m,(b-k+1)m)}(\Z([a,b]_\rho) )&=& 
\Z([a,k-1 ]_\rho) \otimes\Z([k,b]_\rho), \\
\rp_{((b-k+1)m,(k-a)m)}(\L([a,b ]_\rho) )&=& 
\L([k,b ]_\rho) \otimes\L([a,k-1]_\rho).
\end{eqnarray*}
\item 
Si $k$ est un entier tel que $a<k\<b$, alors~:
\begin{eqnarray*}
\rp^{-}_{((b-k+1)m,(k-a)m)}(\Z([a,b ]_\rho) )&=& 
 \Z([k,b ]_\rho) \otimes \Z([a,k-1 ]_\rho), \\
\rp^{-}_{((k-a)m,(b-k+1)m)}(\L([a,b ]_\rho) )&=& 
\L([a,k-1 ]_\rho) \otimes \L([k,b ]_\rho). 
\end{eqnarray*}
\end{enumerate}
\end{prop}

\begin{proof}
D'après \cite[Propositions 4.5 et 4.20]{MS11}, la bijection 
$\boldsymbol{\xi}_{\rho,n}$ induit une bijection entre sous-représentations 
irréductibles de $\Z([a,k-1 ]_\rho)\times\Z([k,b]_\rho)$ et 
sous-modules irréductibles de~:
\begin{equation}
\label{Iciparhasard}
\Hom_{\Hh_{((k-a)m,(b-k+1)m)}}(\Hh,\Zz(a,k-1)\otimes\Zz(k,b))
\end{equation}
avec les notations du paragraphe \ref{AHI}. 
Comme $\Zz(a,b)$ est un sous-module de \eqref{Iciparhasard}, 
on en déduit que $\Z([a,b]_\rho)$ est une sous-représentation de 
$\Z([a,k-1 ]_\rho)\times\Z([k,b]_\rho)$.
Par adjonction, on en déduit un morphisme surjectif~:
\begin{equation*}
\rp_{((k-a)m,(b-k+1)m)}(\Z([a,b]_\rho))\to\Z\([a,k-1 ]_\rho\)\otimes\Z\([k,b ]_\rho\).
\end{equation*}
Posons $\rp=\rp_{((k-a)m,(b-k+1)m)}$ pour alléger les notations et 
supposons que le membre de gauche ci-dessus ne 
soit pas irréductible.  
Il contient un autre sous-quotient irréductible $\Y$, 
qu'on peut supposer être une sous-représentation ou un quotient.  
Suppo\-sons que $\Y$ soit un quotient, l'autre cas se traitant 
de façon analogue.  
On fixe une paire cuspidale $(\M',\vr')$ du sous-groupe de Levi 
$\M=\M_{((k-a)m,(b-k+1)m)}$ 
et un sous-groupe parabolique $\P'$ de $\M$ de facteur de Levi $\M'$ 
tels que $\Y$ soit une sous-représentation de $\ip^\M_{\P'}(\vr')$.
On suppose que $\M'$ est standard.
Alors~:
\begin{equation*}
\Hom_{\M}(\rp(\Z[a,b]_\rho),\ip^\M_{\P'}(\vr'))\neq0.
\end{equation*}
On en déduit que le support cuspidal de $\Z[a,b]_\rho$ est la classe 
de $\G$-conjugaison de $(\M',\vr')$.
Par conséquent, on a $\M'=\M_{(m,\dots,m)}$ et $\vr'$ est $\G$-conjuguée à 
$\rho\nu_\rho^a\otimes\dots\otimes\rho\nu_\rho^b$, ce qui 
contre\-dit le fait que, d'après le lemme \ref{Potter}, 
on a $\rp_{(m,\dots,m)}(\Y)=0$. 
\end{proof}

Dans le cas où $\e(\rho)\neq1$, \ie quand $\qr$ n'est pas congru à $1$ 
modulo $\ell$, les représen\-tations $\Z(\Delta)$ et 
$\L(\Delta)$ peuvent être définies par récurrence sur la longueur de 
$\Delta$. 

\begin{prop}
\label{DefAltRec}
On suppose que $\e(\rho)\neq1$ et que $b-a\>1$. 
\begin{enumerate}
\item 
$\Z([a,b]_\rho)$ est l'unique sous-représentation irréductible de 
$\Z([a,b-1]_\rho)\times\rho\nu_\rho^b$.
\item
$\Z([a,b]_\rho)$ est l'unique quotient irréductible de 
$\Z([a+1,b]_\rho)\times\rho\nu_\rho^a$.
\item
$\L([a,b]_\rho)$ est l'unique quotient irréductible de 
$\L([a,b-1]_\rho)\times\rho\nu_\rho^b$.
\item
$\L([a,b]_\rho)$ est l'unique sous-représentation irréductible de 
$\L([a+1,b]_\rho)\times\rho\nu_\rho^a$.
\end{enumerate}
\end{prop}

\begin{proof}
On prouve la première assertion,
les autres se prouvant de façon ana\-lo\-gue.  
On pose $n=b-a+1$ et $\rp=\rp_{((n-1)m,m)}$. 
Compte tenu de la proposition \ref{prpiseg}, on a~:
\begin{equation*}
\sy{\rp(\Z([a,b-1]_\rho)\times\rho\nu_\rho^b)}=
\sy{\Z([a,b-1]_\rho)\otimes\rho\nu_\rho^b}+
\sy{(\Z([a,b-2]_\rho)\times\rho\nu_\rho^b)\otimes\rho\nu_\rho^{b-1}}.
\end{equation*}
Si $n=2$, on interprète $\Z([a,b-2]_\rho)$ comme la représentation (triviale) 
du groupe trivial $\G_0$. 
Comme $\e(\rho)\neq1$, les représentations $\rho\nu_\rho^b$ et 
$\rho\nu_\rho^{b-1}$ ne sont pas isomorphes~: 
le sous-quotient irréductible $\Z([a,b-1]_\rho)\otimes\rho\nu_\rho^b$ est 
donc de multiplicité $1$ dans le membre de gauche. 
Le résultat se déduit alors du lemme \ref{mult1}.
\end{proof}

\begin{rema}
Cette proposition permet de définir, dans le cas où $\e(\rho)\neq1$, 
les re\-pré\-sen\-tations $\Z([a,b]_\rho)$ et $\L([a,b]_\rho)$ par récurrence, sans 
faire recours à la théorie de types. 
\end{rema}

La proposition suivante décrit le comportement de 
$\Delta\mapsto\Z(\Delta)$ et $\Delta\mapsto\L(\Delta)$ 
par pas\-sa\-ge à la contra\-grédiente 
(voir \eqref{contragr} pour la définition de la notation 
$\Delta^\vee$).

\begin{prop}
\label{ZLDdual}
On a $\Z(\Delta^{\vee})\simeq\Z(\Delta)^{\vee}$ et 
$\L(\Delta^{\vee})\simeq\L(\Delta)^{\vee}$.
\end{prop}

\begin{proof}
Nous proposons deux preuves de ce résultat, la première ne fonctionnant que 
dans le cas où $\e(\rho)\neq1$. 

On suppose d'abord que $\e(\rho)\neq1$ et on prouve le résultat par récurrence 
sur $n=b-a+1$, le cas $n=1$ étant immédiat. 
On suppose que $n\>2$ et on utilise les no\-ta\-tions \eqref{DefMoins} et \eqref{DefPlus}.  
D'après la proposition \ref{DefAltRec}, l'induite 
$\Z(\Delta^-)\times\rho\nu_\rho^b$ admet $\Z(\Delta)$ pour 
sous-représentation irréductible, donc 
$\Z(\Delta^-)^\vee\times\rho^\vee\nu_\rho^{-b}$ admet $\Z(\Delta)^\vee$ 
pour quotient irréductible. 
Par hypothèse de récurrence, $\Z(\Delta^-)^\vee$ 
et $\Z((\Delta^-)^{\vee})$ sont isomorphes. 
Comme $(\Delta^-)^{\vee}$ est égal à $^-(\Delta^\vee)$ 
et d'après la proposition \ref{DefAltRec}(2), 
le quotient irréductible $\Z(\Delta)^\vee$ est isomorphe à 
$\Z(\Delta^\vee)$.
Le cas de $\L(\Delta)$ se traite de façon analogue. 

Voici maintenant une preuve fonctionnant sans hypothèse sur $\e(\rho)$,
s'appuyant sur \cite{HS}.
Il s'agit de prouver que la représentation $\Z(\Delta)^\vee$ possède les deux 
propriétés de la définition \ref{segm} qui caractérisent $\Z(\Delta^\vee)$. 
D'abord, d'après la remarque \ref{Cinderella} et par passage à la contragrédiente, 
$\Z(\Delta)^\vee$ est une sous-représentation irréductible de 
$\Pi(\Delta^\vee)$. 
Ensuite, si $(\J,\l)$ est le type simple maxi\-mal contenu dans $\rho$ fixé au 
paragraphe \ref{BenBen}, alors le type simple maximal 
$(\J,\l^\vee)$ est contenu dans la contragré\-diente $\rho^\vee$ 
(voir \cite[Théorème 4.1]{HS} appliqué au type simple $(\J,\l)$).
On a ainsi une bijection $\boldsymbol{\xi}_{\rho^\vee,n}$ entre 
$\Irr(\Om_{\rho^\vee,n})^\q$ et l'ensemble des classes de $\Hh$-modules 
simples, et il faut vérifier que $\Z(\Delta)^\vee$ correspond au caractère 
$\Zz(-b,-a)$ par cette bijection.
D'après \cite[4.4]{MS11}, il y a un type simple $(\K,\W)$ de $\G$ tel que~:
\begin{enumerate}
\item
l'induite compacte $\ind_\K^\G(\W)$ est isomorphe à 
$\ind^{\G_m}_\J(\l)\times\dots\times\ind^{\G_m}_\J(\l)$ 
($n$ fois)~; 
\item
pour toute représentation $\V$ dans $\Irr(\Om_{\rho,n})^\q$,
son image $\boldsymbol{\xi}_{\rho,n}(\V)$ est $\Hom_{\K}(\W,\V)$ 
considéré comme un $\Hh$-module grâce à la proposition \ref{Meleagant}~; 
\item
de façon analogue, 
si $\V$ est une représentation dans $\Irr(\Om_{\rho^\vee,n})^\q$, alors 
son image $\boldsymbol{\xi}_{\rho^\vee,n}(\V)$ est $\Hom_{\K}(\W^\vee,\V)$ 
considéré comme un $\Hh$-module.  
\end{enumerate}
La représentation $\Z(\Delta)^\vee$ est dans $\Irr(\Om_{\rho^\vee,n})^\q$.
D'après (3) ci-dessus et d'après \cite[Théorème 4.1]{HS} appliqué au type simple 
$(\K,\W)$, le $\Hh$-module $\boldsymbol{\xi}_{\rho^\vee,n}(\Z(\Delta)^\vee)$ 
est le dual de $\boldsymbol{\xi}_{\rho,n}(\Z(\Delta))=\Zz(a,b)$, 
\ie $\Zz(-b,-a)$ comme voulu. 
On en déduit donc que $\Z(\Delta^{\vee})\simeq\Z(\Delta)^{\vee}$. 
La preuve pour $\L(\Delta^\vee)$ est similaire.
\end{proof}

\begin{rema}
À noter que $\Pi(\Delta)^\vee$ n'est pas isomorphe à $\Pi(\Delta^{\vee})$ en 
général. 
\end{rema}

Pour finir ce paragraphe, on prouve le résultat suivant
(voir \eqref{deflsn} pour la définition de $l(\s,n)$). 

\begin{prop}
\label{prpisegK}
Soit $\Delta=[a,b]_\rho$ un segment de longueur $n$.
\begin{enumerate}
\item 
On a $\KMS_{n}(\Z([a,b]_\rho))=\z(\s,n)$.
\item
On a $\KMS_{n}(\L([a,b]_\rho))=l(\s,n)$. 
\item
\label{prpisegK2}
On a $\KMS_{n}(\L([a,b]_\rho))=\st(\s,n)$ si et seulement si $n<\ee(\rho)$.
\end{enumerate}
\end{prop}

\begin{proof}
On procède par récurrence sur $n$. 
On pose $\Z=\Z([a,b]_\rho)$. 
On sait que $\z(\s,n)$ est un sous-quotient de $\KMS_{n}(\Z)$. 
Supposons que $\KMS_{n}(\Z)$ n'est pas irréductible~: 
elle contient donc un autre facteur irréductible, noté $v$. 
Soit un entier $k\in\{1,\dots,n-1\}$, 
soit $\a=(km,(n-k)m)$ et posons 
$\NB=\NB_\a$ et $\MB=\MB_\a$.  
D'après la pro\-po\-sition \ref{prpiseg}, on a~:
\begin{equation*}
\rp_\a(\Z)=\Z([a,c]_\rho)\otimes\Z([c+1,b]),
\end{equation*}
avec $c=a+k-1$. 
Par hypothèse de récurrence, on a~:
\begin{equation*}
\KMS_{(k,n-k)}(\rp_{\a} (\Z))=\z(\s,k)\otimes\z(\s,n-k),
\end{equation*}
qui est aussi égal à $\KMS_{n}(\Z)^{\NB}$ d'après la proposition \ref{KMAX1}. 

\begin{lemm}
\label{OnRevientDeLoin}
On a $\sy{\z(\s,n)^{\NB}}\>\sy{\z(\s,k)\otimes\z(\s,n-k)}$. 
\end{lemm}

\begin{proof}
Notons $\s^{\times n}$ l'induite $\s\times\dots\times\s$ où $\s$ 
apparaît $n$ fois. 
Par définition de la représentation $\z(\s,n)$, on a~:
\begin{equation*}
\Hom_\GB(\s^{\times n},\z(\s,n))\simeq
\Hom_{\MB}(\s^{\times k}\otimes\s^{\times (n-k)},\z(\s,n)^{\NB})
\end{equation*}
qui est 
de dimension $1$, isomorphe au caractère trivial de 
$\Hh(\s,k)\otimes\Hh(\s,n-k)$ considéré comme sous-algèbre de $\Hh(\s,n)$. 
\end{proof}

Comme $\KMS_n(\Z)^{\NB}$ contient $\sy{\z(\s,n)^{\NB}}$ et $\sy{v^{\NB}}$, 
on déduit du lemme \ref{OnRevientDeLoin} d'une part que~:
\begin{equation*}
\z(\s,n)^{\NB}=\z(\s,k)\otimes\z(\s,n-k)
\end{equation*}
et d'autre part que $v^{\NB}=0$, et ce pour tout $k$. 
On en déduit que $v$ est cuspidale, et donc de la forme $\st(\s,n)$ avec 
$n=1$ ou $n=\ee(\s)\ell^r$.  
Par conséquent, $\Z=\St(\rho,n)$ d'après la définition \ref{DefSTRN}. 
Mais pour ces valeurs de $n$, la représentation $\St(\rho,n)$ est cuspidale 
d'après la proposition \ref{StCusp}. 
Comme $\Z$ ne peut pas être cuspidale à moins que $n=1$, 
on obtient une contradiction. 

L'assertion (2) se démontre de façon analogue. 
Pour prouver l'assertion (3), il suffit de prouver que $\st(\s,n)$ correspond 
au caractère signe de $\Hh(\s,n)$ si et seulement si on a $n<\ee(\s)$, 
ce qui est donné par la proposition \ref{prpisegKfini}. 
\end{proof}

\begin{rema}
\label{L=St0}
On renvoie à la remarque \ref{L=St} pour une précision supplémentaire. 
\end{rema}
 
\subsection{Critère d'irréductibilité pour un produit de segments non liés} 
\label{PropProdNonLies}

Le but de cette section est de montrer le théorème suivant.

\begin{theo}
\label{nuevo2}
Soit un entier $r\>1$ et soient $\Delta_1,\dots,\Delta_r$ des segments. 
Les conditions suivantes sont équi\-valentes~:
\begin{enumerate}
\item Pour tous $i,j\in\{1,\dots,r\}$ avec $i\neq j$, les segments $\Delta_i$ 
et $\Delta_j$ sont non liés.
\item La représentation $\Z(\Delta_1)\times \dots \times \Z( \Delta_r)$ est irréductible.
\item La représentation $\L(\Delta_1)\times \dots \times \L( \Delta_r)$ est irréductible.
\end{enumerate}
\end{theo}

Ce théorème généralise \cite[4]{Ze2} et 
\cite[Lemmas 2.5, 4.2]{Tadic} au cas modulaire.
Il justifie le bien-fondé de notre définition \ref{Rotis}. 

\subsubsection{} 

Étant donné un segment $\Delta=[a,b]_{\rho}$, on note~:
\begin{equation}
\label{paradel}
\langle\Delta\rangle=\left<a,b \right>_\rho
\end{equation}
l'une des deux représentations $\Z([a,b]_{\rho})$ ou $\L([-b,-a]_{\rho})$,
et on pose~:
\begin{equation}
\label{paradelmu}
\mu_\rho=
\left\{
\begin{array}{ll}
\nu_\rho & \text{ si } \langle\Delta\rangle=\Z(\Delta)~;\\
\nu_\rho^{-1} & \text{ si } \langle\Delta\rangle=\L([-b,-a]_{\rho}).\\
\end{array}
\right.
\end{equation}
La proposition suivante, qui synthétise les résultats de la proposition 
\ref{prpiseg}, montre l'utilité des notations \eqref{paradel} et  \eqref{paradelmu}.

\begin{prop}
\label{ayuto}
Soient $\rho$ une représentation irréductible cuspidale de $\G_m$ 
et $\Delta= \left[a,b \right]_\rho$. 
\begin{enumerate}
\item 
Si $k$ est un entier tel que $a<k\<b$, on a~:
\begin{equation*}
\rp_{(k-a)m,(b-k+1)m}\(\langle a,b\rangle_\rho \)= 
\langle a,k-1\rangle_\rho \otimes \langle k,b\rangle_\rho.
\end{equation*}
\item 
Si $k$ est un entier tel que $a<k\<b$, on a~: 
\begin{equation*}
{\rp}^{-}_{(b-k+1)m,(k-a)m}\(\langle a,b\rangle_\rho \)= 
\langle k,b\rangle_\rho \otimes\langle a,k-1\rangle_\rho.
\end{equation*}
\item 
On a~:
\begin{equation*}
\rp_{(m,\dots,m)}\(\langle\Delta \rangle \)=
\rho\mu_{\rho}^a \otimes \rho \mu_{\rho}^{a+1} 
\otimes \dots \otimes \rho\mu_{\rho}^{b}
\end{equation*}
et~:
\begin{equation*}
\rp^{-}_{(m,\dots,m)}\(\langle\Delta\rangle \)=
 \rho\mu_{\rho}^b\otimes \rho \mu_{\rho}^{b-1} 
\otimes \dots \otimes \rho \mu_{\rho}^{a}.
\end{equation*}
\item 
On a 
$\langle\Delta^{\vee}\rangle=\langle\Delta\rangle^{\vee}$,
\ie que
$\langle-b,-a\rangle_{\widetild{\rho}}=\widetild{\langle a,b\rangle_\rho}$.
\end{enumerate} 
\end{prop}

Pour montrer le théorème \ref{nuevo2}, 
il suffira donc de montrer le théorème suivant. 

\begin{theo}
\label{nuevo3}
Soit un entier $r\>1$ et soient $\Delta_1,\dots,\Delta_r$ des segments. 
Alors~:
\begin{equation*}
\left< \Delta_1\right>\times \dots \times \left< \Delta_r\right>
\end{equation*} 
est irréductible si et seulement si $\Delta_i$ 
et $\Delta_j$ sont non liés pour tous $i,j\in\{1,\dots,r\}$ avec $i\neq j$. 
\end{theo}

On montre d'abord le résultat suivant. 

\begin{prop}
\label{2seg}
Soient $\Delta$ et $\Delta'$ des segments non liés. 
Alors 
$\left<\Delta\right>\times\left<\Delta'\right>$ est irréductible. 
\end{prop}

\begin{rema}
\label{2sep}
D'après la proposition \ref{commu}, cette proposition implique que, 
si $\Delta$ et $\Delta'$ sont deux segments 
non liés, alors 
$\left< \Delta \right> \times \left< \Delta'\right>$ et
$\left<\Delta' \right> \times \left< \Delta \right>$ sont 
isomorphes.
\end{rema}

On écrit $\Delta$ et $\Delta'$ respectivement sous la forme $\left[a,b\right]_\rho$ et 
$\left[a',b'\right]_{\rho'}$, 
où $\rho$ et $\rho'$ sont irré\-duc\-ti\-bles cus\-pi\-dales. 
On note $n= b-a+1$ et $n'=b'-a'+1$.
On fait d'abord quelques réductions~: 
elles ne sont pas nécessaires dans notre preuve
mais permettent de simplifier les notations. 
\begin{enumerate}
\item[{\bf R1}]
D'après la proposition \ref{commu}, on peut supposer que $n \< n'$.  
\item[{\bf R2}]
D'après la proposition \ref{Zdroites}, et quitte à modifier $a$,
$b$, $a'$ et $b'$, on peut supposer que $\rho$ et $\rho'$ sont isomorphes. 
\item[{\bf R3}]
Quitte à tordre $\rho$ par un caractère non ramifié, on peut supposer que 
$a=0$.  
\item[{\bf R4}]
Par la méthode du changement de groupe (voir la proposition \ref{ZZp} 
et \cite[4.4]{MS11}), 
on peut supposer que $\rho$ est le caractère trivial de $\mult\F$, que l'on 
notera $\left<0\right>$. 
\end{enumerate}

Aussi peut-on supposer que 
$\Delta=\left[0,b\right]$ et $\Delta'=\left[a',b'\right]$ avec $b+a'\<b'$. 

\begin{rema}
Si $\CR$ est de caractéristique non nulle $\ell$, 
on note $e$ l'ordre de $q$ dans $\FF_\ell^{\times}$.  
Puisque $\Delta$ et $\Delta'$ ne sont pas liés et que $n\<n'$, 
on a $n\<e-1$, ce qui implique qu'on a $e\>2$.  
Si $\CR$ est de caractéristique nulle, on convient que $e$ est infini et 
une congruence de la forme $a \equiv a' \mod e$ voudra 
juste dire que $a=a'$. 
\end{rema}

Nous utilisons la méthode des foncteurs de Jacquet et le lemme \ref{sir}.  
Malheureusement, nous ne pouvons éviter de traiter plusieurs cas en petite 
dimension (où les foncteurs de Jacquet ne sont pas de grande aide).
On distingue les cas suivants~:
\begin{enumerate}
\item[1.]
$n=n'=1$.
\item[2.]
$e \> 3$ et $n=1$.
\item[3.]
$e \> 3$ et $n,n' \> 2$.
\item[4.]
$e=2$ et $n'=3$.
\item[5.]
$e=2$ et $n'\>4$.
\end{enumerate}
On remarque que, dans les deux derniers cas, on a $n=1$. 
Dans le cas où $n=n'=1$, la propo\-sition est une conséquence 
de \cite[Proposition 4.35]{MS11}.  
On passe maintenant aux cas suivants. 

\subsubsection{}
\label{cas2}

On suppose que $e \> 3$ et $n=1$. 
On procède par récurrence sur $n'$, 
le cas $n'=1$ étant déjà prouvé. 
Remarquons que l'hypothèse sur $\Delta$ et $\Delta'$ se traduit par
$a' \not\equiv 1 \mod e$
et
$b' \not\equiv -1 \mod e$. 
\vspace{0.5cm}
Il faut encore différencier les sept cas suivants~:\\
\vspace{0.5cm}
\begin{minipage}{0.5\linewidth}
\begin{enumerate}
\item[2.1.]  $\Delta'= \left[0,1\right]$. 
\item[2.2.] $b' \not\equiv -1,0,1 \mod e$. 
\item[2.3.] $a'=0$, $b'\equiv 0,1 \mod e$ et $n'\> 3$.
\item[2.4.] $a' =-1$ et $b' \equiv 1 \mod e$.  
\end{enumerate}
\end{minipage}
\begin{minipage}{0.5\linewidth}
\begin{enumerate}
\item[$2.1'$.]  $\Delta'= \left[-1,0\right]$. 
\item[$2.2'$.] $a' \not\equiv -1,0,1 \mod e$.
\item[$2.3'$.] $b'=0$, $a'\equiv 0,-1 \mod e$ et $n'\> 3$. 
\item[\quad]
\end{enumerate}
\end{minipage}\\
Les cas $2.1'$, $2.2'$ et $2.3'$ découlent respectivement 
des cas 2.1, 2.2 et 2.3 par passage à la contra\-gré\-diente
(ce que nous pouvons faire puisque $e\neq1$). 
Traitons les cas 2.1, 2.2, 2.3 et 2.4.

\begin{lemm}
L'induite $\langle 0\rangle\times\langle 0,1\rangle$ est irréductible. 
\end{lemm}

\begin{proof}
Comme dans Rogawski \cite{Rog2}, ce cas est particulier et il faut le traiter 
à part.  
On pose $\M=\M_{(1,2)}$ et on note $\chi$ 
le caractère formé des vecteurs de $\langle 0,1\rangle$ qui sont 
invariants par le sous-groupe d'Iwahori de $\GL_2(\F)$
(il s'agit donc de $\Zz(0,1)$ ou de $\Ll(0,1)$ selon les cas).
D'après \cite[Pro\-po\-sition 4.13]{MS11}, l'induite est irréductible si et 
seulement si 
$\Hom_{\Hh_\M}(\Hh,1\otimes\chi)$ est irréductible.
L'irré\-duc\-tibilité de ce $\Hh$-module 
se montre comme dans \cite[Lemma 5.1]{Rog2}.
\end{proof}

Traitons le cas 2.2.
D'après la proposition \ref{ayuto}(1), 
la représentation $\langle 0\rangle \times \langle  a',b'\rangle$ est 
une sous-re\-pré\-sentation de 
$\langle 0\rangle \times \langle a',b'-1\rangle \times\langle b'\rangle$.  
Par hypothèse de récurrence (puisqu'on a $b' \not\equiv 0 \mod e$) et d'après la 
remarque \ref{2sep}, celle-ci est isomorphe à $\pi_1 \times \pi_2$, 
où l'on a posé 
$\pi_1=\langle a',b'-1\rangle$ et $\pi_2=\langle 0\rangle \times \langle b'\rangle $ 
(remarquons que $\pi_2$ est irréductible par hypothèse de récurrence puisque $b' 
\not\equiv \pm 1 \mod e$).  
De même, par la proposition \ref{ayuto}(2), c'est un quotient de $\pi_2 \times 
\pi_1$.  
D'après le lem\-me géo\-mé\-trique, on a~:
\begin{equation*}
\begin{split}
\sy{\rp_{(n'-1,2)}(\pi_1\times \pi_2)} =
\langle a',b'-1\rangle \otimes (\langle0\rangle \times \langle b'\rangle ) &
+
(\langle a',b'-2\rangle\times \langle0\rangle)\otimes (\langle b'-1\rangle 
\times \langle b'\rangle ) \\
& + 
(\langle a',b'-2\rangle\times \langle b'\rangle)\otimes (\langle b'-1\rangle 
\times \langle 0\rangle ) \\
& + 
(\langle a',b'-3\rangle \times
\langle b'\rangle\times \langle 0\rangle) \otimes \langle b'-2,b'-1\rangle
\end{split}
\end{equation*}
et donc la représentation 
$\pi_1 \otimes \pi_2$ apparaît avec multiplicité 1 dans 
$\rp_{(n'-1,2)}(\pi_1\times \pi_2)$.  
On déduit du lemme \ref{sir}
que $\langle 0\rangle \times \langle  a',b' \rangle$ est irréductible. 

\medskip

Le cas 2.3 se montre comme le cas précédent en utilisant 
$\pi_1= \langle 0\rangle\times\langle 0,1\rangle$ et $\pi_2=\langle 2,b\rangle $.  

\medskip

Traitons enfin le cas 2.4. 
D'après le lemme \ref{mult1}, la représentation $\langle 0\rangle 
\times\langle -1,b'\rangle$ 
a une unique sous-re\-pré\-sentation 
irréductible $\pi$ et un unique quotient irréductible $\pi'$.  
De plus, $\pi$ et $\pi'$ apparaissent avec multiplicité $1$ dans 
$\langle 0\rangle \times\langle -1,b'\rangle$. Pour prouver que cette
induite est irréductible, il suffit donc de montrer que $\pi$ et $\pi'$ 
sont isomorphes. 
On pose $\alpha=(1,\dots,1)$.
Le module de Jacquet $\rp_{\a}(\langle 0\rangle \times\langle -1,b'\rangle)$
contient comme sous-quotient la représentation~:
\begin{equation*}
\tau=\langle -1\rangle \otimes\langle 0\rangle \otimes 
\langle 0\rangle
\otimes\langle 1\rangle\otimes\langle 2\rangle\otimes\dots\otimes\langle b'\rangle
\end{equation*}
avec multiplicité $2$.  
Montrons que,  de même, $\rp_\alpha(\pi)$ contient  comme sous-quotient $\tau$
avec multiplicité $2$.  
La représentation $\pi$ étant une
sous-représentation de $\langle 0\rangle \times\langle -1,b'\rangle $, 
elle est aussi, par
la proposition \ref{ayuto}(1), une sous-représentation de 
$\langle 0\rangle \times\langle -1,0\rangle \times \langle 1,b'\rangle$. 
Par réciprocité de Frobenius on trouve donc~: 
$$  \Hom\left(\rp_{(3,b')}( \pi),  ( \langle 0\rangle  \times\langle -1,0\rangle )
  \otimes \langle 1,b'\rangle \right) \neq\{0\}.$$ 
La représentation $ ( \langle 0\rangle \times\langle -1,0\rangle ) \otimes 
\langle 1,b'\rangle$, 
d'après le cas $2.1'$ ci-dessus, est irréductible.  
On a donc~: 
$$\rp_\alpha\left( ( \langle 0\rangle \times\langle -1,0\rangle ) 
\otimes \langle 1,b'\rangle\right) \< \rp_\alpha(\pi).$$ 
D'après le lemme géométrique, $\tau$ apparaît dans
$\rp_\alpha\left( ( \langle 0\rangle \times\langle -1,0\rangle ) 
\otimes \langle 1,b'\rangle\right)$ 
avec multipli\-cité $2$. 
On en déduit donc que $\tau$ apparaît avec multiplicité $2$ dans 
$\rp_\alpha(\pi)$. 
De fa\c{c}on analogue, $\tau$ apparaît avec multiplicité
$2$ dans $\rp_\alpha(\pi')$. 
Puisque $\pi$ et $\pi'$ sont deux sous-quotients de $\langle 0\rangle
\times\langle -1,b'\rangle$ et que $\tau$ apparaît aussi avec multiplicité $2$ 
dans $\rp_\alpha( \langle 0\rangle \times\langle -1,b'\rangle)$, on en déduit 
que $\pi$ et $\pi'$ sont isomorphes. 

\subsubsection{}
\label{cas6}

On suppose que $e \> 3$ et $n,n'\>2$.  
La preuve se fait par récurrence sur $nn'$. 

On commence par traiter le cas particulier où $a'\equiv 0\mod e$. 
D'après la proposition \ref{ayuto}(1), la représentation 
$\left< \Delta \right> \times\left< \Delta' \right>$ se plonge dans 
$\left<0\right> \times \left< ^-\Delta \right> \times \left< \Delta' \right> $. 
Par hypothèse de récurrence et d'après les remarques \ref{dosliados2}(2) et 
\ref{2sep}, cette dernière est isomorphe à l'induite 
$\left<0\right> \times \left< \Delta' \right> \times \left< ^-\Delta \right> $ 
qui est, toujours d'après la proposition \ref{ayuto}(1), une sous
représentation de 
$\left<0\right> \times \left<0\right> \times \left<^- \Delta' \right> \times \left< ^-\Delta 
\right> $.  
De façon analogue, on montre que  
$\left< \Delta \right> \times \left< \Delta '\right> $ est un quotient de
$\left<^- \Delta' \right> \times \left< ^-\Delta \right> 
\times \left<0\right> \times\left<0\right>$. 
Les représentations $\left<0\right> \times \left<0\right>$ et 
$\left<^- \Delta' \right>\times \left< ^-\Delta \right> $ sont irréductibles 
par hypothèse de récurrence.
D'après le lemme géométrique, on a~:
\begin{equation*}
\begin{split}
\sy{\rp_{(2,n+n'-2)}( \left<0\right> \times \left<0\right>\times 
\langle ^-\Delta '\rangle\times \langle ^-\Delta \rangle )} = 
 (\left<0\right> & \times \left<0\right> )\otimes 
(\left<^- \Delta' \right> \times \left< ^-\Delta\right>) \\
& +2\cdot (\left<0\right>\times \left<1\right> )\otimes 
(\left<0\right>\times\left<^{--} \Delta' \right> \times \left< 
  ^-\Delta\right>) \\
& + 2\cdot (\left<0\right>\times \left<1\right> )\otimes 
(\left<0\right>\times\left<^{-}     \Delta'      \right>     \times     \left<
  ^{--}\Delta\right>) \\
& + 
(\left<1\right>\times \left<1\right> )\otimes 
(\left<0\right>\times  \left<0\right> \times  \left<^- \Delta'  \right> \times
\left< ^-\Delta\right>)
\end{split}
\end{equation*}
donc ce module de Jacquet contient 
$(\left<0\right>\times \left<0\right> )\otimes 
(\left<^- \Delta' \right> \times \left< ^-\Delta\right>)$ 
avec multiplicité $1$. 
Donc la représentation 
$\left< \Delta \right> \times \left< \Delta' \right> $ 
est irréductible d'après le lemme \ref{sir}.

On traite maintenant le cas général.  
D'après le paragraphe précédent, on peut supposer, 
quitte à passer à la contragrédiente, que $a' \not \equiv 0 \mod e$. 
Alors $\left< a',b' \right> \times \left< 0,b \right> $ est une 
sous-représentation de~: 
\begin{equation}\label{4321} 
\left< a',b' \right> \times \left< 0\right> \times \left< 1,b \right>. 
\end{equation} 
Comme on a supposé que $n \< n'$, les segments
$\left[ a',b' \right]$ et $\left[0\right]$ ne sont pas liés et donc
\eqref{4321} est isomorphe, par hypothèse de récurrence, à~: 
$$ \left< 0\right> \times \left< a',b' \right> \times \left< 1,b \right>.$$
D'après la remarque \ref{dosliados2}(2), les segments $\left[ a',b' \right]$ et 
$\left[ 1,b \right]$ ne sont pas liés. 
Par hypothèse de récurrence, 
la représentation $ \left< a',b' \right> \times \left< 1,b \right>$ est donc 
irréductible. 
On conclut comme dans le cas 2.2 traité plus haut, 
avec $\pi_1= \left<0\right>$ et 
$\pi_2= \left< a',b' \right> \times \left< 1,b \right>$.

\subsubsection{}
\label{cas3}

On suppose ici que $e=2$ (ce qui implique, étant donné que $n \< n'$, que 
$n=1$). 

\begin{lemm}
\label{TwistPapier}
L'induite $\left<0\right>\times\langle 0,1,2\rangle$ est irréductible. 
\end{lemm}

\begin{proof}
La représentation $\Pi=\left<0,1,2\right> \times \left<0\right>$
se plonge dans $\left<0,1\right>\times\left<0\right> \times \left<0\right>$. 
D'après le lemme géométrique, 
$\left<0,1\right>\otimes(\left<0\right> \times 
\left<0\right>)$ est irréductible et apparaît avec multi\-plicité $1$ dans 
le module de Jacquet $\rp_{(2,2)}(\Pi)$. 
D'après le lemme \ref{sir},
on en déduit que $\Pi$ possède une unique sous-représentation 
irréductible $\pi$.
Par conséquent, $\pi^\vee$ est l'unique quotient irréductible de 
$\Pi^\vee\simeq\Pi$.
Si $\Pi$ n'est pas irréductible, alors $\sy{\Pi}\>\sy{\pi}+\sy{\pi^\vee}$. 
Remarquons que, d'après la proposition \ref{Z=L}, les représentations $\Z([0,1,2])$ et 
$\L([0,1,2])$ sont isomorphes parce que $e=2$ et $n'$ est impair.  

Appliquons le foncteur  $\KM_{4}$, qui n'est rien d'autre  ici que le foncteur
des invariants sous le groupe $1+\p_\F\cdot\Mat_{4}(\Oo_\F)$.
Plus précisément, on va utiliser l'homomorphisme $\KMS_{4}$. 
On utilise aussi la notation $\z(\s,\mu)$ du paragraphe \ref{ClasJam}, 
où $\s$ est ici le caractère trivial de $\kk_\F^{\times}$ et où $\mu$ est une 
partition. 
Pour alléger les notations, on omettra $\s$, 
de sorte qu'on notera simplement $\z(\mu)$. 
D'après les pro\-po\-si\-tions \ref{KMAX1} et \ref{prpisegK} et le 
théorème \ref{ResumeJames}, on a~:
\begin{equation*}
\KMS_{4}(\Pi)=\z(3)\times\z(1)=\z(3,1)+k\cdot\z(4),
\quad
k\>0.
\end{equation*}
Si l'on écrit $\sy{\Pi}=\sy{\pi}+\sy{\pi^\vee}+\sy{\tau}$,
alors $\z(3,1)$, qui est de multiplicité $1$, ne peut appa\-raî\-tre que dans 
$\KMS_{4}(\tau)$. 
Comme $\rp_{(2,2)}(\Pi)$ est de longueur $2$, on a $\rp_{(2,2)}(\tau)=\{0\}$, 
ce qui impli\-que que $\rp_{(2,2)}(\KMS_{4}(\tau))=\{0\}$.
Mais $\rp_{(2,2)}(\z(3,1))\neq\{0\}$ 
puisque $\rp_{(1,1,1,1)}(\z(3,1))\neq\{0\}$, ce qui donne une contra\-diction.
\end{proof}

\label{cas4}

On suppose maintenant que $a'=0$ et que $b'$ est pair et $\>4$.
D'après la proposition \ref{ayuto}(1),
la représentation $\left<0\right>\times \left< \Delta' \right>$ 
est une sous-représentation de 
$\left<0\right> \times \left<0\right> \times \left< ^-\Delta' \right>$, 
laquelle, par le lemme \ref{mult1}, possède une 
unique sous-représentation irréductible, notée $\pi$, 
qui apparaît
avec multiplicité $1$ dans $\left<0\right> \times \left<0\right> 
\times \left< ^-\Delta'\right>$. 
On déduit que $\pi$ est l'unique sous-représentation irréductible de
$\left<0\right> \times \left< \Delta' \right>$ et y apparaît avec multiplicité $1$. 
De même, on montre que $\left<0\right> \times \left< \Delta' \right>$ possède un 
unique quotient irréductible, notée $\pi'$, qui apparaît avec 
multiplicité $1$ dans $\left<0\right> \times \Delta'$.  
Pour prouver que cette induite est irréductible, il suffit donc de montrer que 
$\pi \simeq \pi'$.  
Le module de Jacquet associé à la partition $\alpha=(1,1,\dots,1)$ de 
$\left<0\right> \times \left< \Delta' \right> $ contient comme sous-quotient la 
représentation~:
\begin{equation*}
\tau= \left<0\right> \otimes\left<1\right> \otimes\left<0\right> \otimes \left<0\right> 
\otimes\left<1\right> \otimes\left<0\right> \otimes \dots \otimes\left<1\right> \otimes 
\left<0\right>
\end{equation*}
avec multiplicité $2$. 
Montrons que, de même, les modules de Jacquet $\rp_\alpha(\pi)$ et 
$\rp_\alpha(\pi')$ contiennent $\tau$ comme sous-quotient avec 
multiplicité $2$, ce qui prouvera que $\pi \simeq \pi'$.  
D'après la proposition \ref{ayuto}(1), la représentation $\pi$ est une 
sous-représentation de l'induite 
$\left<0\right> \times \left< \Delta'_1 \right> \times 
\left< \Delta'_2 \right>$, avec $\Delta'_1= \left[0,1, 2\right] $ et 
$\Delta'_2= \left[3,b'\right] $, donc par réciprocité de Frobenius 
on trouve que~:
\begin{equation*}
\Hom(\rp_{(4,b'-2)}( \pi), (\left<0\right> \times \left<  \Delta'_1 \right>
  )\otimes \left< \Delta'_2 \right>) \neq \{0\}.
\end{equation*} 
D'après le lemme \ref{TwistPapier}, la représentation 
$ (\left<0\right> \times \left< \Delta'_1 \right> )\otimes \left<
\Delta'_2 \right>$ est irréductible donc~:
\begin{equation*}
\sy{\rp_\alpha\left((\left<0\right> \times 
\left< \Delta'_1 \right> )\otimes \left<\Delta'_2 \right>\right)} 
\< \sy{\rp_\alpha(\pi)}.
\end{equation*}
Comme $\tau$, d'après le lemme géométrique, apparaît avec
multiplicité $2$ dans le module de Jacquet 
$\rp_\alpha\left((\left<0\right> \times \left< \Delta'_1 \right> )
\otimes \left< \Delta'_2 \right>\right)$, on trouve que $\tau$ apparaît
avec multiplicité $2$ dans $\rp_\alpha(\pi)$. 
De la même fa\c{c}on, on prouve que $\tau$ apparaît avec multiplicité
$2$ dans $\rp_\alpha(\pi')$. 

\medskip

Ceci met fin à la démonstration de la proposition \ref{2seg}.

\subsubsection{}

On en déduit le résultat suivant. 

\begin{coro}
\label{nuevo2cor}
Soient $\Delta_1,\dots,\Delta_r$ des segments. 
Supposons que, pour tous $1\<i,j\<r$ tels que $i\neq j$, 
les segments $\Delta_i$ et $\Delta_j$ soient non liés. 
Alors, la représentation~:
\begin{equation*}
\left< \Delta_1\right>\times \dots \times \left< \Delta_r\right>
\end{equation*} 
est irréductible. 
\end{coro}

\begin{proof}
La preuve se fait par récurrence sur $r$, le cas $r=2$ étant 
traité par la proposition \ref{2seg}. 
Supposons donc que $r\>3$. 
Si $\rho_i\mu_{\rho_i}^{a_i}\simeq\rho_j\mu_{\rho_j}^{a_j}$ et 
$\rho_i\mu_{\rho_i}^{b_i}\simeq\rho_j\mu_{\rho_j}^{b_j}$ pour 
tous $1 \< i,j \< r$, alors la preuve est similaire au 
cas particulier traité dans le paragraphe \ref{cas6}. 
Remarquons aussi que, par hypothèse et d'après la remarque 
\ref{2sep}, 
$\langle \Delta_1 \rangle \times \dots \times \langle \Delta_r \rangle$ est
isomorphe à la représentation 
$\langle\Delta_{\sigma(1)}\rangle\times
\dots\times\langle\Delta_{\sigma(r)}\rangle$ pour toute permutation
$\sigma$ de l'ensemble $\{1, \dots, r \}$. 

On note $\Delta_i= \left[a_i,b_i\right]_{\rho_i}$ 
et on suppose que $n(\Delta_1) \> n(\Delta_i)$ pour chaque $i$. 
Quitte à passer à la contragrédiente, on peut donc supposer, d'après la
remarque \ref{dosliados}(1), qu'il existe un entier 
$1\<r_1<r$ tel que 
$\rho_i\mu_{\rho_i}^{b_i}\simeq \rho_1\mu_{\rho_1}^{b_1}$
pour tout $1\< i \< r_1$ et 
$\rho_1\mu_{\rho_1}^{b_1} \notin\supp(\Delta_k)$
pour tout $r_1< k \< r$.
Ainsi, par le lemme géométrique et par hypothèse de récurrence, 
les re\-pré\-sentations irréductibles~:
\begin{equation*}
\left< \Delta_1 \right> \times \dots \times \left< \Delta_{r_1} \right>,
\quad
\left< \Delta_{r_1+1} \right> \times \dots \times \left< \Delta_r \right>
\end{equation*}
satisfont aux conditions du lemme \ref{mmm1} avec~:
\begin{equation*}
\beta=\left( \deg\(\Delta_1\)+\dots+\deg\( \Delta_{r_1}\),\deg\(\Delta_{r_1+1}\)+
 \dots+\deg\( \Delta_{r}\)\right)
\end{equation*}
et donc~: 
$$\( \left< \Delta_1 \right> \times \dots \times \left< \Delta_{r_1}
 \right>\) \otimes \( \left< \Delta_{r_1+1} \right> \times \dots
 \times \left< \Delta_r \right>\)$$ 
apparaît avec multiplicité $1$ dans le module de Jacquet
$\rp_{\beta }\( \left< \Delta_1 \right> \times \dots \times \left< \Delta_r
 \right>\)$.
On utilise maintenant le lemme \ref{sir} pour en déduire que 
$\left< \Delta_1 \right> \times \dots \times \left< \Delta_r \right>$ est
irréductible. 
\end{proof}

Pour compléter la preuve des théorèmes \ref{nuevo2} et \ref{nuevo3}, il ne
reste à montrer que la propo\-sition suivante.

\begin{prop}
Soient $\Delta$ et $\Delta'$ deux segments \textit{liés}.  
Alors $\left<\Delta\right> \times \left<\Delta'\right>$ est 
réductible. 
\end{prop}

\begin{proof}
On peut supposer que $\Delta=[a,b]_\rho$ et $\Delta'=[a',b']_{\rho}$, 
où $\rho$ est une représenta\-tion irréductible cuspidale de degré $m$. 
Quitte à échanger $\Delta$ et $\Delta'$, on peut aussi supposer qu'on a 
$a'+1\<a\<b'+1\<b$. 
On pose~:
\begin{equation*}
\Delta^\cup=[a',b]_\rho,
\quad
\Delta^\cap=[a,b']_\rho,
\quad
\pi=\langle\Delta\rangle \times \langle\Delta'\rangle,
\quad
\om= \langle\Delta^\cup\rangle\times\langle\Delta^\cap\rangle.
\end{equation*}
Si $a=b'+1$, on interprète $\left<\Delta^\cap\right>$ comme la 
représentation triviale du groupe trivial $\G_0$, de sorte que 
$\om=\langle\Delta^\cup\rangle$. 
Si l'on pose $\a=(m,\dots,m)$, le lemme géométrique montre que 
le nombre de sous-quotient irréductibles (comptés avec multiplicités) 
de $\rp_\a(\om)$ est strictement inférieur à celui de $\rp_\a(\pi)$. 

\begin{lemm}
On a $\Hom(\pi,\om)\neq\{0\}$.
\end{lemm}

\begin{proof}
Supposons d'abord que $a=b'+1$, \ie que $\om=\langle a',b\rangle$.
Dans ce cas, le résul\-tat découle de la proposition \ref{ayuto}.  

Sinon, on peut appliquer l'argument de la partie (2) de 
la preuve de \cite[Proposition 4.6]{Ze2}.
\end{proof}

Il existe donc un quotient non nul de $\pi$ qui est une sous-représentation de 
$\om$. 
Supposons que $\pi$ est irréduc\-tible.
Elle se plonge donc dans $\om$, 
ce qui contredit le fait énoncé plus haut comparant les modules de Jacquet 
$\rp_\a(\om)$ et $\rp_\a(\pi)$.  
\end{proof}


\section{Représentations résiduellement non dégénérées de $\GL_m(\D)$}
\label{Sec9}

Dans cette section, on introduit la notion de représentation résiduellement 
non dégé\-né\-rée de $\GL_m(\D)$, qui généralise celle de représentation 
non dégénérée de $\GL_n(\F)$.  
En parti\-culier, toute représentation irréductible cuspidale est résiduellement 
non dégé\-né\-rée. 
Dans le paragra\-phe \ref{ClasBan}, on donne des conditions néces\-sai\-res à 
l'apparition de sous-quotients cuspidaux dans une induite parabolique.
Dans le para\-gra\-phe \ref{USSC}, on prouve la conjecture d'unicité du 
support supercuspidal pour $\GL_m(\D)$.  
Ceci nous permet de classer les représentations 
résiduellement non dégénérées (proposition \ref{equivSt}) en fonction de 
leur support supercuspidal.

\subsection{Représentations résiduellement non dégénérées} 
\label{ModWhitResRND}

Lorsque $\D$ est non commutative, on ne peut pas définir la notion de 
représentation non dégé\-né\-rée de $\GL_m(\D)$ comme dans \cite{Vig2}
parce que la théorie des dérivées de Bernstein et Zele\-vin\-ski ne fonctionne pas.
On va définir la notion de représentation résiduellement non dégé\-nérée 
en utilisant l'homomorphisme $\KMS_{n}$ défini au paragraphe \ref{PropZLS}, 
qui permet de se ramener à la notion de 
repré\-sentation non dégénérée d'un groupe linéaire général fini.  
On verra au chapitre \ref{Sec10} (voir le corollaire \ref{coroSt}) que, 
dans le cas déployé, les deux notions coïncident. 

\subsubsection{}

Soient $m,n\>1$ des entiers, soit $\rho$ une 
représentation irréductible supercuspidale de 
$\G_{m}$ et soit $\Om=\Om_{\rho,n}$ la classe d'inertie 
du support cuspidal 
$\sy{\rho}+\dots+\sy{\rho}=n\cdot\sy{\rho}$.
On choisit un type simple maximal $\l=\k\otimes\s$ contenu dans $\rho$
et on forme le morphisme $\KMS_{n}$ (voir le paragraphe \ref{PropZLS}).
On rappelle (voir \eqref{Mellamphy}) que $\Irr(\Om)$ est l'ensemble 
des classes de re\-pré\-sen\-ta\-tions irréductibles
de $\G$ qui sont sous-quotients 
d'une induite parabolique d'un élément de $\Om$.

\begin{defi}
\label{ResNonDeg1}
Une représentation $\pi\in\Irr(\Om)$ 
est \textit{résiduel\-le\-ment non dégé\-nérée} 
si $\KMS_{n}(\pi)$ possède un sous-quotient irréductible non dégénéré
(voir le paragraphe \ref{RNDWh}). 
\end{defi}

D'après la proposition \ref{DeliceDOrient}, toute représentation irréductible 
cuspidale 
dans $\Irr(\Om)$ est rési\-duel\-lement non dégénérée.  
D'après la proposition \ref{CalculKMSI}, 
et puisque l'unique sous-quotient irréductible 
non dégénéré de $\s\times\dots\times\s$ est $\st(\s,n)$ 
(voir \ref{Clothier}), 
la représentation $\pi$ est résiduel\-le\-ment non dégénérée si
et seulement si $\KMS_{n}(\pi)$ contient $\sy{\st(\s,n)}$.

\begin{exem}
Par exemple, la représentation $\St(\rho,n)$ de la définition \ref{DefSTRN} 
est l'uni\-que sous-quotient irré\-duc\-tible résiduellement non dégénéré de 
$\rho\times\rho\nu_\rho^{}\times\dots\times\rho\nu_\rho^{n-1}$. 
\end{exem}

Cette définition ne dépend ni du choix de $\l$ ni de la 
décomposition $\l=\k\otimes\s$.  

\subsubsection{}

Soit maintenant $\pi$ une représentation irréductible quelconque, 
qu'on décom\-po\-se sous la forme $\pi=\pi_1\times\dots\times\pi_r$
donnée par le théorème \ref{ssvarscinertie}.

\begin{defi}
\label{ResNonDeg2}
La représentation $\pi$ est dite \textit{résiduellement non dégénérée} si
les re\-pré\-sentations $\pi_1,\dots,\pi_r$ sont résiduellement non dégénérées 
au sens de la défi\-nition \ref{ResNonDeg1}. 
\end{defi}

On note $\Rnd$ le sous-ensemble de $\Irr$ formé des classes 
d'équivalence de représentations irré\-ductibles résiduel\-lement non 
dégénérées. 

La proposition suivante donne des conditions nécessaires et suffisantes 
d'apparition d'un sous-quotient irréductible résiduel\-lement non dégénéré 
dans une induite parabolique. 

\begin{prop}
\label{SQIRND1}
Soient $\pi_1,\dots,\pi_r$ des représentations irréductibles.  
Les conditions sui\-van\-tes sont équivalentes~:
\begin{enumerate}
\item 
L'induite~:
\begin{equation*}
\pi_1\times\dots\times\pi_r
\end{equation*}
possède un sous-quotient irréductible résiduellement non dégénéré.
\item 
Pour tout $1 \< i \< r$, la représentation $\pi_i$ est résiduellement non
dégénérée. 
\end{enumerate}
Si ces conditions son satisfaites, ce sous-quotient irréductible
résiduellement non dégénéré est unique et sa multiplicité dans l'induite est 
$1$. 
On le note $\St(\pi_1, \dots, \pi_n)$. 
\end{prop}

\begin{proof}
D'après la définition \ref{ResNonDeg2}, on peut se ramener au cas où 
il y a une repré\-sen\-tation irréductible supercuspidale $\rho$ est des 
entiers $n_1,\dots,n_r$ tels 
que $\pi_i\in\Irr(\Om_{\rho,n_i})$ pour tout entier $i=1,\dots,r$. 

On suppose dans un premier temps que $\pi_1,\dots,\pi_r$ sont résiduellement 
non dégénérées.  
Dans cas, on écrit~:
\begin{equation*}
\sy{\KMS_{n}(\pi_1\times\dots\times\pi_r)}=
\sy{\KMS_{n_1}(\pi_1)\times\dots\times\KMS_{n_r}(\pi_r)}\>
\sy{\st(\s,n_1)\times\dots\times\st(\s,n_r)}. 
\end{equation*}
Le résultat vient de ce que le membre de droite contient $\st(\s,n)$ 
avec multiplicité $1$. 

On suppose maintenant que $\pi_1$ n'est pas résiduellement non 
dégénérée, \ie que $\KMS_{n_1}(\pi_1)$ ne contient pas de facteur non dégénéré. 
Le résultat découle de la proposition \ref{Adeg}. 
\end{proof}

\begin{coro}
\label{MrRubb}
Soient $\Delta_1,\dots,\Delta_r$ des segments.
La représentation~:
\begin{equation*}
\Z(\Delta_1)\times\dots\times\Z(\Delta_r)
\end{equation*}
contient un sous-quotient irréductible 
résiduel\-lement non dégénérée si et seulement si tous les $\Delta_i$
sont de longueur $1$. 
\end{coro}

\begin{proof}
D'après la proposition \ref{prpisegK}, pour tout segment $\Delta$, on a~:
\begin{equation*}
\KMS_{n}\(\Z(\Delta)\)=\z(\s,n(\Delta))
\end{equation*}
qui est non dégénérée si et seulement si $n(\Delta)=1$.  
Le résultat se déduit de la proposition \ref{SQIRND1}. 
\end{proof}

\subsubsection{}
\label{DefPartition}

On définit maintenant la notion de représentation résiduellement dégénérée 
par rapport à une partition. 
Soit $\a=(m_1,\dots,m_r)$ une famille d'entiers $\>1$ de somme $m$.

\begin{defi}
\label{ResNonDeg3}
\begin{enumerate}
\item 
Une représentation irréduc\-tible $\pi_1\otimes\dots\otimes\pi_r$
du sous-groupe de Levi $\M_{\a}$ de $\G_m$ est dite 
\textit{résiduellement non dégéné\-rée} si, pour tout 
$i\in\{1,\dots,r\}$, la représentation $\pi_i$ est résiduellement non 
dégénérée au sens de la défi\-nition \ref{ResNonDeg2}. 
\item 
Une représentation irréductible $\pi$ de $\G_m$ est 
\textit{résiduellement $\a$-dégé\-né\-rée} si 
$\rp_{\a}(\pi)$ pos\-sè\-de un sous-quotient irréductible résiduellement 
non dégénéré. 
\end{enumerate}
\end{defi}

La définition \ref{ResNonDeg3}(2) ne dépend pas de l'ordre des entiers $m_i$. 
Si $\b$ est une famille d'entiers $\>1$ obtenue par permutation de $\a$, alors les 
sous-groupes paraboliques $\P_\a$ et $\P_\b$ sont conjugués, et 
$\pi$ est résiduellement $\a$-dégénérée si et seulement si elle est 
résiduellement $\b$-dégénérée.
Si l'on oublie l'ordre des $m_i$ et qu'on les range dans l'ordre décroissant, 
on obtient une partition de $m$ dite \textit{associée} à la famille $\a$. 

\begin{defi}
\label{LaisserAller}
Soit $\mu$ une partition de $m$.
La représentation $\pi$ est dite \textit{résiduellement $\mu$-dégé\-né\-rée} 
si elle est résiduellement $\a$-dégé\-né\-rée pour au moins une 
(donc pour toute) famille $\a$ associée à $\mu$.
\end{defi}

On remarque qu'une 
représentation irréductible est résiduellement non dégé\-né\-rée
si et seulement si elle est résiduellement $(m)$-dégé\-né\-rée.

\begin{prop}
\label{prodSt}
Soient $\mu$ et $\nu$ des partitions.
Soit $\pi$ une représentation résiduellement $\mu$-dé\-gé\-nérée et 
soit $\s$ une représentation résiduellement $\nu$-dégénérée. 
Alors $\pi \times\s$ contient un sous-quotient irré\-duc\-tible 
résiduellement $(\mu+\nu)$-dégénéré. 
\end{prop}

\begin{proof}
Le résultat se déduit du lemme géométrique et de la proposition \ref{SQIRND1}.  
\end{proof}

\subsection{Conditions d'apparition d'un facteur cuspidal dans une induite}
\label{ClasBan}

Dans ce paragraphe, on donne des conditions nécessaires d'apparition 
d'un facteur cus\-pi\-dal dans une induite parabolique.
Le résultat suivant, qui n'apparaît pas dans \cite{Vig1}, 
est une étape cruciale dans la 
preuve de l'unicité du support supercuspidal (théorème \ref{unicitesupp}).  
Il permet d'utiliser les arguments de Zelevinski (voir \cite{MS}).

\begin{prop}
\label{ST}
Soient $\rho_1,\dots,\rho_n$ des représentations
irréductibles cuspidales avec $n\>2$.
On suppose que~:
\begin{equation}
\label{CondBanalit}
\ZZ_{\rho_i}\nsubseteq\{\rho_1,\dots,\rho_n\},
\quad
i=1,\dots,n.
\end{equation}
Alors aucun sous-quotient irréductible de l'induite 
$\rho_1\times\dots\times\rho_n$ n'est cuspidal.
\end{prop}

\begin{proof}
D'après le théorème \ref{ss}, on peut se ramener au cas où les $\rho_i$ sont 
toutes inertiellement équivalentes à une même représentation cuspidale $\rho$ 
de $\G_{m}$, $m\>1$. 
La condition \eqref{CondBanalit} entraîne d'une part que 
$\rho$ est supercuspidale avec $\ee(\rho)=\e(\rho)\>2$ 
(voir la remarque \ref{ValeCNS}), d'autre part qu'il existe des segments 
$\Delta_1,\dots,\Delta_r$ satisfaisant aux conditions suivantes~:
\begin{enumerate}
\item
on a
$\sy{\rho\chi_1}+\dots+\sy{\rho\chi_n}=\supp(\Delta_1)+\dots+\supp(\Delta_r)$~;
\item
pour tous $i\neq j$, les segments $\Delta_i$ et $\Delta_j$ ne 
sont pas liés~;
\item
pour tout $i$, la longueur $n_i$ de $\Delta_i$ est $\<e(\rho)-1$.
\end{enumerate}
Ces segments sont uniques à l'ordre près. 
D'après le théorème \ref{nuevo2}, la représentation~:
\begin{equation}
\Pi=\L(\Delta_1)\times\dots\times\L(\Delta_r)
\end{equation}
est irréductible.

\begin{lemm}
\label{Contra1}
La représentation $\Pi$ est l'unique sous-quotient irréductible 
rési\-duel\-le\-ment 
non dégénéré de l'induite $\rho_1\times\dots\times\rho_n$,
où elle apparaît avec multiplicité $1$.
\end{lemm}

\begin{proof}
D'abord, $\Pi$ est un sous-quotient irréductible de 
$\rho_1\times\dots\times\rho_n$.
D'après la proposition \ref{SQIRND1}, 
cette induite contient un unique facteur irréduc\-tible résiduellement 
non dégé\-néré.
D'après la proposition \ref{SQIRND1} encore, $\Pi$ est résiduellement 
non dégénérée si et seulement si chaque $\L(\Delta_i)$ est résiduellement 
non dégénéré, ce qui suit de la proposition \ref{prpisegK}. 
\end{proof}

\begin{lemm}
\label{Contra2}
La représentation $\Pi$ n'est pas cuspidale.
\end{lemm}

\begin{proof}
Le résultat est immédiat si $r\>2$.
Si $r=1$, $\Pi$ est une repré\-sen\-tation associée à un segment, 
qui est par définition quotient d'une induite parabolique propre 
(car $n\>2$).
\end{proof}

La conjonction des lemmes \ref{Contra1} et \ref{Contra2} et du fait que 
toute représentation irréductible cuspi\-dale est 
résiduellement non dégénérée entraîne que l'induite 
ne possède pas de sous-quotient ir\-ré\-ductible cuspidal. 
\end{proof}

\begin{exem}
\label{nohay}
Si $\CR$ est de caractéristique nulle, \eqref{CondBanalit} est toujours vérifiée.
\end{exem}

Soient $m,n\>1$ et soit $\rho$ une représentation cuspidale de $\G_m$.  

\begin{prop}
\label{simsolo}
Soient $\chi_1,\dots,\chi_n$ des caractères non ramifiés de $\G_m$
tels que~:
\begin{equation}
\label{IndRhoXi}
\rho\chi_1\times\dots\times\rho\chi_n
\end{equation}
possède un sous-quotient cuspidal.
Alors~:
\begin{equation*}
\sy{\rho\chi_1}+\dots+\sy{\rho\chi_n}=
\sy{\rho\chi_1}+\sy{\rho\chi_1\nu_\rho}+\dots+\sy{\rho\chi_1\nu_\rho^{n-1}}.
\end{equation*}
\end{prop}

\begin{proof}
Quitte à remplacer $\rho$ par $\rho\chi_1$, on peut supposer que $\chi_1$ est 
trivial. 
On écarte le cas trivial où $n=1$, 
de sorte que, d'après le corollaire \ref{MlaDim}, 
il existe $r\>0$ tel que $n=\ee(\rho)\ell^r$.
Ensuite, d'après la proposition \ref{Zdroites},  on peut supposer que
$\chi_i=\nu_{\rho}^{k_i}$ avec $k_i\in\ZZ$. 
Si $\e(\rho)=1$, chaque $\chi_i$ est trivial et le résultat est immédiat. 
On suppose donc dorénavant que $\e(\rho)\>2$, 
de sorte que $\ee(\rho)=\e(\rho)$.
D'après la proposition \ref{ST}, il existe un entier $t\>1$ tel que,
quitte à réordonner les $\chi_i$, on ait~:
\begin{enumerate}
\item
pour tout $i\<t\ee(\rho)$, on a $\sy{\rho\chi_i}=\sy{\rho\nu_\rho^{i-1}}$~; 
\item
la condition \eqref{CondBanalit} est vérifiée par les représentations 
$\rho\chi_{t\ee(\rho)+1},\dots,\rho\chi_{n}$.
\end{enumerate}
On écrit $n=t\ee(\rho)+k$ et on suppose que $k\>1$. 
Soit $\pi$ un sous-quotient irréductible cuspidal de (\ref{IndRhoXi}).
Soit $\pi_1$ un sous-quotient irréductible de 
$\rho\chi_1\times\dots\times\rho\chi_{t\ee(\rho)}$ 
et $\pi_2$ un sous-quotient irréductible de
$\rho\chi_{t\ee(\rho)+1}\times\dots\times\rho\chi_n$
tels que $\pi$ soit un sous-quotient irréductible 
de $\pi_1\times\pi_2$. 
Comme $\pi$ est résiduellement non dégénérée, 
c'est vrai aussi de $\pi_1$ et de $\pi_2$. 
En parti\-culier, $\pi_1$ est égal à $\St(\rho,\ee(\rho)t)$. 

Si l'on écrit $t=t'\ell^r$ avec $t'\>1$ premier à $\ell$ et $r\>0$, 
et si l'on pose $\tau=\St(\rho,\ee(\rho)\ell^r)$, 
alors $\tau$ est cuspidale non supercuspidale 
et $\pi_1$ est égale à $\St(\tau,t')$ d'après la proposition 
\ref{StCusp} et le lemme \ref{SpSp2}.
En particulier, 
le support cuspidal de $\pi_1$ est dans $\Dive(\ZZ_\tau)$.
D'autre part, d'après la proposition \ref{ST}, 
le support cuspidal de $\pi_2$ est dans $\Dive(\ZZ_\rho)$.
D'après le théorème \ref{ss},
l'induite $\pi_1\times\pi_2$ est irréductible, ce qui contredit 
le fait que $\pi$ est cuspidale. 
On en déduit que $k=0$, ce qui met fin à la démonstration.
\end{proof}

\begin{rema}
\label{L=St}
D'après la proposition \ref{prpisegK}, les représentations 
$\L([0,n-1]_\rho)$ et $\St(\rho,n)$ sont isomorphes si et seulement si 
$n<\ee(\rho)$. 
Si $n\>\ee(\rho)$, la représenta\-tion $\L([0,n-1]_\rho)$ n'est jamais 
résiduellement non dégénérée. 
Par exemple, si $\rho$ est le caractère trivial de $\mult\F$ 
et si $q$ est d'ordre $2$ dans $\CR^\times$, alors la représenta\-tion 
$\L([0,1]_\rho)$ est un caractère de $\GL_2(\F)$. 
\end{rema}

\begin{rema}
Une représentation irréductible dont le support cuspidal 
$\sy{\rho_1}+\dots+\sy{\rho_n}$ vérifie la condition \eqref{CondBanalit} 
est dite \textit{banale}. 
Les représentations banales sont étudiées en détail dans \cite{MS}. 
\end{rema}

\subsection{Unicité du support supercuspidal}
\label{USSC}
\label{supercuspidal}

Le but de ce paragraphe est de montrer le théorème suivant.

\begin{theo}
\label{unicitesupp}
Soient $\rho_1,\dots,\rho_n$ et $\rho'_1,\dots,\rho'_{n'}$ des 
représentations irréductibles supercuspidales. 
Alors $\rho_1\times \dots \times \rho_n$ 
et $\rho'_1\times \dots \times \rho'_{n'}$ 
ont un sous-quotient irréductible en commun si et seulement si
$n'=n$ et~:
\begin{equation}
\label{Gretel2}
\sy{\rho_1}+\dots+\sy{\rho_n}=\sy{\rho'_1}+\dots+\sy{\rho'_{n'}}.
\end{equation}
\end{theo}

\begin{rema}
Voir \cite{Vig2} page 598 dans le cas où $\D=\F$.
\end{rema}

\begin{proof}
Si (\ref{Gretel2}) est vérifié, alors $n=n'$ et,
d'après la proposition \ref{commu}, les
re\-pré\-sen\-ta\-tions $\rho_1\times \dots \times \rho_n$ et 
$\rho'_1\times\dots\times\rho'_{n'}$ ont les mêmes sous-quotients 
irréductibles.  
Pour prouver la réciproque, on fixe un sous-quotient irréductible $\pi$ 
commun à $\rho_1\times \dots \times \rho_n$ 
et $\rho'_1\times \dots \times \rho'_{n'}$, et on commence par traiter le 
cas où $\pi$ est cuspidal.
D'après le théorème \ref{ss},
il existe des représentations irréductibles supercuspidales $\rho$ et 
$\rho'$ telle que les $\rho_i$ (les $\rho'_i$)
soient inertiellement équivalentes à $\rho$ (à $\rho'$).
D'après la proposition \ref{simsolo}, on peut même
supposer que~:
\begin{equation*}
\sy{\rho_1}+\dots+\sy{\rho_n}=
\sy{\rho}+\sy{\rho\nu_\rho}+\dots+\sy{\rho\nu_\rho^{n-1}}.
\end{equation*}
Comme $\pi$ est cuspidal, c'est l'unique sous-quotient irréductible 
résiduellement non dégé\-né\-ré de $\rho_1\times \dots \times \rho_n$.
Il est donc isomorphe à $\St(\rho,n)$.
On a quelque chose d'analogue pour $\rho'$, ce dont on 
déduit que $\St(\rho,n)$ et $\St(\rho',n')$ sont tous les deux 
isomorphes à $\pi$.
Le résultat est une conséquence de la proposition \ref{AppCuspSuper}. 

On traite maintenant le cas général.
Soit $\a$ une famille d'entiers de somme $\deg(\pi)$ telle que $\rp_\a(\pi)$ 
soit cuspidale, et écrivons $\sy{\tau_1}+\dots+\sy{\tau_r}$ son support 
cuspidal. 
Par exactitude du foncteur de
Jacquet, la représentation $\tau_1 \otimes \dots \otimes \tau_r$ 
est un sous-quotient irréductible de 
$\rp_\a(\rho_1\times \dots\times \rho_n)$ et de 
$\rp_\a(\rho'_1\times \dots \times \rho'_{n'})$,
où $\a$ est la famille d'entiers $(\deg(\tau_1),\dots,\deg(\tau_r))$. 
Ces représentations ont chacune une filtration dont les sous-quotients 
irréductibles sont de la forme \eqref{lemmegeometrique}. 
Le résultat étant vrai pour les
représentations $\tau_i$ on déduit alors (\ref{Gretel2}).  
\end{proof}

\begin{defi}
Étant donnée une représentation irréductible $\pi\in\Irr$,
il existe une unique somme 
$\sy{\rho_1}+\dots+\sy{\rho_n}\in\Dive(\Ss)$ telle que $\pi$ soit un sous-quotient 
irréductible de $\rho_1\times \dots\times \rho_n$.
On l'appelle le \textit{support supercuspidal} de $\pi$,
que l'on note $\scusp(\pi)$.
\end{defi}

Pour $\ss\in\Dive(\Ss)$, on note $\Irr(\ss)$ l'ensemble des classes 
de représentations irréductibles de support supercuspidal $\ss$.  
On en déduit une variante supercuspidale du théorème \ref{ss}.

\begin{theo}
\label{ssvarsc}
Soit $r\>1$ un entier et soient $\rho_1,\dots,\rho_r$ des 
représentations irré\-ductibles supercuspidales deux à deux 
non inertiellement équivalentes. 
Pour chaque $i$, on fixe un support supercuspidal 
$\ss_i\in\Dive(\Om_{\rho_i})$.
\begin{enumerate}
\item 
Pour chaque $i$, soit $\pi_i$ une représentation irréductible 
de support supercuspidal $\ss_i$.
Alors $\pi_1\times\dots\times\pi_r$ est irréductible.
\item 
On pose $\ss=\ss_1+\dots+\ss_r$.
L'application~:
\begin{equation*}
(\pi_1,\dots,\pi_r)\to\pi_1\times\dots\times\pi_r
\end{equation*}
induit une bijection de $\Irr(\ss_1)\times\dots\times\Irr(\ss_r)$ 
dans $\Irr(\ss)$.
\end{enumerate}
\end{theo}

La proposition suivante découle de l'unicité du support supercuspidal 
et de la proposition \ref{SQIRND1}. 

\begin{prop}
\label{equivSt}
L'application 
$(\rho_1,\dots,\rho_r)\mapsto\St(\rho_1,\dots,\rho_r)$
de $\Dive(\Ss)$ dans $\Rnd$ est une bijection, et sa réciproque est donnée 
par $\pi\mapsto\scusp(\pi)$.
\end{prop}


\section{Classification des représentations irréductibles de $\GL_m(\D)$}
\label{Sec10}

L'objectif de cette section est le théorème \ref{bijj2thm}, qui fournit une
classification à la Zele\-vinski des représen\-ta\-tions irré\-ductibles 
de $\GL_m(\D)$, $m\>1$, en termes de multisegments. 
Les multisegments sont définis au paragraphe \ref{multisegment}.  
Dans le paragraphe \ref{scap}, on définit les multisegments 
supercuspidaux et les multisegments apériodiques, et on montre 
comment passer bijectivement des uns aux autres. 
On donne au paragraphe \ref{SST} une classi\-fi\-cation des
représentations résiduellement non dégénérées en fonction de leur 
support cuspi\-dal. 
Dans le paragraphe \ref{ccll}, on associe à tout multisegment $\m$ 
une représentation irréductible $\Z(\m)$.
Dans le paragraphe \ref{ccll2}, on montre que 
l'application $\m\mapsto \Z(\m)$ est surjective et on calcule ses fibres.  
Dans le para\-graphe \ref{rrre}, on étudie la réduction modulo 
$\ell$ des $\qlb$-représentations irréductibles. 


\subsection{Multisegments} 
\label{multisegment}

Dans ce paragraphe, on définit la notion de multisegment.  
On note $\seg$ l'ensem\-ble des classes d'équivalence de segments 
(définition \ref{DefEquSeg}). 
Le degré, la longueur, le support d'un segment $\Delta$ ne
dépendent que de la classe d'équivalence de ce segment, ainsi que les classes 
de représentations $\sy{\Z(\Delta)}$ et $\sy{\L(\Delta)}$.

\begin{defi}
Un \textit{multisegment} est un multi-ensemble de classes 
de segments, \ie un élément de $\Dive(\seg)$.
On note~:
\begin{equation}
\label{DefMS}
\MS=\Dive(\seg)
\end{equation}
 l'ensemble des multisegments.
\end{defi}

Un multisegment non nul s'écrit sous la forme~:
\begin{equation}
\label{Huffam}
\m
=\Delta_1+\dots+\Delta_{r}
=\left[a_1,b_1\right]_{\rho_1}+\dots+\left[a_r,b_r\right]_{\rho_r},
\end{equation}
où chaque $\Delta_i=\left[a_i,b_i\right]_{\rho_i}$ est une classe 
de segments. 
La longueur, le degré et le support, définis pour les segments en 
(\ref{LDE}), sont définis pour les multisegments par ad\-di\-tivité~:
\begin{equation*}
n(\m)=\sum\limits_{i=1}^{r} n(\Delta_i),
\quad
\deg\(\m\)=\sum\limits_{i=1}^{r}\deg\(\Delta_i\),
\quad
\supp(\m)=\sum\limits_{i=1}^{r}\supp\(\Delta_i\)
\end{equation*}
désignent respectivement la longueur, le degré et le support de $\m$.
On définit aussi par addi\-ti\-vi\-té les applications 
$\m\mapsto\m^\vee$ et $\m\mapsto\m^-$ 
de $\MS$ dans $\MS$ (voir \eqref{contragr} et \eqref{DefMoins}).  
Pour la seconde, on convient que, si $\Delta$ est un 
segment de longueur $1$, alors $\Delta^-$ est le multisegment nul.

Voici une série de lemmes qui seront utiles par la suite.

\begin{lemm}
\label{Silas}
Soient $\m,\m'$ des multisegments tels que $\m^-=\m'^-$. 
Il existe alors des multi\-segments $\nn$ et $\nn'$ qui sont sommes 
de segments de longueur $1$ et tels qu'on ait $\m+\nn=\m'+\nn'$.
\end{lemm}

\begin{proof}
Si l'on étend l'opération $\m\mapsto\m^-$ en un endomorphisme de 
groupe de $\Div(\seg)$, il suffit de prouver que le noyau de cet endomorphisme 
est réduit au sous-groupe engendré par les segments de 
longueur $1$, ce qui est immédiat.
\end{proof}

Soit $\m$ un multisegment, qu'on écrit sous la forme \eqref{Huffam}.
On pose~:
\begin{equation*}
\m^{(1)}= \sy{\rho_1\nu_{\rho_1}^{b_1}}+\sy{\rho_2\nu_{\rho_2}^{b_2}}
+\dots+\sy{\rho_r\nu_{\rho_r}^{b_r}}
\end{equation*}
la somme des extrémités finales des segments composant $\m$.
Ceci définit une application $\m\mapsto\m^{(1)}$ de $\MS$ 
dans $\Dive(\Cc)$.
Le lemme suivant est immédiat. 

\begin{lemm}
On a $\m^{(1)}=\m$ si et seulement si $\m^-=0$,
\ie si et seulement si $\m$ est une somme de segments de longueur $1$.
\end{lemm}

On définit maintenant par récurrence des multisegments $\m^{(1)},\m^{(2)},\dots$ 
associés à $\m$ en posant $\m^{(i+1)}=(\m^-)^{(i)}$ pour tout $i\>1$.

\begin{lemm}
\label{Digweed}
L'application $\m\mapsto(\m^{(1)},\m^{(2)},\dots)$ est injective. 
\end{lemm}

\begin{proof}
Soient $\m$ et $\m'$ des multisegments tels que $\m^{(i)}=\m'^{(i)}$ 
pour tout $i\>1$.
Alors les multisegments $\m$ et $\m'$ ont la même longueur, notée $n$.
On remarque aussi qu'ils possèdent le même nombre de segments, noté $r$.
On procède par récurrence sur $n$, le cas où $n=1$ étant immédiat. 
En considérant les égalités $\m^{(i)}=\m'^{(i)}$ 
pour tout $i\>2$, on obtient $\m^-=\m'^-$ par hypothèse de récurrence. 
On déduit du lemme \ref{Silas} qu'il existe $\nn$ et $\nn'$ des multisegments 
sommes de segments de longueur $1$ tels que $\m+\nn=\m'+\nn'$.
On en déduit que $\m^{(1)}+\nn^{(1)}=\m'^{(1)}+\nn'^{(1)}$.
Mais on a aussi $\m^{(1)}=\m'^{(1)}$, ce dont on déduit que $\nn=\nn'$,
puis que $\m=\m'$.
\end{proof}

On introduit maintenant une relation $\vdash$ sur $\MS$.  
Que le lecteur prenne garde au fait que ce n'est pas une relation d'ordre. 

\begin{defi}
\label{nnnuevo} 
Soit $\m$ un multisegment, qu'on écrit sous la forme \eqref{Huffam}, et soit 
$\nn$ un autre multisegment. 
Si $\nn$ est de la forme~: 
\begin{equation*}
\left[a_1,c_1\right]_{\rho_1}+\dots+\left[a_r,c_r\right]_{\rho_r}
\end{equation*}
avec $c_i \in \{b_i-1, b_i\}$ pour tout $i$, on écrit $\m\vdash\nn$ et 
on pose $\delta(\m,\nn)=\deg(\m)-\deg(\nn)$. 
\end{defi}

\begin{rema}
\label{remannnuevo} 
\begin{enumerate}
\item 
Pour tout $\m\in\MS$, on a $\m\vdash\m^-$ et 
$\delta(\m,\m^-)=\deg(\m^{(1)})$. 
\item 
Si $\m \vdash \nn$, alors on a $\m^- \vdash \nn^-$ et
$\delta(\m,\nn)-\delta(\m^-,\nn^-)=\deg(\m^{(1)})-\deg(\nn^{(1)})$. 
\end{enumerate}
\end{rema}

\subsection{Multisegments supercuspidaux et apériodiques}
\label{scap}
\label{bijection_multi}

Dans ce paragraphe on étend au cas non deployé les définitions de 
\cite[V.6]{Vig2}. 
Un multi\-seg\-ment $\m$ est une \textit{période} s'il est de la forme~:
\begin{equation}
\label{cycle}
\m=[a,b]_\rho+[a+1,b+1]_\rho+\dots+[a+n-1,b+n-1]_\rho,
\end{equation}
avec $\rho$ une représentation irréductible cuspidale, $a,b\in\ZZ$ 
et $n=\ee(\rho)\ell^r$ pour $r\>0$.

\begin{defi}
Un multisegment est dit \textit{apériodique} s'il ne contient pas de pério\-de
et \textit{supercuspidal} si son support est formé de représentations 
supercuspidales. 
\end{defi}

Soit $\Delta$ un segment. 
D'après le théorème \ref{AppCuspSuper}, on peut l'écrire 
$\Delta=[a,b]_{\St(\rho,n)}$, où $\rho$ est une représentation 
irréductible supercuspidale et $n=1$ ou $n=\ee(\rho)\ell^r$ pour $r\>0$.
La représentation $\rho$ n'est pas unique, mais le multisegment~:
\begin{equation*}
\Delta_\sc=[a,b]_\rho+[a+1,b+1]_\rho+\dots+[a+n-1,b+n-1]_\rho
\end{equation*}
ne dépend pas du choix de $\rho$.
Ce procédé définit par additivité une application~:
\begin{equation}
\label{supercuspidalis}
\m\mapsto\m_\sc
\end{equation}
de $\MS$ vers l'ensemble $\MS^\sc$ des multisegments 
supercuspidaux, qui est l'indentité sur $\MS^\sc$. 
Elle est donc surjective, mais pas injective en général.
Cependant, on a le résultat suivant.



\begin{lemm}
\label{unimap}
Soit $\m$ un multisegment supercuspidal.
Il y a un unique multisegment apério\-dique $\mathfrak{a}$ tel que 
$\mathfrak{a}_\sc=\m$.
\end{lemm}

\begin{proof}
On peut se ramener au cas de multisegments dont l'image par 
\eqref{supercuspidalis} a un support dans 
$\Dive(\ZZ_\rho)$ avec $\rho$ super\-cuspidale.  
Pour alléger les notations, on pose $\St_r(\rho)=\St(\rho,\ee(\rho)\ell^r)$ pour $r\>0$.
Un multisegment apériodique s'écrit~:
\begin{equation*}
\mathfrak{a}=
\sum\limits_{r\>0}\sum\limits_{n\>1} u_{r,n}\cdot[1,n]_{\St_r(\rho)}+\mathfrak{a}_0
\end{equation*}
où $\mathfrak{a}_0$ est la partie supercuspidale de $\mathfrak{a}$, 
\ie le plus grand multisegment $\<\mathfrak{a}$ dont le support soit 
supercuspidal, et où $u_{r,n}\>0$ est la multiplicité du segment 
$[1,n]_{\St_r(\rho)}$ dans 
$\mathfrak{a}_{1}=\mathfrak{a}-\mathfrak{a}_0$. 
(On remarque que pour $r\>0$ la représentation cuspidale $\St_r(\rho)$ 
n'est pas supercuspidale.
Ceci implique que $\e(\St_r(\rho))=1$, donc que tout 
segment de longueur $n\>1$ de la forme $[a,b]_{\St_r(\rho)}$ est équivalent 
à $[1,n]_{\St_r(\rho)}$.)
L'hypothèse d'apériodicité implique que, pour chaque $r\>0$ et chaque 
$n\>1$, on a $u_{r,n}<\ell$.
On écrit maintenant~:
\begin{equation*}
\m=\mathfrak{a}_\sc=
\sum\limits_{r\>0}\sum\limits_{n\>1} \Big(\ell^{r}u_{r,n}\cdot
\sum\limits_{j=0}^{\ee(\rho)-1} [j,j+n-1]_{\rho}\Big)
+\mathfrak{a}_0.
\end{equation*}
Pour $n\>1$ et $a\in\ZZ$, soient $v_{n,a}$ et $w_{n,a}$ les multiplicités 
du segment $[a,a+n-1]_{\rho}$ dans $\m$ et dans $\mathfrak{a}_0$ 
respectivement. 
Si $\e(\rho)=1$, alors pour $n\>1$ et $a\in\ZZ$ on trouve~:
\begin{equation*}
v_{n,a} = w_{n,a} + \sum\limits_{r\>0}\ell^{r+1}u_{r,n}
\end{equation*}
avec la condition $w_{n,a}<\ell$ 
provenant de l'apériodicité du multisegment $\mathfrak{a}_0$.
Par unicité du déve\-lop\-pement $\ell$-adique de $v_{n,a}$,
on retrouve les $w_{n,a}$ et les $u_{r,n}$ à partir de $\m$.
Si $\e(\rho)>1$, alors pour $n\>1$ et $a\in\ZZ$ on trouve~:
\begin{equation*}
v_{n,a} = w_{n,a} + \sum\limits_{r\>0}\ell^{r}u_{r,n},
\end{equation*}
avec la condition $w_{n,b}=0$ pour au moins un $b\in\ZZ$,
provenant de l'apériodicité de $\mathfrak{a}_0$.
Pour un tel $b$, l'unicité du développement $\ell$-adique de $v_{n,b}$ 
permet de retrouver les $u_{r,n}$ à partir de $\m$.
Puis on en déduit la valeur de $w_{n,a}$ pour tout $a\in\ZZ$. 
\end{proof}

En d'autres termes, 
la restriction de l'application \eqref{supercuspidalis} à l'ensemble 
$\MS^\ap$ des multi\-segments apériodiques induit une bijection de 
$\MS^\ap$ vers $\MS^\sc$.
On note $\m_\ap$ l'apériodisé d'un multi\-segment $\m$,
\ie l'antécédent de $\m_\sc$ par cette bijection.
Les applications $\m\mapsto\m_\sc$ et $\m\mapsto\m_\ap$ sont des 
bijections réciproques l'une de l'autre entre les ensembles $\MS^\ap$ et $\MS^\sc$.

Étant donné $\ss\in\Dive\(\Cc\)$, on pose~:
\begin{equation}
\MS(\ss)=\{\m\in\MS\ |\ \supp(\m)=\ss\}
\end{equation}
et on note $\MS(\ss)^\ap$ l'intersection $\MS^\ap\cap\MS(\ss)$.
On a le lemme suivant. 

\begin{lemm}
\label{limon}
On a les propriétés suivantes. 
\begin{enumerate}
\item
Pour tout $\m\in\MS^{\sc}$, on a $\supp(\m)_\ap=\supp(\m_\ap)$.
\item
Pour tout $\m\in\MS^{\ap}$, on a $\supp(\m)_\sc=\supp(\m_\sc)$.
\item 
Pour tout support supercuspidal $\ss\in\Dive\(\Ss\)$, 
l'application $\m\mapsto\m_\ap$ induit une bijection~: 
\begin{equation}
\label{bijj2}
\MS(\ss)\to\coprod\limits_{\tt}\MS(\tt)^\ap
\end{equation}
où $\tt$ décrit les $\tt\in\Dive(\Cc)$ tels que $\tt_\sc=\ss$. 
\end{enumerate}
\end{lemm}

\begin{proof}
On passe de (1) à (2) en appliquant les 
bijections $\m\mapsto\m_\sc$ et $\m\mapsto\m_\ap$ entre 
$\MS^\ap$ et $\MS^\sc$.
Il suffit donc de prouver (1).
Il suffit de le prouver lorsque $\m$ est une période, auquel cas le résultat 
est immédiat. 
Pour prouver (3), il suffit de prouver que l'application \eqref{bijj2} est 
bien définie, ce qui découle de  (1).  
\end{proof}

\subsection{Classification des représentations résiduellement non 
  dégénérées}
\label{SST}

Dans ce paragraphe, on classe les représentations résiduellement non
dégénérées (voir aussi la proposition \ref{equivSt}) en termes de 
multisegment apériodiques.  
Ceci permet de donner quelques propriétés de
ces représentations, qui seront utiles par la suite. 

On note $\Dive(\Cc)^\ap$ l'image de $\Dive(\Ss)$ par la bijection 
$\m\mapsto\m_\ap$ de $\MS^\sc$ dans $\MS^\ap$, 
\ie l'ensemble des supports cuspidaux apériodiques. 

\begin{theo}
\label{equivSt3}
Soit $\pi$ une représentation résiduellement non dégénérée. 
Il y a un unique multisegment $\m=\Delta_1+\dots+\Delta_r$ 
tel que, pour tous $i\neq j$, les segments $\Delta_i,\Delta_j$ 
ne soient pas liés et tel que $\pi$ soit isomorphe à
$\L(\Delta_1) \times \dots \times \L(\Delta_r)$. 
\end{theo}

\begin{proof}
Soit $\ss=\sy{\tau_1}+\dots+\sy{\tau_n}\in\Dive(\Ss)$ le support supercuspidal 
de $\pi$ et soit $\tt=\ss_\ap$ le support cuspidal apériodique lui correspondant,
qu'on écrit $\sy{\rho_1}+\dots+\sy{\rho_t}$ avec $t\>1$. 
Grâce à la condi\-tion d'apériodicité, il existe des segments 
$\Delta_1,\dots,\Delta_r$ de la forme $\Delta_i=\left[a_i,b_i\right]_{\rho_i}$
satisfai\-sant aux conditions suivantes~:
\begin{enumerate}
\item
on a $\tt=\supp(\Delta_1)+\dots+\supp(\Delta_r)$~;
\item
pour tous $i\neq j$, les segments $\Delta_i$ et $\Delta_j$ ne sont pas liés~;
\item
pour tout $i$, la longueur de $\Delta_i$ est strictement inférieure à $\ee(\rho_i)$.
\end{enumerate}
(On raisonne par récurrence sur $t$, le cas où $t=1$ étant immédiat. 
On choisit un segment $\Delta_1$ dont le support est inclus dans 
$\tt=\sy{\rho_1}+\dots+\sy{\rho_t}$ et de longueur maximale pour 
cette propriété. 
Puis on applique l'hypothèse de récurrence à $\tt-\supp(\Delta_1)$.) 

Ces segments sont uniques à l'ordre près. 
D'après le théorème \ref{nuevo2}, la représentation~:
\begin{equation*}
\Pi=\L(\Delta_1)\times\dots\times\L(\Delta_r)
\end{equation*}
est irréductible, et elle est résiduellement non dégénérée
d'après la proposition \ref{prpisegK}\eqref{prpisegK2}. 
Comme $\rho_1\times\dots\times\rho_r$ est un 
sous-quotient de $\tau_1\times\dots\times\tau_n$, on 
déduit de la proposition \ref{SQIRND1}
que $\pi$ est isomorphe à $\Pi$. 
\end{proof}

On en déduit le résultat suivant. 

\begin{prop}
\label{equivSt2}
L'application $(\rho_1,\dots,\rho_r)\mapsto\St(\rho_1,\dots,\rho_r)$ 
est une bijection de $\Dive(\Cc)^\ap$ dans $\Rnd$, notée 
$\tt\mapsto\St(\tt)$, et sa réciproque est donnée par $\pi\mapsto\cusp(\pi)$.  
\end{prop}

\begin{proof}
Si $\pi$ est une représentation irréductible résiduellement non dégénérée, 
on pose $\tt=\scusp(\pi)_\ap$ et on note $\m$ le mutisegment donné par 
le théorème \ref{equivSt3}, de sorte que $\tt$ est égal au support de $\m$. 
Par construction, $\pi$ est égale à $\St(\tt)$. 
Par définition, le support cuspidal de $\L(\Delta_i)$ est égal au support de 
$\Delta_i$.
Le théorème \ref{equivSt3} implique donc que le support cuspidal de $\pi$ 
est égal à $\supp(\m)=\tt$. 
\end{proof}

En comparant \cite[Theorem V.7]{Vig2} à la proposition \ref{equivSt2}, 
on obtient le résultat suivant. 

\begin{coro}
\label{coroSt}
Soit $\pi$ une représentation irréductible de $\GL_n(\F)$.  
Alors~:
\begin{enumerate}
\item 
$\pi$ est non dégénérée si et seulement si $\pi$ est résiduellement non dégénérée.
\item 
Si $\mu$ est une partition de $n$,
alors $\pi$ est $\mu$-dégénérée \cite[V.5]{Vig2} si et seulement si elle est 
résiduellement $\mu$-dégénérée.  
\end{enumerate}
\end{coro}

\begin{proof}
Compte tenu de \cite[V.7]{Vig2} et de la remarque \ref{L=St}, il suffit de 
prouver que, 
pour toute représentation irréductible cuspidale $\rho$ et pour 
tout entier $n\<\ee(\rho)-1$, la représentation $\St(\rho,n)$ est 
isomorphe à l'unique sous-quotient irréductible non dégénéré de l'induite~:
\begin{equation*}
\rho\times\rho\nu_\rho^{}\times\dots\times\rho\nu_\rho^{n-1}. 
\end{equation*}
Ceci découle de \cite[III.5.13]{Vig1}. 
\end{proof}
Le corollaire suivant est une conséquence de la proposition \ref{ZLDdual}. 

\begin{coro}
\label{coroSt0}
\begin{enumerate}
\item 
Pour tout support cuspidal $\tt \in \Dive(\Cc)$, on a $\St(\tt^\vee)\simeq 
\St(\tt)^\vee$. 
\item 
Soit $\mu$ une partition de $n$.  
Alors une représentation est résiduellement $\mu$-dégénérée si 
et seulement si sa contragrédiente est résiduellement $\mu$-dégénérée. 
\end{enumerate}
\end{coro}

En particulier, une représentation irréductible est non dégénérée si 
et seulement si sa contragrédiente est résiduellement non 
dégénérée.  

\subsection{La partition $\mu_\m$ et la représentation 
  $\St_{\overline{\mu}_\m}(\m)$} 
\label{ModWhitRes}

Dans ce paragraphe, on associe à tout multisegment $\m$ une 
famille croissante d'entiers $\overline{\mu}_\m$ de somme $\deg(\m)$ 
et une représentation irréductible résiduellement non dégénérée du 
sous-groupe de Levi $\M_{\overline{\mu}_\m}$, que l'on note 
$\St_{\overline{\mu}_\m}(\m)$.  

\subsubsection{}

Soit $\m$ un multisegment de degré noté $m$.
On note ${\rm Part}(\m)$ 
l'ensemble des partitions $\mu$ de $m$ telles que l'induite~:
\begin{equation}
\label{eq:1KL}
\I(\m)=\Z(\Delta_1)\times\dots\times\Z(\Delta_r)
\end{equation}
possède un sous-quotient irréductible $\mu$-dégénéré 
(définition \ref{LaisserAller}). 
Il n'est pas vide, 
parce qu'on peut choisir $\mu$ de sorte que 
$\rp_\mu(\I(\m))$ contienne un sous-quotient irréductible cuspi\-dal, 
qui est rési\-duel\-lement non dégénéré d'après le lemme \ref{SQIRND1}.

\begin{defi}
\label{DefMum1}
On pose~:
\begin{equation}
\label{DefMum}
\mu_{\m}=(\deg(\m^{(1)})\>\deg(\m^{(2)})\>\dots).
\end{equation}
C'est une partition de $m$
dont la conjuguée est la partition associée à 
$(\deg(\Delta_1),\deg(\Delta_2),\dots)$, 
la famille des degrés des segments composant $\m$.
\end{defi}
\begin{rema}
\label{remavee}
Par construction, on a $\mu_{\m^\vee}=\mu_{\m}$ pour tout multisegment $\m$.
\end{rema}

On remarque que tous les multisegments $\m^{(i)}$ sont des sommes de 
segments de longueur $1$, de sorte que, d'après la proposition \ref{SQIRND1}, 
chacune des induites $\I(\m^{(i)})$ contient le sous-quotient irréductible 
résiduellement non dégénéré $\St(\m^{(i)})$.

\begin{defi}
Soit $t$ le plus grand entier tel que $\m^{(t)} \neq 0$.
On note
$\overline{\mu}_{\m}$ la famille croissante d'entiers associée à la partition $\mu_\m$, \ie:
\begin{equation}\label{ovnu}
\overline{\mu}_{\m}=(\deg(\m^{(t)})\<\deg(\m^{(t-1)})\<\dots\<\deg(\m^{(1)})),
\end{equation}
et on note~:
\begin{equation}
\label{quotdeg}
\St_{\overline{\mu}_\m}(\m)= \St(\m^{(t)}) \otimes \St(\m^{(t-1)}) \otimes  \dots \otimes \St(\m^{(1)}),
\end{equation}
qui est une représentation irréductible résiduellement non dégénérée de 
$\M_{\overline{\mu}_{\m}}$. 
\end{defi}

\begin{lemm}
\label{hola}
La représentation $\St_{\overline{\mu}_\m}(\m)$ est l'unique sous-quotient
irréductible rési\-duelle\-ment non dégénéré de $\rp_{\overline{\mu}_\m}(\I(\m))$.
\end{lemm}

\begin{proof}
On va le prouver par récurrence sur la longueur $n$ de $\m$. 
Si $n=1$, le résultat est immédiat puisqu'alors 
$\I(\m)$ est irréductible cuspidale.
On suppose maintenant que $n\>2$ et on pose 
$\a=(m-\deg(\m^{(1)}),\deg(\m^{(1)}))$. 
D'après le lemme géométrique (\S\ref{lemmegeo}) 
et la proposition \ref{prpiseg}, 
le module de Jacquet $\rp_{\a}(\I(\m))$ est composé de représentations de la forme~: 
\begin{equation*}
\I(\m_1)\otimes\I(\m_2)
\end{equation*}
avec $\m_1,\m_2$ des multisegments de la forme~:
\begin{equation*}
\m_1=\sum\limits_{i=1}^{r}[a_i,c_i]_{\rho_i},
\quad 
\m_2=\sum\limits_{i=1}^{r}[c_i+1,b_i]_{\rho_i},
\quad
a_i-1\<c_i\< b_i,
\quad
\deg(\m_2)=\deg(\m^{(1)}).
\end{equation*}
D'après le corollaire \ref{MrRubb}, $\I(\m_2)$ contient une
représentation résiduellement non dégénérée si et seulement si
pour tout $i$, on a $c_i=b_i-1$, \ie si $\m_1=\m^-$ et $\m_2=\m^{(1)}$.
Par hypothèse de récurrence, 
la représentation $\St_{\overline{\mu}_{\m^-}}(\m^-)$ est l'unique 
sous-quotient rési\-duel\-lement non dégénéné de 
$\rp_{\overline{\mu}_{\m^-}}(\I(\m^-))$. 
On remarque que~:
\begin{eqnarray*}
\overline{\mu}_{\m^-}&=&(\deg(\m^{(t)})\<\deg(\m^{(t-1)})\<\dots\<\deg(\m^{(2)})),\\
\St_{\overline{\mu}_{\m^-}}(\m^-)&=&\St(\m^{(t)}) \otimes \St(\m^{(t-1)}) \otimes \dots \otimes \St(\m^{(2)}).
\end{eqnarray*}
On déduit le résultat par transitivité des foncteurs de Jacquet.
\end{proof}

On renvoie à la définition \ref{nnnuevo} et au paragraphe \ref{Lasso}
pour les définitions de $\vdash$ et de la relations d'ordre 
$\trianglelefteq$ sur l'ensemble des partitions.  

\begin{lemm}
\label{llleme}
Soient $\m, \nn$ deux multisegments tels que $\m\vdash\nn$. 
Alors la partition $\mu_{\m}$ est plus grande que la partition 
associée à la famille $(\delta(\m,\nn),\mu_{\nn})$.
\end{lemm}

\begin{proof}
On veut montrer que, pour tout $k\>1$, on a~: 
\begin{eqnarray}
\label{437}
\underset{i=1}{\overset{k}{\sum}}\deg\(\nn^{(i)}\)\< 
\underset{i=1}{\overset{k}{\sum}}\deg\(\m^{(i)}\)
& \text{ si }\delta(\m,\nn) \< \deg\(\nn^{(k)}\), \\ 
\delta(\m,\nn)+
\underset{i=1}{\overset{k-1}{\sum}}\deg\(\nn^{(i)}\) 
\<\underset{i=1}{\overset{k}{\sum}}\deg\(\m^{(i)}\)
& \text{ si } \delta(\m,\nn)> \deg\(\nn^{(k)}\).  
\end{eqnarray}
La première inégalité découle du fait que, pour tout $i\> 1$, 
on a $\deg(\m^{(i)})\> \deg(\nn^{(i)})$.  
Prouvons la seconde inégalité par récurrence sur la longueur 
$n$ de $\m$.  
Si $k=1$, on a $\deg(\m^{(1)}) \> \delta(\m,\nn)$. 
Si $k\>2$, d'après la remarque \ref{remannnuevo}, l'inégalité 
\eqref{437} équivaut à~: 
$$\underset{i=2}{\overset{k}{\sum}}\deg(\m^{(i)}) \> 
\delta(\m,\nn)
+ \underset{i=2}{\overset{k-1}{\sum}}\deg(\nn^{(i)}), $$ 
c'est-à-dire~:
$$\underset{i=1}{\overset{k-1}{\sum}}\deg(\m^{-(i)}) \> 
\delta(\m,\nn)
+\underset{i=1}{\overset{k-2}{\sum}}\deg(\nn^{-(i)}), $$ 
ce qui est vrai par hypothèse de récurrence.
\end{proof}

\begin{prop}
\label{ejemplofund}
La partition $\mu_{\m}$ est le plus grand élément de ${\rm Part}(\m)$, 
et $\I(\m)$ possède un unique sous-quotient irréductible résiduellement
$\mu_\m$-dégénéré.  
Il apparaît avec multiplicité $1$ comme sous-quotient dans $\I(\m)$. 
\end{prop}

\begin{rema}
L'ordre $\trianglelefteq$ sur les partitions n'étant pas total, il n'est pas évident
\textit{a priori} que ${\rm Part}(\m)$ possède un plus grand
élément.
Ainsi l'assertion de la proposition \ref{ejemplofund} est triple~: 
\begin{enumerate}
\item
l'ensemble ${\rm Part}(\m)$ admet un plus grand élément 
(automatiquement unique)~;
\item
ce plus grand élément est $\mu_{\m}$
(ce qui implique en particulier que $\mu_{\m}\in{\rm Part}(\m)$)~; 
\item
le module de Jacquet $\rp_{\mu_\m}(\I(\m))$ possède un unique 
sous-quotient irréductible résiduelle\-ment non dégénéré avec multiplicité $1$. 
\end{enumerate}
\end{rema}

\begin{proof}
La partition $\mu_{\m}$ appartient à ${\rm Part}(\m)$ d'après 
le lemme \ref{hola} et la défi\-nition \ref{LaisserAller}.  
Ce lemme implique aussi que $\rp_{\mu_\m}(\I(\m))$ 
a un uni\-que sous-quotient irréductible résiduelle\-ment non dégénéré 
avec multiplicité $1$.  
Soit maintenant $\nu=(\nu_1\>\nu_2\>\dots)\in{\rm Part}(\m)$.  
On va montrer, par récurrence sur $m$, que l'on a $\nu\trianglelefteq\mu_{\m}$.  

D'après le lemme géométrique et la proposition \ref{prpiseg}, la représentation 
$\rp_{(m-\nu_1,\nu_1)}(\I(\m))$ est soit nulle, soit composée de 
représentations de la forme~: 
\begin{equation*}
\label{produ} 
\I(\m_1)\otimes\I(\m_2)
\end{equation*}
avec $\m_1,\m_2$ des multisegments de la forme~:
\begin{equation*}
\m_1=\sum\limits_{i=1}^{r}[a_i,c_i]_{\rho_i},
\quad 
\m_2=\sum\limits_{i=1}^{r}[c_i+1,b_i]_{\rho_i},
\quad
a_i-1\<c_i\< b_i,
\quad
\deg(\m_2)=\nu_1.
\end{equation*}
D'après le corollaire \ref{MrRubb}, 
l'induite $\I(\m_2)$ contient une
représentation résiduellement non dégénérée si et seulement si
pour tout $i$, on a $c_i=b_i$ ou $c_i=b_i-1$, \ie si $\m\vdash\m_1$. 
Puisque $\nu\in \Pi(\m)$, en écrivant $\nu^-=(\nu_2\>\dots)$, il existe 
(par transitivité du foncteur de Jacquet) un 
sous-quotient irréductible de $\rp_{(m-\nu_1,\nu_1)}(\I(\m))$ 
de la forme $\I(\m_1)\otimes\I(\m_2)$  
où $\I(\m_2)$ contient un sous-quotient irréductible 
résiduellement non dégénéré et où $\I(\m_1)$ 
contient un sous-quotient irréductible 
résiduellement $\nu^-$-dégénéré. 
La condition sur $\I(\m_1)$ entraîne par hypothèse de récur\-rence 
qu'on a $\nu^-\trianglelefteq\mu_{\m_1}$.  
Comme $\nu_1=\delta(\m,\m_1)$, on en déduit que 
$\nu$ est plus petite que la partition associée à 
$(\delta(\m,\m_1),\mu_{\m_1})$. 
Finalement, on déduit du lemme \ref{llleme} que $\nu\trianglelefteq\mu_\m$. 
\end{proof}

\subsubsection{} 
\label{OccamG}

Étant donnés deux multisegments $\m, \m'$, on va donner dans ce paragraphe des
conditions nécessaires et suffisantes pour que $\St_{\overline{\mu}_\m}(\m)$ 
et $\St_{\overline{\mu}_{\m'}}(\m')$ soient isomorphes. 

\begin{lemm}
\label{injSt}
Soient $\m,\m'$ deux multisegments supercuspidaux.  
Si $\St_{\overline{\mu}_\m}(\m)$ et $\St_{\overline{\mu}_{\m'}}(\m')$ sont 
isomorphes, alors $\m=\m'$. 
\end{lemm}

\begin{proof}
L'hypothèse implique d'une part que $\overline{\mu}_\m=\overline{\mu}_{\m'}$, 
d'autre part que $\St(\m^{(i)})$ et 
$\St(\m^{\prime(i)})$ sont isomorphes pour tout $i \> 1$. 
D'après la proposition \ref{equivSt}, ceci entraîne que 
$\m^{(i)}=\m^{\prime(i)}$ pour tout $i \> 1$. 
Le lemme \ref{Digweed} implique alors que $\m=\m'$.
\end{proof}

\begin{lemm}
\label{parties}
Pour tout multisegment $\m$, on a $\mu_\m=\mu_{\m_\sc}$ et
$\St_{\overline{\mu}_\m}(\m) \simeq\St_{\overline{\mu}_{\m_\sc}}(\m_\sc)$. 
\end{lemm}

\begin{proof}
Remarquons que, par additivité de l'application $\m \mapsto \m^{(i)}$, 
on a~:
\begin{equation}
\label{addi}
\mu_{\m+\nn}=\mu_{\m}+\mu_{\nn}
\end{equation}
pour tous multisegments $\m$ et $\nn$.
D'après la proposition \ref{prodSt}, il suffit de prouver 
le lemme lorsque $\m$ est une période, auquel cas le résultat est immédiat. 
\end{proof}

On en déduit la proposition suivante.

\begin{prop}\label{fibresSt}
Soient $\m,\m'$ deux multisegments.  
Alors $\St_{\overline{\mu}_\m}(\m)$ et $\St_{\overline{\mu}_{\m'}}(\m')$ sont 
isomorphes si et seulement si $\m^{}_\sc=\m'_\sc$. 
\end{prop}

\subsection{La représentation $\Z(\m)$}
\label{ccll}

Dans cette section, on associe à tout $\m\in\MS$ une 
représentation irré\-duc\-tible $\Z(\m)$. 

\subsubsection{}

La proposition \ref{ejemplofund} nous permet de donner la définition suivante.

\begin{defi}
\label{DefiZ}
Soit $\m=\Delta_1+\dots+\Delta_r$ un multisegment.  
Il y a un unique sous-quotient irréductible de~:
\begin{equation}
\I(\m)=\Z(\Delta_1)\times\dots\times\Z(\Delta_r),
\end{equation}
noté $\Z(\m)$, qui soit résiduellement $\mu_{\m}$-dégénéré 
(voir \eqref{DefMum} et la définition \ref{LaisserAller}).
Il apparaît avec multiplicité $1$ dans $\I(\m)$ 
et l'unique sous-quotient résiduellement non dégé\-né\-ré de 
$\rp_{\overline{\mu}_\m}(\Z(\m))$ est $\St_{\overline{\mu}_\m}(\m)$. 
\end{defi}

Ceci découle de l'exactitude du foncteur de Jacquet et de 
la propriété de multiplicité $1$ de la proposition 
\ref{ejemplofund}. 

\begin{rema}
D'après le corollaire \ref{coroSt}, 
notre définition coïncide avec celle de Vignéras \cite[V.9.2]{Vig2}
dans le cas où $\D=\F$. 
\end{rema}

\begin{exem}
Si $\tt\in \Dive(\Cc)$ est un support cuspidal, alors 
$\Z(\tt)\simeq\St(\tt)$. 
\end{exem}

\begin{prop}
Pour tout multisegment $\m$, on a $\Z(\m)^\vee\simeq\Z(\m^\vee)$. 
\end{prop}

\begin{proof}
Par passage à la contragrédiente, et puisque $\Z(\m)$ est un sous-quotient 
irré\-ductible de $\I(\m)$, on trouve que $\Z(\m)^\vee$ est un sous-quotient 
irréductible de $\I(\m)^\vee$.  
Or, d'après la proposition \ref{ZLDdual}, les représentations $\I(\m)^\vee$ et 
$\I(\m^\vee)$ sont isomorphes.
D'après le corollaire \ref{coroSt0} et la remarque \ref{remavee}, la 
représentation $\Z(\m)^\vee$ est résiduellement $\mu_\m$-dégénérée, 
donc résiduellement $\mu_{\m^\vee}$-dégénérée. 
On en déduit que $\Z(\m)^\vee$ est isomorphe à $\Z(\m^\vee)$.
\end{proof}

\begin{prop}\label{propprod}
Soient $\m,\nn$ des multisegments.
Alors la représentation $\Z(\m+\nn)$ 
est un sous-quotient de $\Z(\m)\times\Z(\nn)$.
\end{prop}

\begin{proof}
D'après le lemme \ref{prodSt} et \eqref{addi}, l'induite 
$\Z(\m) \times \Z(\nn)$ contient un sous-quotient $\pi$ 
résiduellement $\mu_{\m+\nn}$-dégénéré.  
Comme $\Z(\m) \times \Z(\nn)$ est un sous-quotient de $\I(\m+\nn)$, 
il en est de même pour $\pi$. 
Elle est donc par définition isomorphe à $\Z(\m+\nn)$. 
\end{proof}

\subsubsection{}

Soit $m\>1$, soient $\rho$ une représentation irréductible
cuspidale de $\G_m$ et $\k\otimes\sigma$ un type simple maximal 
contenu dans $\rho$. 
On forme l'homomorphisme $\KMS_{n}$ et on rappelle que $\Om_\rho$ 
désigne la classe d'inertie de $\rho$. 
On a la caractérisation suivante de 
$\Z(\m)$ dans le cas où $\supp(\m)\in\Dive(\Om_\rho)$.

\begin{prop}
Soit $\m$ un multisegment de longueur $n$ 
dont le support est dans $\Dive(\Om_\rho)$.
Alors $\Z(\m)$ est l'unique sous-quotient irréductible de $\I(\m)$ 
tel que~:
\begin{equation*}
\KMS_{n}(\Z(\m))\>\sy{\z(\sigma,m^{-1}\mu_{\m}')},
\end{equation*}
où $\mu_{\m}'$ est la partition conjuguée de $\mu_{\m}$. 
\end{prop}

\begin{proof}
On note $n_i$ la longueur de $\Delta_i$, de sorte que la partition 
des $n_i$ est égale à $m^{-1}\mu_{\m}'$.
D'après les propositions \ref{KMAX1} et \ref{prpisegK}, on a~:
\begin{equation}
\label{eq:1RMF}
\KMS_{n}(\I(\m))=\z(\s,n_1)\times\dots\times\z(\s,n_r).
\end{equation}
D'après le théorème \ref{ResumeJames}, la représentation 
$\z(\sigma,m^{-1}\mu_{\m}')$ est l'unique sous-quotient irré\-duc\-tible 
$\mu_{\m}$-dégénéré de \eqref{eq:1RMF} et il y apparaît 
avec multiplicité $1$. 
\end{proof}

\subsection{Classification des représentations irréductibles}
\label{ccll2}

Dans ce paragraphe, on étudie l'application~:
\begin{equation*}
\Z: \MS \rightarrow \Irr.
\end{equation*}
On prouve qu'elle est surjective et que sa restriction à $\MS^\sc$
est injective.  
Puis on étudie ses fibres et on calcule les supports cuspidal et 
supercuspidal de $\Z(\m)$ en fonction de $\m$.

\subsubsection{}

On montre d'abord l'injectivité de la restriction de $\Z$ à $\MS^\sc$.

\begin{theo}
\label{inje}
Soient $\m,\m'$ deux multisegments supercuspidaux. 
Alors les représen\-ta\-tions $\Z(\m)$ et $\Z(\m')$ 
sont isomorphes si et seulement si $\m=\m'$.  
\end{theo}

\begin{proof}
Comme $\mu_\m$ est l'unique plus grande partition pour laquelle 
$\Z(\m)$ soit rési\-duel\-lement dégénérée et comme $\St_{\overline{\mu}_\m}(\m)$ 
est l'unique sous-quotient irréductible rési\-duel\-lement non dégénéré 
de $\rp_{\overline{\mu}_\m}(\Z(\m))$, l'hypothèse implique d'une part que 
$\mu_\m=\mu_{\m'}$, et d'autre part que $\St_{\overline{\mu}_\m}(\m)$ et 
$\St_{\overline{\mu}_{\m'}}(\m')$ sont isomorphes.
Il ne reste plus qu'à appliquer le lemme \ref{injSt}.
\end{proof}

Si $\m$ est un multisegment supercuspidal, alors 
$\scusp(\Z(\m))$ est égal au support de $\m$.  
Quel est alors son support cuspidal~?  
La réponse à cette question est donnée dans la proposition \ref{cuspZ}.  

\begin{lemm}
\label{lemmprod}
Soit $\m\in\MS^\sc$. 
On suppose qu'il y a une représentation irré\-duc\-ti\-ble super\-cuspidale 
$\rho$ telle que $\supp(\m)\in\Dive(\ZZ_\rho)$.
S'il existe $r\>0$ et $k\>1$ tels que~:
\begin{equation*}
\cusp\(\Z(\m)\)=k\cdot\St_r(\rho),
\end{equation*} 
alors $\m$ n'est pas apériodique. 
\end{lemm}

\begin{proof}
Comme $\St_{\overline{\mu}_\m}(\m)$ est un sous-quotient non 
dégénéré de $\rp_{\overline{\mu}_\m}(\Z(\m))$ 
et comme 
$[\rp_{\overline{\mu}_\m}(\Z(\m))]\<[\rp_{\overline{\mu}_\m}(\I(\cusp\(\Z(\m)\)))]$, 
il y a une par\-ti\-tion 
$\tau=(k_1\>k_2\>\dots)$ de $k$ telle que, pour tout $i\>1$, on ait~:
\begin{equation*}
\St(\m^{(i)})=\St\(k_i\cdot\St_r(\rho)\).
\end{equation*} 
Notons $\tau'=(k'_1\>k'_2\>\dots\>k'_s)$ la partition conjuguée de $\tau$ et 
posons~:
\begin{equation*}
\m'=[1,k'_1]_{\St_r(\rho)}+[1,k'_2]_{\St_r(\rho)}+\dots+[1,k'_s]_{\St_r(\rho)}.
\end{equation*}
Les partitions $\mu_{\m'}$ et $\mu_\m$ sont égales et les 
représentations $\St_{\overline{\mu}_{\m'}}(\m')$ et 
$\St_{\overline{\mu}_\m}(\m)$ sont isomorphes. 
On déduit de la proposition \ref{fibresSt} que $\m=\m'_\sc$.
Comme $\m'_\sc \neq \m'$, on en déduit que $\m$ n'est pas apériodique. 
\end{proof}

\subsubsection{}
\label{NotZF}

Le résultat suivant donne une classification des représentations irréductibles 
en termes de multisegments.  
Soit $\ZZ_\F$ l'ensemble des caractères non ramifiés de $\mult\F$ de la 
forme $x\mapsto|x|_{\F}^{i}$ avec $i\in\ZZ$.
Pour tout $m\>0$, 
on note $\MS^\ap_m(\ZZ_\F)$ l'ensemble des multisegments 
apériodiques de degré $m$ et de support contenu dans $\Dive(\ZZ_\F)$,
et $\MS^{\sc}_m$ l'ensemble des multisegments supercuspidaux de degré $m$. 

\begin{prop}
\label{bijj} 
Soit $m\>0$ un entier. 
\begin{enumerate}
\item 
L'image de $\MS^\ap_m\(\ZZ_{\F}\)$ par $\m\mapsto\Z(\m)$ est
l'ensemble des classes de représentations ir\-réductibles de $\G_m$ dont le 
support cuspidal est contenu dans $\Dive(\ZZ_{\F})$. 
\item 
L'application $\m\mapsto\Z(\m)$ induit une bijection de 
$\MS^{\sc}_m$ dans $\Irr(\G_m)$ et, 
pour $\m\in \MS^{\sc}_m$, on a
$\scusp(\Z(\m))=\supp(\m)$. 
\item 
Pour tout multisegment $\m$ de degré $m$, on a
$\Z (\m)\simeq\Z(\m_\sc) \simeq\Z(\m_\ap)$. 
\end{enumerate}
\end{prop}

\begin{proof}
On va prouver ce résultat par récurrence sur $m$, le cas $m=1$ étant trivial. 
Supposons donc que la proposition soit vraie pour tout $i<m$ et prouvons-la
au rang $m$.  

On prouve d'abord (1).  
Soit $\m\in\MS^\ap_m$, et supposons que $\Z(\m)$ n'ait pas le support cuspidal 
attendu, \ie qu'il existe des entiers $1\<i\<m$ et $r\>0$, $t\>1$ et des 
représentations 
irréductibles $\pi_1\in\Irr(\G_i)$ et $\pi_2\in\Irr(\G_{m-i})$ tels que~:
\begin{equation*}
\Z(\m)\simeq\pi_1\times\pi_2,
\quad
\cusp(\pi_1)=t\cdot\St_r(1_{\F^\times}).
\end{equation*}
D'après le lemme \ref{lemmprod}, on a $i<m$.
La proposition étant vraie pour $i$ et $m-i$, il existe d'après (2) 
des multisegments supercuspidaux $\m_1$ et $\m_2$ tels qu'on ait 
$\pi_1=\Z(\m_1)$ et $\pi_2=\Z(\m_2)$.
D'après la proposition \ref{propprod} et le théorème \ref{inje}, on a 
$\m=\m_1+\m_2$.  
D'après le lemme \ref{lemmprod}, le multisegment $\m_1$ n'est pas apériodique, 
donc $\m$ non plus~: contradiction. 
Soit maintenant $\ss\in \Dive(\ZZ_{\F})$ de degré $m$. 
On a une application injective~: 
\begin{equation}
\label{INJineg}
\Z:\MS^\ap(\ss) \rightarrow \Irr(\ss)^{\q}
\end{equation}
(rappelons que $\Irr(\ss)^{\q}$ désigne l'ensemble des classes de représentations 
irréductibles de support cuspidal $\ss$).

\begin{lemm}
\label{FinaleAriki}
Les ensembles finis $\MS^\ap(\ss)$ et $\Irr(\ss)^\q$ ont le même cardinal. 
\end{lemm}

\begin{proof}
On utilise les notations du paragraphe \ref{BlaACGM}. 
On écrit $\ss$ sous la forme~:
\begin{equation*}
\sy{\nu^{i_1}}+\dots+\sy{\nu^{i_n}},
\quad i_1,\dots,i_n\in\ZZ,
\end{equation*}
où $\nu$ désigne la valeur absolue normali\-sée de $\mult\F$.
Si l'on note $\xi$ l'image de $q$ dans $\CR$ et si l'on pose 
$\Xi=\sy{\xi^{i_1}}+\dots+\sy{\xi^{i_n}}\in\Dive(\xi^\ZZ)$, 
on peut identifier les ensembles $\Psi(\Xi)$ et $\MS^\ap(\ss)$.
L'application $\V\mapsto\V^\I$, où $\I$ désigne le sous-groupe 
d'Iwahori standard, induit une bijection de 
$\Irr(\ss)^\q$ dans $\Irr(\Hh_N,\Xi)$
(voir \cite[2 et 6]{Vig3}).
L'égalité cherchée est donc une conséquence de l'égalité entre les 
cardinaux respectifs des ensembles $\Psi(\Xi)$ et $\Irr(\Hh_n,\Xi)$, 
qui est donnée par le corollaire \ref{BlaMathas}. 
\end{proof}

Le lemme \ref{FinaleAriki} implique que l'application \eqref{INJineg} est bijective. 

On prouve maintenant (2). 
Il suffit de prouver que, pour tout 
$\ss\in\Dive(\Ss)$ de degré $m$, l'appli\-ca\-tion $\m\mapsto\Z(\m)$ induit 
une bijection de $\MS(\ss)$ dans $\Irr(\ss)$.
Soit donc un support supercuspidal $\ss\in\Dive(\Ss)$.
La restriction de $\Z$ à $\MS(\ss)$ est à valeurs dans $\Irr(\ss)$
d'après la définition \ref{DefiZ}, et elle est injective 
d'après le théorème \ref{inje}. 
Il reste donc à montrer que les ensembles 
$\MS(\ss)$ et $\Irr(\ss)$ ont le même cardinal.
D'après le théorème \ref{ssvarsc} et la proposition \ref{Zdroites}, 
on peut se ramener au cas où $\ss\in\Dive\(\ZZ_\rho\)$, 
avec $\rho$ une repré\-sen\-ta\-tion irréductible supercuspidale.
Grâce à l'unicité du support supercuspidal (voir le théorème \ref{unicitesupp}), 
on a une égalité~:
\begin{equation}
\label{bijj3}
\Irr(\ss)=
\underset{\tt_\sc=\ss}{\coprod\limits_{\tt\in\Dive\(\Cc\)}}\Irr(\tt)^{\q}.
\end{equation}
D'après le lemme \ref{limon}, on a une bijection~:
\begin{equation*}
\MS(\ss)\ \ffr{\simeq}\ 
\underset{\tt_\sc=\ss}{\coprod\limits_{\tt\in\Dive\(\Cc\)}}\MS(\tt)^\ap.
\end{equation*}
Il suffit donc de prouver que, pour tout $\tt\in\Dive(\Cc)$ tel que 
$\tt_\sc=\ss$, les ensembles finis 
$\MS(\tt)^\ap$ et $\Irr(\tt)^{\q}$ ont le même cardinal. 
D'après le théorème \ref{AppCuspSuper}, chaque $\tt\in\Dive\(\Cc\)$ tel que 
$\tt_\sc=\ss$ se décompose sous la forme~:
\begin{equation*}
\tt=\tt_{-1}+\tt_0+\tt_1+\dots+\tt_r,
\quad
r\>-1,
\end{equation*}
où $\tt_{-1}\in\Dive\(\ZZ_\rho\)$ et $\tt_k\in\Dive\(\ZZ_{\St_k(\rho)}\)$
pour $k\>0$, où $\St_k(\rho)$ désigne la repré\-sen\-ta\-tion 
cuspidale $\St(\rho,\ee(\rho)\ell^k)$.
D'après le théorème \ref{ss}, on a une bijection~:
\begin{equation*}
\Irr(\tt)^{\q}\to\Irr(\tt_{-1})^{\q}\times\Irr(\tt_0)^{\q}\times\dots\times\Irr(\tt_r)^{\q},
\end{equation*}
et on a une décomposition analogue~:
\begin{equation*}
\MS(\tt)^\ap=\MS(\tt_{-1})^\ap\times\MS(\tt_0)^\ap\times\dots\times\MS(\tt_r)^\ap,
\end{equation*}
de sorte qu'on peut supposer que $\tt$ est égal à un seul $\tt_k$, par exemple 
$k=-1$ pour simplifier les no\-ta\-tions. 
Il correspond à $\tt$ un support cuspidal 
$\tt'\in\Dive(\ZZ_{\F'})$ où $\F'$ est une extension finie de $\F$
associée à la représentation cuspidale $\rho$, et une bijection~:
\begin{equation*}
\boldsymbol{\Phi}_{\tt}:\Irr(\tt)^{\q}\to\Irr(\tt')^{\q}
\end{equation*}
en vertu de \cite[Proposition 4.33]{MS11} et de la bijection de 
\eqref{BeautyCeleste}. 
On a une bijection analogue~:
\begin{equation*}
\label{bijectiones2}
\MS({\tt})^\ap\to\MS({\tt'})^\ap
\end{equation*}
grâce à la proposition \ref{ZZp}. 
On est donc ramené à prouver l'égalité de cardinaux dans le cas où 
$\rho$ est le caractère trivial de $\F'^{\times}$.
On conclut grâce à la partie (1).  

On prouve finalement (3).  
Soit $\m$ un multisegment de degré $m$.  
D'après la partie (2), il existe $\m'\in \MS^\sc_m$ tel que 
$\Z(\m')\simeq\Z(\m)$, ce qui implique que $\mu_\m$ et $\mu_{\m'}$ 
sont égales et que 
$\St_{\overline{\mu}_{\m}}(\m)$ et $\St_{\overline{\mu}_{\m'}}(\m')$ sont 
isomorphes. 
Le résultat découle alors des lemmes \ref{injSt} et \ref{parties}. 
\end{proof}

\subsubsection{} 

La proposition suivante donne le support cuspidal de 
$\Z(\m)$ en fonction de $\m$. 

\begin{prop}
\label{cuspZ}
Pour tout multisegment $\m$, on a $\cusp(\Z(\m))=\supp(\m_\ap)$.
\end{prop}

\begin{proof}
D'après la proposition \ref{bijj}(3), on peut supposer que $\m=\m_\ap$. 
On va mon\-trer le lemme suivant par récurrence sur la longueur de $\m$. 
On note $\MS^\ap_{(n)}$ l'ensemble des multi\-segments 
apé\-rio\-diques de longueur $n$. 

\begin{lemm}
L'application $\m\mapsto\Z(\m)$ induit une bijection entre $\MS^\ap_{(n)}$ 
et l'ensemble 
des classes de représentations irréductibles de support cuspidal de longueur $n$ et, 
pour tout multi\-segment $\m\in\MS^\ap_{(n)}$, on a $\cusp(\Z(\m))=\supp(\m)$.
\end{lemm}

\begin{proof}
D'après la proposition \ref{bijj} et le paragraphe \ref{bijection_multi}, 
l'application~: 
$$\Z:\MS^\ap\rightarrow \Irr$$
est une bijection. 
Si $n=1$, le résultat est trivial.  
On suppose donc que $n\>2$ et on pose $\tt=\supp(\m)$.
Soient $\rho_1,\dots,\rho_r$ des représentations irréductibles cuspidales 
telles que $\Z(\m)$ soit un quotient de $\rho_1 \times \dots \times \rho_{r}$. 
Du lemme géométrique on déduit que $r\< n$ et, si $r=n$, on a~: 
\begin{equation*}
\tt=\sy{\rho_1}+\dots+\sy{\rho_n},
\end{equation*}
\ie que $\cusp(\Z(\m))=\supp(\m)$.
Supposons donc que $r<n$. 
Par hypothèse de récurrence, il existe un multisegment 
apériodique $\m'$ de longueur $r$ tel que $\Z(\m')=\Z(\m)$, 
ce qui donne une contradiction. 
\end{proof}

Ceci met fin à la preuve de la proposition \ref{cuspZ}. 
\end{proof}

On résume ici les resultats obtenus dans cette section.

\begin{theo}
\label{bijj2thm} 
\begin{enumerate}
\item 
L'application $\Z:\MS\to\Irr$ est surjective.  
\item
Étant donnés $\m,\m'\in\MS$, on a $\Z(\m)\simeq\Z(\m')$ 
si et seulement si $\m_\sc^{}=\m'_\sc$. 
\item
Pour tout $\m\in\MS$, on a
$\scusp(\Z(\m))=\supp(\m_\sc)$ et 
$\cusp(\Z(\m))=\supp(\m_\ap)$.
\item Pour tout $\m \in\MS$, on a $\Z(\m)^\vee\simeq\Z(\m^\vee)$. 
\end{enumerate}
\end{theo}

\begin{rema}
Dans le cas où $\D=\F$, la surjectivité de l'application $\Z$ n'est pas 
prouvée en détail dans \cite{Vig1}.
Notre argument est différent de celui suggéré dans la note de 
\cite[p.~606]{Vig1}.
\end{rema}

\begin{rema}
L'application $\m\mapsto\Z(\m)$ dépend du choix de la racine carrée $\sqrt{q}$ 
effectué au paragraphe \ref{MenelasEtHermione}.
Pour éliminer cette dépendance, on peut renormaliser $\Z$ en posant~:
\begin{equation*}
\widehat{\Z}([a_1,b_1]_{\rho_1}+\dots+[a_r,b_r]_{\rho_r}) = 
\Z([a_1,b_1]_{\rho_1\nu^{(b_1-a_1)/2}}+\dots+[a_r,b_r]_{\rho_r\nu^{(b_r-a_r)/2}}).
\end{equation*}
L'application $\m\mapsto\widehat{\Z}(\m)$ vérifie les propriétés 1, 2 et 4 
du théorème \ref{bijj2thm}, mais pas la propriété 3 telle quelle.
En modifiant convenablement la définition du support d'un multisegment, 
on peut obtenir pour $\widehat{\Z}$ une propriété analogue à 3.
\end{rema}

\subsection{Réduction modulo $\ell$}
\label{rrre}

Soit $\ell$ un nombre premier différent de $p$,
soit un entier $m\>1$ et soit $\G=\G_m$. 

\begin{theo}
\label{reductionsegment1} 
Soit $\tilde{\rho}$  une $\qlb$-représentation irréductible  cuspidale entière
de $\G$.  
On suppose que $\rho=\r_\ell(\tilde{\rho})$ est une $\flb$-représentation 
irréductible.
\begin{enumerate}
\item 
Soient $a\<b$ des entiers. 
Alors la $\qlb$-représentation $\Z([a,b]_{\tilde{\rho}})$ est entière et~: 
\begin{equation*}
\r_\ell\left(\Z([a,b]_{\tilde{\rho}})\right)=\Z\left([a,b]_{\rho}\right).
\end{equation*}
\item 
Soit un multisegment
$\tilde\m=[a_1,b_1]_{\tilde\rho}+\dots+[a_r,b_r]_{\tilde\rho}$
et posons~:
\begin{equation*}
\m=[a_1,b_1]_{\rho}+\dots+[a_r,b_r]_{\rho}.
\end{equation*}
Alors $\Z(\tilde\m)$ est une représentation entière et $\Z(\m)$ est un facteur 
irréductible de $\r_\ell(\Z(\tilde\m))$.  
\end{enumerate}
\end{theo}

\begin{proof}
On commence par montrer (1).  
On note $\Delta=[a,b]_\rho$ et $\tilde\Delta=[a,b]_{\tilde\rho}$ et on pose $n=b-a+1$. 
On fixe un type simple maximal $\l=\k\otimes\s$ contenu 
dans $\rho$, et on forme le morphisme $\KMS_{n}$ comme au paragraphe 
\ref{PropZLS}. 
Soit $\tilde\k\otimes\tilde\s$ un type simple maximal 
contenu dans $\tilde\rho$ et relevant $\k\otimes\s$. 
Il corres\-pond à ce relèvement (\S\ref{DefKMS}) une $\b$-extension 
$\tilde\k_{n}$ et un foncteur~:
\begin{equation*}
\tilde\KM_{n}:\Rr_{\qlb}(\G_{mn})\to\Rr_{\qlb}({\GB}_{m'n}). 
\end{equation*}
Remarquons, dans un premier temps, que $\r_\ell\left(\Z(\Delta)\right)$ 
ne contient pas de représentation cuspidale.
En effet, la représentation $\KMS_{n}{(\Pi(\tilde\Delta)})$ 
contient, par la proposition \ref{CalculKMSI} et le paragraphe \ref{Clothier}, 
la représentation $\st(\tilde{\s},n)$ avec multiplicité $1$.  
Comme $\KMS_{n}{(\L(\tilde\Delta))}$ contient $\st(\tilde\s,n)$ 
(voir la proposition \ref{prpisegK}\eqref{prpisegK2}), la réduction 
$\r_\ell(\KMS_{n}{(\L(\tilde\Delta))})$ contient $\st(\s,n)$. 
Comme $ \Z(\tilde\Delta) \neq \L(\tilde\Delta)$ dès que $n\neq1$, 
la réduction $\r_\ell(\KMS_{n}{(\Z(\tilde\Delta))})$ ne contient pas de représentation 
cuspidale.  
On déduit que $\r_\ell(\Z(\tilde\Delta))$ ne 
contient pas de re\-pré\-sen\-tation cuspidale. 

Par récurrence, on peut supposer que pour tout sous-groupe parabolique propre 
$\P=\P_\a$, la réduction $\r_\ell(\rp_\a(\Z(\tilde\Delta))$ 
est nulle ou irréductible.  
Si $\r_\ell(\Z(\tilde\Delta))$ n'est pas irréductible, comme elle ne 
contient pas de représentation cuspidale, il existe un sous-groupe 
parabolique propre $\P=\P_\a$ tel que 
$\r_\ell(\rp_\a(\Z(\tilde\Delta)))$ soit de 
longueur au plus $2$~: contradiction. 

Montrons maintenant (2). 
Compte tenu de ce qui précède et du paragraphe \ref{inds}, on a l'égalité 
$\r_\ell(\I(\tilde\m)) =\sy{\I(\m)}$.  
Comme le foncteur de Jacquet commute aussi à la réduction
modulo $\ell$ (voir le paragraphe \ref{jacqs})
et que dans le cas fini la réduction modulo $\ell$ d'une 
représentation non dégénérée contient une représentation non dégénérée, 
le module de Jacquet~:
\begin{equation*}
\rp_{\mu_{\m}}(\r_\ell(\Z(\tilde\m)))
\end{equation*}
contient une représentation non dégénérée.
Par définition de $\Z( \m)$, on a le résultat annoncé. 
\end{proof}

\begin{theo}
\label{reductionsegment2}
Pour tout $m\>1$, l'homomorphisme de réduction~:
\begin{equation}
\label{Ackroyd}
\r_{\ell}:\Gg_{\qlb}(\G_m)^{{\rm en}}\to\Gg_{\flb}^{}(\G_m)
\end{equation}
est surjectif.
\end{theo}

\begin{proof}
Soit $\ss\in\Dive(\Ss_{\flb})$ un support supercuspidal, et soit 
$\Gg_{\flb}^{}(\G_m,\ss)$ le sous-groupe de $\Gg_{\flb}^{}(\G_m)$
engendré par $\Irr(\ss)$. 

\begin{lemm}
\label{vite!}
L'ensemble des représentations de la forme
$\sy{\I(\m)}$, avec $\m\in \MS^{\sc}(\ss)$, est une base de 
$\Gg_{\flb}^{}(\G_m,\ss)$. 
\end{lemm}

\begin{proof}
D'après le théorème \ref{bijj2thm}, l'ensemble des $\sy{\Z(\m)}$, 
$\m\in\MS^{\sc}(\ss)$, est une base de $\Gg_{\flb}^{}(\G_m,\ss)$.  
Il suffit donc de prouver que l'ensemble des 
$\sy{\I(\m)}$, $\m\in\MS^{\sc}(\ss)$,
qui a le même cardinal, engendre le $\ZZ$-module libre 
$\Gg_{\flb}^{}(\G_m,\ss)$.  

On peut supposer que $\supp(\ss)\in\Dive(\ZZ_\rho)$ avec $\rho$ une 
représentation supercuspidale. 
Supposons qu'il existe $\m'\in\MS^{\sc}(\ss)$ tel que 
$\Z(\m')$ n'appartienne pas au sous-$\ZZ$-module engendré par les 
$\sy{\I(\m)}$, $\m\in\MS^{\sc}(\ss)$, et choisissons $\m'$ tel que 
$\mu_{\m'}$ soit minimal. 
D'après la proposition \ref{ejemplofund}, on a une décomposition~: 
\begin{equation*}
\sy{\I(\m')}=\sy{\Z(\m')}+
\sum\limits_{\m} a_\m\sy{\Z(\m)}
\end{equation*}
où $\m$ décrit les multisegments dans $\MS^{\sc}(\ss)$
tels que $\mu_\m\lhd\mu_{\m'}$ 
(où $\lhd$ désigne la relation d'ordre stricte associée à $\trianglelefteq$) 
et où les $a_\m$ sont des entiers $\>0$.
Par minimalité de $\mu_{\m'}$, les $\sy{\Z(\m)}$ avec 
$\mu_\m\lhd\mu_{\m'}$ appartien\-nent au 
sous-groupe engendré par les 
$\sy{\I(\m)}$ avec $\m\in\MS^{\sc}(\ss)$,
donc $\sy{\Z(\m')}$ aussi, ce qui conduit à une contradiction. 
\end{proof}

Soit $\pi$ une $\flb$-représentation irréductible de 
$\G_m$ et soit $\ss$ son support supercuspidal.  
D'après le lemme \ref{vite!}, 
il existe des $a_\m\in \ZZ$, $\m\in \MS^{\sc}(\ss)$ tels que~: 
$$\sy{\pi}=\underset{\m\in \MS^{\sc}(\ss)}{\sum}a_\m\left[\I(\m)\right].$$
Par les théorèmes \ref{reductionsegment1}(1) 
et \ref{RappelLiftSupercusp}, 
pour tout $\m\in \MS^{\sc}(\ss)$, 
il existe un multisegment $\tilde\m$ tel que 
$\r_\ell(\left[\I(\tilde\m)\right])=\left[\I(\m)\right]$. 
On en déduit que~: 
$$\sy{\pi}=\r_\ell\Big(\sum\limits_{\m} a_\m\left[\I(\tilde\m)\right]\Big),$$
où la somme porte sur les $\m\in \MS^{\sc}(\ss)$, 
et donc \eqref{Ackroyd} est surjectif.
\end{proof}

\begin{rema}
\begin{enumerate}
\item 
En général, la réduction modulo $\ell$ 
d'une représentation de la forme $\L(\tilde\Delta)$ n'est pas irréductible.  
Par exemple, si $\tilde\rho$ est la représentation triviale de $\F^\times$, 
si $\tilde\Delta=\left[0,1\right]_{\tilde{\rho}}$ et si $q$ est d'ordre $2$
modulo $\ell$, la représentation $\L(\tilde\Delta)$ contient, d'après
\cite{Vig3}, un caractère et une représentation cuspidale non supercuspidale.
\item 
Si $n(\Delta)<\ee(\rho)$,
alors $\r_\ell(\L(\tilde\Delta))=\L(\Delta)$ (voir \cite{MS}).
\item 
En général, on ne peut pas relever une représentation de la forme $\L(\Delta)$. 
Par exemple, si $\rho$ est la représentation triviale de $\F^\times$, 
si $\Delta=\left[0,2\right]_\rho$ et si $q$ est d'ordre $3$ modulo $\ell$, 
la représentation $\L(\Delta)$ ne peut pas être relevée.  
\item 
Si $\r_\ell(\tilde\rho)$ n'est pas irréductible, alors 
$\r_\ell(\Z(\tilde\Delta))$ ne l'est pas non plus. 
Par exemple, si $\tilde\rho$ est une $\qlb$-représentation 
cuspidale entière telle que $\r_\ell(\tilde\rho)=\sy{\rho}+\sy{\rho'}$ 
avec $\rho$ et $\rho'$ irréductibles non isomorphes, on a~:
\begin{equation*}
\r_\ell(\Z(\left[0,1\right]_{\tilde\rho}))=
\Z(\left[0,1\right]_{\rho}) + \Z( \left[0,1\right]_{\rho'}) 
+ \sy{\rho \times \rho'}.
\end{equation*}
\item 
Si $\rho$ est une $\flb$-représentation cuspidale telle qu'il existe 
une $\qlb$-représentation cuspidale irréductible $\tilde\rho$ qui relève 
$\rho$, alors $\Z(\left[a,b\right]_{\rho})$ peut être aussi caractérisé 
comme la réduction modulo $\ell$ 
de $\Z(\left[a,b\right]_{\tilde\rho})$.  
Cette définition ne dépend pas du choix de $\tilde\rho$
(voir \cite[Proposition 2.2.3]{Dat5}). 
\end{enumerate}
\end{rema}

\providecommand{\bysame}{\leavevmode ---\ }
\providecommand{\og}{``}
\providecommand{\fg}{''}
\providecommand{\smfandname}{\&}
\providecommand{\smfedsname}{\'eds.}
\providecommand{\smfedname}{\'ed.}
\providecommand{\smfmastersthesisname}{M\'emoire}
\providecommand{\smfphdthesisname}{Th\`ese}

\end{document}